\documentclass{amsart}

\usepackage%% [margins=1in]%%
           [top=0.8in, bottom=0.8in, left=0.8in, right=0.8in]
           {geometry}

\usepackage{mathrsfs}
\usepackage{amsthm}
\usepackage{thmtools}
\usepackage{amsmath}
\usepackage[shortlabels]{enumitem}

\usepackage{amssymb,amsxtra,amsfonts}
\usepackage{graphicx}
\usepackage{tikz-cd}

\usepackage[normalem]{ulem}

\usepackage[colorlinks=true,linkcolor=blue,urlcolor=blue, citecolor=red, pagebackref=True]{hyperref}

\usepackage[all]{xy}
\ifpdf

\makeatletter
\@namedef{subjclassname@2020}{\textup{2020} Mathematics Subject Classification}
\makeatother

\numberwithin{equation}{section}
\numberwithin{table}{section}

\declaretheorem[style=plain,parent=section]{theorem}
\declaretheorem[style=plain,sibling=theorem]{corollary}
\declaretheorem[style=plain,sibling=theorem]{lemma}

\declaretheorem[style=plain,sibling=theorem]{proposition}

\declaretheorem[style=definition,sibling=theorem]{definition}

\declaretheorem[style=definition, qed=\hfill $\diamond$, sibling=definition]{example}
\declaretheorem[style=remark,sibling=theorem]{remark}

\theoremstyle{definition}
\newtheorem*{theorem*}{Theorem}
\newtheorem*{exampleth}{Proof}

\newtheoremstyle{named}{}{}{\itshape}{}{\bfseries}{.}{.5em}{\thmnote{#3 }#1}
\theoremstyle{named}

\DeclareMathOperator{\Hom}{Hom}
\DeclareMathOperator{\Spec}{Spec}

\newcommand{\topo}{\mathrm{top}}

\newcommand{\gp}{\mathrm{gp}}

\newcommand{\rob}{\mathrm{rob}}

\newcommand{\tv}{\mathcal{X}}

\newcommand{\R}{\mathbb{R}}

\newcommand{\CC}{\mathbb{C}}
\newcommand{\Z}{\mathbb{Z}}
\newcommand{\N}{\mathbb{N}}
\newcommand{\D}{\mathbb{D}}
\newcommand{\bS}{\mathbb{S}}
\newcommand{\bP}{\mathbb{P}}

\newcommand{\cF}{\mathcal{F}}

\newcommand{\cM}{\mathcal{M}}
\newcommand{\cN}{\mathcal{N}}
\newcommand{\cO}{\mathcal{O}}

\newcommand{\cS}{\mathcal{S}}

\begin{document}
\title[An introduction to real oriented blowups]{An introduction to real oriented blowups \\ in toric, toroidal
   and logarithmic geometries}
\author[P. Popescu-Pampu]{Patrick Popescu-Pampu}

\date{14 July 2025}

{\bf To appear in the book {\em Singularity Theory from Modern Perspectives} edited by
Javier Fern\'andez de Bobadilla and Anne Pichon, {\em Panoramas et synth\`eses}, Soci\'et\'e Math\'ematique de France.}

\bigskip

\begin{abstract}
        This text is an introduction to the applications of {\em rounding of complex log spaces}
        (also known as  {\em Kato-Nakayama} or {\em Betti realization}) to singularity theory.
        Log spaces in the sense of Fontaine and Illusie were first described in print by Kato,
        in a 1988 paper. Rounding of complex log spaces was introduced in a 1999 paper by
        Kato and Nakayama and is a functorial generalization of A'Campo's 1975 notion of
        a {\em real oriented blowup}. It allows to cut canonically any complex toroidal
        variety $X$ along its toroidal boundary $\partial X$,
        producing a topological manifold-with-boundary, whose boundary is a canonical
        representative of the boundary of any tubular neighborhood of $\partial X$ in $X$.
        In singularity theory, roundings may be used to get canonical representatives
        of links of isolated complex analytic singularities and of Milnor fibers
        of smoothings of complex singularities, once
        toroidal resolutions of the singularity or of the smoothing are chosen.
        The text starts with introductions to not necessarily normal toric varieties,
        it passes then to toroidal varieties and to their real oriented blowups.
        It continues with introductions to
        log spaces and to rounding of complex log spaces. It concludes with an important theorem of
        Nakayama and Ogus about the local triviality of the rounding of special types of log morphisms.
        The notions of {\em affine toric variety}, {\em real oriented blowup}, {\em log structure}
        and {\em rounding} are introduced by means of the classical
        {\em passage to polar coordinates}.
 \end{abstract}

\maketitle
\vspace{-10mm}
\tableofcontents

\section{\bf Introduction}   \label{sec:introd}

\medskip
\subsection{The problem of non-uniqueness of tubular neighborhoods}   $\  $  \label{ssec:nutub}
\medskip

Riemann introduced the surfaces which bear his name as ramified covers of portions of the {\em Riemann sphere} $\CC \cup \{\infty\}$ associated to multivalued functions. In order to understand the topological structure of those covers, he performed {\em cross-cuts} along curves which go from one boundary point to another such point, as shown by the following quote of  \cite[Section 6]{R 51}:

   \begin{quote}
        {\em The study of the connectivity of a surface is based on its decomposition via transverse cuts,
        that is lines which cut through the interior from one boundary point simply (no point
        occurring multiply) to another boundary point.}
  \end{quote}
He also explained what to do when the surface under scrutiny had no boundary, as was the case for the Riemann surface associated to an algebraic function (see \cite[Section 2]{R 57}):

     \begin{quote}
        {\em To apply this treatment to a surface which has no boundary, in other words a closed surface,
        it must first be transformed into one with a boundary by excluding an arbitrary point, so that
        the first cross-cut is a closed curve which begins and ends in this point.}
     \end{quote}
This sentence leads to imagine that it should be possible to cut a surface at a point
without simply removing it, but by replacing it instead by a boundary curve.

 In the XX-th century, the intuitive operation of cutting was formalized and generalized to
 submanifolds of higher-dimensional manifolds, initially as a basis for {\em surgery operations}
 in differential topology. The first such operation was introduced by Milnor \cite{Mi 61}, following
 an idea of Thom,  under the name of {\em surgery} and independently
 by Wallace \cite{W 60} under the name of
 {\em spherical modification} (see Kosinski's historical explanations in \cite[Chapter VII.8]{K 93}).
It consists in removing a sphere with trivialized normal bundle and replacing it with another such sphere
of complementary dimension minus one (see \cite[Definition 5.2.1]{GS 99}).
The related operation of {\em Dehn surgery} (see \cite[Chapter 5.3]{GS 99}) replaces a circle inside a
three-dimensional manifold by another such circle in a way which may be more general
than in Milnor's and Wallace's surgery. In general, both kinds of surgery operations
modify the global topology of the manifold.
This may be difficult to understand at first sight: for instance, how is it possible to replace
a circle by itself and change nevertheless the global topology?

In order to explain this phenomenon, the notion of {\em regular} or  {\em tubular} {\em neighborhood}
of the submanifold proves itself crucial:
one does not simply remove the submanifold, but the interior of a compact tubular neighborhood
of it in the ambient manifold, creating a {\em boundary} which has the structure of sphere bundle over the
submanifold. One understands then how to replace a circle by a circle in a different way inside a three-dimensional manifold: one removes a tubular neighborhood of the circle, which is diffeomorphic to a solid torus
$\bS^1 \times \D^2$, creating a boundary component diffeomorphic to a two-dimensional torus
$\bS^1 \times \bS^1$. Then, one glues back the solid torus by a diffeomorphism of the boundaries. The
result depends only on the isotopy class of this diffeomorphism. There are many ways to perform this gluing  because up to isotopy, there are many diffeomorphisms of two tori.

Tubular neighborhoods became the basic ingredients of {\em surgery operations}
along connected submanifolds $S$ of a connected manifold $M$, which may be defined
as the results of procedures of the following kind:
  \begin{itemize}
     \item choose a compact tubular neighborhood $N(S)$ of $S$ in $M$;
     \item cut $M$ along the boundary $\partial_{\topo} N(S)$ of $N(S)$;
     \item remove the resulting connected component which is
         diffeomorphic to the tubular neighborhood $N(S)$;
     \item glue to the remaining manifold-with-boundary
        another manifold-with-boundary along a boundary component diffeomorphic to
        $\partial_{\topo} N(S)$.
   \end{itemize}
The main drawback of such procedures is that they involve non-canonical choices of tubular
neighborhoods, which may create problems in some
contexts. For instance, Grothendieck described such problems as follows in
his 1984 text \cite[Section 5]{G 83}:

\begin{quote}
     {\em  This naive vision immediately encounters various difficulties. The first is the
     somewhat vague nature of the very notion of tubular neighbourhood, which acquires
     a tolerably precise meaning only in the presence of structures which are much more
     rigid than the mere topological structure, such as ``piecewise linear'' or Riemannian
     (or more generally, space with a distance function) structure; the trouble here is that in
     the examples which naturally come to mind, one does not have such structures at
     one's disposal -- at best an equivalence class of such structures, which makes it possible
     to rigidify the situation somewhat. If on the other hand one assumes that one might find
     an expedient in order to produce a tubular neighbourhood having the desired properties,
     which moreover would be unique modulo an automorphism (say a topological one)
     of the situation -- an automorphism which moreover respects the fibered structure provided
     by the glueing map, there still remains the difficulty arising from the lack of canonicity
     of the choices involved, as the said automorphism is obviously not unique, whatever
     may be done in order to ``normalise'' it.}
\end{quote}

\medskip
\subsection{Bypassing the problem of non-uniqueness of tubular neighborhoods}   $\  $
   \label{ssec:bypasspbm}
\medskip

It turns out that there is a way to bypass the problem of non-canonicity of tubular neighborhoods
whenever one wants to cut a complex analytic manifold $M$
along a smooth complex hypersurface or, more generally, a divisor $D$ with normal crossings and
with smooth irreducible components (one says, then, that $D$ has {\em simple normal crossings}).
Indeed, in this case A'Campo introduced in his 1975 paper \cite{A 75} the operation of a {\em real oriented
blowup of $M$ along $D$}, which generalizes the {\em passage to polar coordinates}.

The topological effect of passing to polar coordinates is to transform the complex plane $\CC$
into an annulus by replacing the origin of $\CC$
by a circle whose points represent the {\em oriented real} lines
of $\CC$ passing through the origin. If one had replaced instead the origin by a circle
representing the {\em unoriented} real lines passing through it, then one would perform the
usual algebro-geometric blowup of $\CC$ seen as a real plane,
which explains why the passage to polar coordinates is also called a {\em real oriented blowup}.
Note that the usual blowup of $\CC$ at the origin produces a M\"obius band instead of an annulus,
the circle which replaces the origin of $\CC$ being a core circle of this M\"obius band (see \autoref{fig:twobo}).

 \begin{figure}[!h]
    \centering
        \includegraphics[width=8cm]{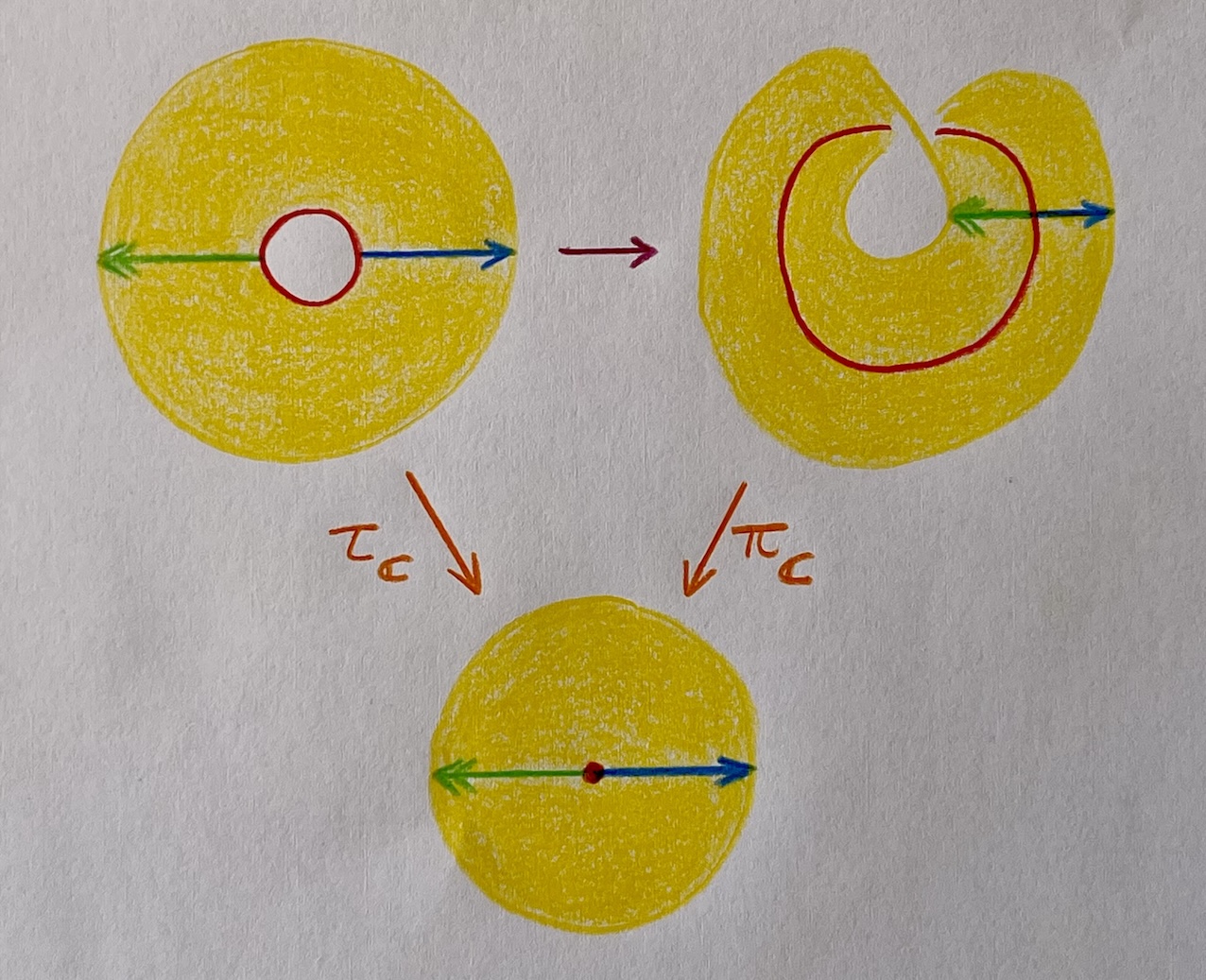}
     \caption{The real oriented blowup $\tau_{\CC}$ of $\CC$ at the origin
         factors through its real blowup $\pi_{\CC}$}
   \label{fig:twobo}
    \end{figure}

As in this simple case, A'Campo's real oriented blowup replaces canonically
the divisor $D$, which may be seen as an {\em algebro-geometric boundary} of $M$,
by a {\em topological boundary} of the complement $M \setminus D$
of the divisor, and the resulting manifold-with-boundary
(and {\em corners} when the divisor has singular points)
is homeomorphic to the complement of the interior
of a tubular neighborhood of the divisor $D$. That is, {\em no canonical tubular neighborhood
of the divisor is introduced, one defines instead a canonical cutting operation along the divisor},
which {\em replaces the algebro-geometric boundary $D$ by a topological boundary}
(see \autoref{fig:robncd}).
In order to distinguish these two kinds of {\em boundaries}, we denote in this paper
by $\partial X$ the algebro-geometric boundary of a space $X$
and by $\partial_{\topo} Y$ the topological boundary of a topological manifold-with-boundary $Y$.

\begin{figure}[!h]
    \centering
        \includegraphics[width=9cm]{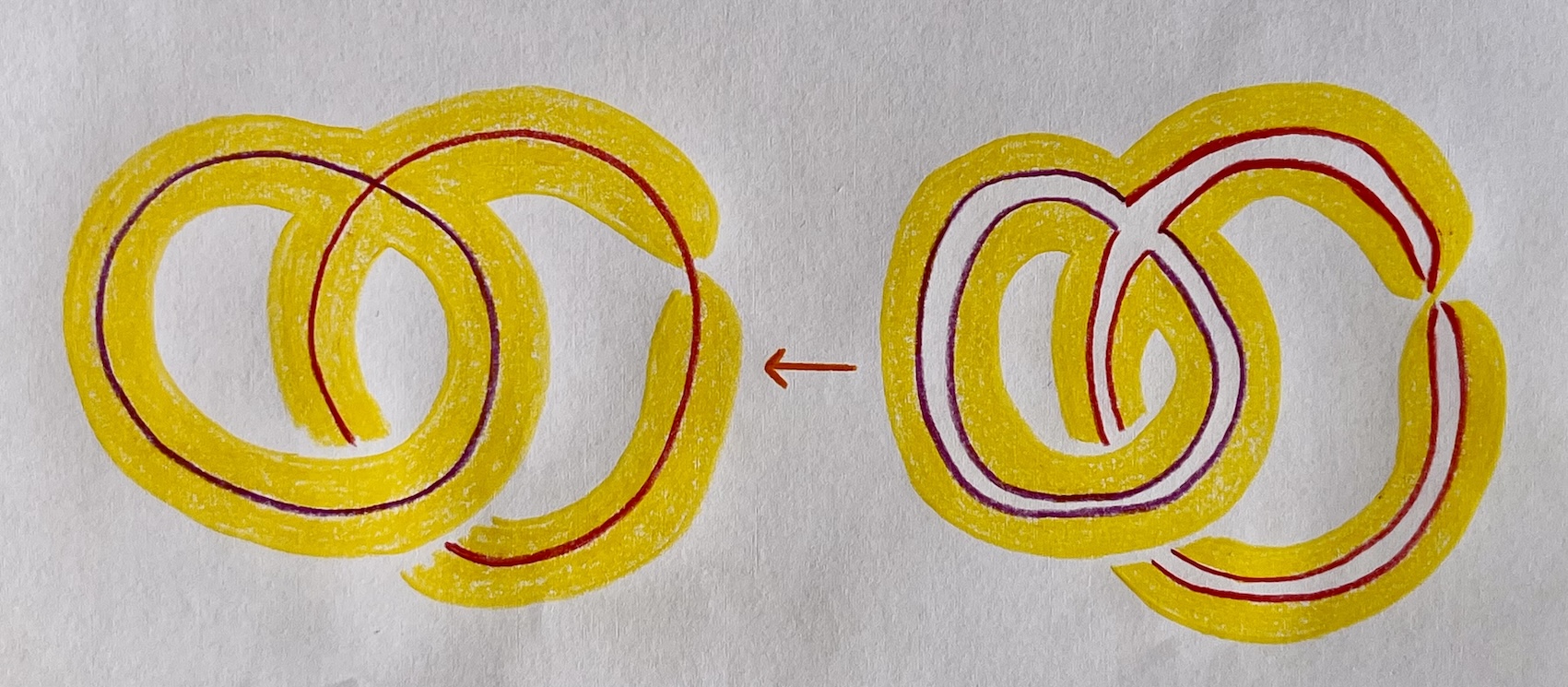}
     \caption{The real oriented blowup of a smooth surface along a divisor with simple normal crossings}
   \label{fig:robncd}
    \end{figure}

{\em When does one need to cut a complex manifold along a  divisor with simple normal crossings?}
A natural such situation arises in {\em singularity theory} when one wants to understand
the structure of the {\em link} or {\em boundary} $\partial(X,x)$ of an {\em irreducible isolated singularity}
$(X,x)$, that is, of an analytically irreducible germ of complex variety admitting a representative
with smooth complement $X \setminus \{x\}$.
By definition, $\partial(X,x)$ is the boundary of a preferred compact neighborhood of $x$ in $X$,
which we will call a {\em Milnor representative} of the singularity,
obtained by intersecting a representative of the germ $(X,x)$ embedded in an affine space $\CC^n$
with an associated compact Milnor ball.
After resolving the singularity $x$ in such a way that it gets replaced by an
exceptional divisor $E$ with simple normal crossings
 in a manifold $\tilde{X}$ (which is always possible by Hironaka's resolution theorem),
the link $\partial(X,x)$ appears as the boundary of a tubular neighborhood of $E$ in $\tilde{X}$. In this
context, in general one cannot choose $E$ to be smooth. This may already be seen in the
two-dimensional case, in which this method of studying links of singularities was first introduced
by Mumford in \cite{M 61}. Indeed, in this dimension,  one has a minimal resolution of singularities
and if its exceptional divisor is not smooth, then no resolution with smooth exceptional divisor may exist.

Mumford used the previous viewpoint on links of isolated complex surface singularities
to describe them as {\em plumbed manifolds}, constructed from elementary pieces by following
a plan represented by the {\em dual graph} of the exceptional divisor $E$.
Its vertices are endowed with two kinds of {\em decorations}:
the {\em genera} and the {\em self-intersection numbers} of the corresponding irreducible
components of the
exceptional divisor in the ambient smooth complex surface $\tilde{X}$ (see \cite{PP 22}  for details about the historical evolution of the relations between graph theory and singularity theory). One may
describe more generally in this way the boundary of a tubular neighborhood of a compact reduced divisor
$D$ with simple normal crossings in any smooth complex surface $M$. In other words, one
has a {\em plumbing presentation} of the boundary of the real oriented blowup of the manifold
$M$ along the divisor $D$.

If one wants to describe similarly the boundary of the real oriented blowup of a higher-dimensional
manifold $M$ along a simple normal crossings divisor $D$,  it is necessary
to find an analog of the dual graph and of its decorations. The dual graph has such an analog:
it is the {\em dual complex} of the exceptional divisor $E$, whose simplices correspond to the
irreducible components of all the non-empty intersections of components of $E$
(see Danilov \cite{D 75}, Stepanov \cite{S 06} and de Fernex, Koll\'ar and Xu \cite{FKX 17}).
In the surface case, the genus associated to a vertex encodes the abstract topological
type of the corresponding irreducible component,
and the self-intersection number encodes the topological type of its embedding
in a tubular neighborhood. In higher dimensions, these topological types cannot
be any longer encoded by single numbers, the best
one can hope for is to encode them by homology and homotopy groups and suitable
characteristic classes. That is, no simple generalization of the decorated dual graph of a compact
divisor $D$ with simple normal crossings is available in higher dimensions,
which would allow us to describe the structure of the topological boundary of the real oriented
blowup of $M$ along $D$.

This problem can be bypassed by defining the operation of  real oriented blowup differently,
using Fontaine and Illusie's {\em logarithmic structures} described for the first time in print by Kato
in \cite{K 88}.  Succinctly, $D$ defines a canonical {\em logarithmic structure} on $M$,
whose {\em restriction}
to $D$ gives birth, by a canonical operation of {\em rounding}, to a manifold with corners which is
homeomorphic to the topological boundary of the real oriented
blowup of $M$ along $D$. That is, {\em in order to encode this topological boundary,
one does not need to define a decorated dual complex
of the embedding of $D$ in $M$, it suffices to work with $D$ endowed with a
suitable logarithmic structure}.

\medskip
\subsection{Aspects of the logarithmic viewpoint on real oriented blowups} $\:$  \label{ssec:advtoric}
\medskip

When encountered for the first time, the definition of a {\em logarithmic structure} or
a {\em log structure} on a ringed space $W$ is not very enlightening: {\em it is a
sheaf of monoids $\cM_W$ endowed with a morphism $\alpha_W$ of sheaves of monoids
to the structure sheaf $\cO_W$ of $W$, seen as a sheaf of multiplicative monoids,
such that $\alpha_W$ induces an isomorphism of the corresponding subsheaves of invertible
elements of $\cM_W$ and $\cO_W$.} The objective of this text is to lead the readers
naturally to this definition starting from the classical passage to polar coordinates
(see \autoref{ssec:polcoordintr})
and to give a feeling for some of its properties related to real oriented blowups. Namely,
we hope to make the readers learn that:
  \begin{enumerate}
     \item A {\em log space} is a complex space endowed with a log structure (see \autoref{def:logspace}).
          There exists a notion of {\em morphism} between log spaces, which allows to
          define a {\em log category} (see \autoref{def:morlogspaces}).
     \item If $D$ is a reduced divisor in a complex variety $W$, then there
         is an induced {\em divisorial log structure} on $W$, whose sheaf $\cM_W$ is formed by the
         holomorphic functions which do not vanish outside $D$ and whose morphism
         $\alpha_W$ is the canonical embedding of $\cM_W$ in the structure sheaf $\cO_W$
         of $W$ (see \autoref{def:divlogstruct}).
     \item   \label{incldiv}
        If $E \hookrightarrow V, D \hookrightarrow W$ are  reduced divisors in the
        complex varieties $V, W$
         and if $f : V \to W$ is a holomorphic morphism such that $f^{-1}(D) \subseteq E$, then
        there is an associated morphism $f^{\dagger} : V^{\dagger} \to W^{\dagger}$
        of log spaces from $V$ endowed with the
        divisorial log structure induced by $E$ to $W$ endowed with the
        divisorial log structure induced by $D$ (see \autoref{def:logenhancement}).
     \item    \label{strictmorph}
         Log structures may be {\em pulled back} by morphisms of complex spaces
         (see \autoref{def:pb}). The morphisms of log spaces obtained in this way are called {\em strict}
         (see \autoref{def:strict}).
     \item  There exists a functor of {\em rounding} from the category of complex log
          spaces to that of topological spaces (see \autoref{def:rounding} and \autoref{thm:torsorfibre}).
          This functor
          was introduced by Kato and Nakayama in \cite{KN 99} and given this
          name by Nakayama and Ogus in \cite{NO 10}. We denote by
          $\phi^{\odot} : V^{\odot} \to W^{\odot}$ the rounding of the morphism $\phi : V \to W$
          of the complex log category.
     \item If $W$ is a complex log space with underlying topological space $|W|$, then
          there exists a canonical continuous map $\tau_W : W^{\odot} \to |W|$, called
         the {\em rounding morphism} of the log space $W$. {\em Whenever $W$ is a complex
        manifold endowed with the log structure induced by a simple normal
        crossings divisor,  $\tau_W$ is isomorphic to the real oriented blowup of
        the manifold along the divisor} (see \autoref{prop:topmantoroid}).
      \item If $\phi : V \to W$ is a morphism of logarithmic spaces
      with associated continuous map $|\phi| : |V| \to |W|$, then the associated diagram
                        \[
                            \xymatrix{
                                V^{\odot}     \ar[r]^{\phi^{\odot}}   \ar[d]_{\tau_V} &   W^{\odot} \ar[d]^{\tau_W} \\
                                |V|    \ar[r]_{|\phi|}                                  &   |W|.}
                          \]
                    is commutative (see \autoref{thm:torsorfibre} \eqref{functbr}).
         \item If the morphism $\phi : V \to W$ is {\em strict} in the sense of point
              \eqref{strictmorph}, the previous commutative
             diagram is moreover {\em cartesian}, that is, it is a pullback diagram
             (see \autoref{thm:torsorfibre} \eqref{cartdiag}). In particular,
             if $\varphi : D \hookrightarrow W$ is the embedding of a simple normal
             crossings divisor $D$ in a manifold $W$,
             then the rounding of the restriction to $D$ of the log structure induced by $D$ on $W$
             gets identified by the map $\phi^{\odot} : D^{\odot} \hookrightarrow W^{\odot}$
             with the boundary of the real oriented blowup of $W$ along $D$. Here
             $\phi: D^{\dagger} \to W^{\dagger}$ is the morphism between the previous
             log spaces induced by $\varphi$.
   \end{enumerate}

   \medskip
   \subsection{Consequences for singularity theory} \label{ssec:advsing}  $\ $
   \medskip

  An advantage of the logarithmic viewpoint on real oriented blowups is that one has now
  a {\em functor of rounding} between suitable categories and that the rounding of a complex
  log space is endowed with a {\em canonical} projection onto the underlying topological space
  of the log space.  This fact has three main consequences for singularity theory:
    \begin{itemize}
          \item Assume that $(X, x)$ is an isolated singularity of complex analytic space and
              that $\pi : (\tilde{X}, E) \to (X,x)$ is a resolution of it whose exceptional divisor $E$
              has simple normal crossings. Consider the
              restriction to $E$ of the divisorial log structure induced by $E$ on $\tilde{X}$. Then {\em the
              rounding of this log structure on $E$ is a representative of the link $\partial(X,x)$ of the
              singularity $(X,x)$}. This representative is endowed with a canonical projection to $E$,
              unlike what happens with the boundaries of the tubular neighborhoods
              of $E$ inside $\tilde{X}$.
          \item  Assume that $(X, x)$ is an isolated singularity of a complex analytic space and
              that $\pi_i : (\tilde{X}_i, E_i) \to (X,x)$ for $i \in \{1, 2\}$ are two resolutions of it
              whose exceptional divisors $E_i$ have simple normal crossings, such that
              $\pi_1$ factors through $\pi_2$, that is, such that there exists a bimeromorphic morphism
              $\psi :  (\tilde{X}_1, E_1) \to  (\tilde{X}_2, E_2)$ with $\pi_1 = \psi \circ \pi_2 $.
             By point \eqref{incldiv} above, the morphism $\psi$ induces a log morphism
             from $\tilde{X}_1$ endowed with the divisorial log structure induced by $E_1$
             to $\tilde{X}_2$ endowed with the divisorial log structure induced by $E_2$.
             Restricting this morphism to the exceptional divisors $E_i$ one gets a log morphism
             whose rounding induces a morphism between the two associated respresentatives
             of the link $\partial(X,x)$ of the singularity $(X,x)$. Thus, {\em the inverse system of
             such resolutions of $(X,x)$ provides us with an inverse system of
             representatives of $\partial(X,x)$}.
          \item Assume that $f : (X, x) \to (\CC, 0)$ is a germ of a holomorphic function defined on a germ
             of an irreducible complex variety such that the fiber $f^{-1}(0)$ contains the
             singular locus of $(X,x)$. Therefore, the other fibers of $f$ are smooth
             in a neighborhood of $x$.
             For this reason, one says that $f$ is a {\em smoothing of the singularity $(Z(f),x)$}.
             Then, $f$ admits an associated {\em Milnor} or {\em Milnor-L\^e fibration over the circle}
             (see \cite{M 68}, \cite{L 77}, \cite[Theorem 5.1]{S 19} and \cite[Theorem 6.7.1]{CS 21}).
             Choose a Milnor representative of $X$ and a resolution
             $\pi : (\tilde{X}, E) \to (X,x)$ such that the set-theoretic fiber $Z(f \circ \pi) := (f \circ \pi)^{-1}(0)$
             of the composition $f \circ \pi$  is a divisor with simple normal
             crossings. By point \eqref{incldiv} above, the morphism $f \circ \pi :(\tilde{X}, E) \to  (\CC, 0)$
             yields a log morphism from $\tilde{X}$ endowed with the divisorial log structure
             induced by $Z(f \circ \pi)$
             to $\CC$ endowed with the divisorial log structure induced by the origin $0$, seen as a
             reduced divisor on $\CC$. {\em Restricting this log morphism to $Z(f \circ \pi)$ and $0$
             respectively, we get a log morphism whose rounding is a
             representative of the Milnor-L\^e fibration of $f$}. This fact is a consequence of
             a special case of an important local triviality theorem of Nakayama and Ogus
             (see \autoref{cor:milntubelog}).
    \end{itemize}

In fact, A'Campo had introduced real oriented blowups in \cite{A 75}
precisely in order to study Milnor fibrations of smoothings (and their geometric monodromies).
A'Campo's real oriented blowups or the more general roundings of complex log structures
were again used in the study of Milnor fibrations by Parusi\'nski \cite{P 98},
Campesato, Fichou and Parusi\'nski \cite{CFP 21}, Fern\'andez de Bobadilla and Pe\l ka \cite{FP 24},
Portilla Cuadrado and Sigurdsson \cite{PS 23}  and Cueto, the present author and Stepanov
\cite{CPPS 23}.

A specificity of this last reference compared to the other ones
is that it considers not only divisorial log structures induced
by simple normal crossings divisors in manifolds, but more generally by the boundaries
of {\em toroidal varieties}. Those are by definition complex analytic varieties endowed
with an algebro-geometric boundary locally analytically isomorphic to open sets in
toric varieties endowed with their toric boundary divisor. It turns out that the logarithmic
viewpoint on real oriented blowup makes as easy to perform real oriented blowups
of toroidal varieties along their boundaries as of complex manifolds along simple normal crossings
divisors. This gives more flexibility to the tool of real oriented blowup, allowing to study
singularities using modifications which, endowed with their exceptional loci, are not
necessarily smooth, but are toroidal. In addition to \cite{CPPS 23}, where such modifications
are obtained by using fans which subdivide the {\em local tropicalizations} of the singularities
under scrutiny in the sense of \cite{PPS 13} and \cite{PPS 25},
one may look also at Bultot and Nicaise \cite{BN 20}, where they are
used in the study of the so-called {\em monodromy conjecture} concerning Denef and Loeser's
motivic Zeta function.

{\em For this reason, in this text we chose to give not only an introduction to real oriented blowups
for divisors  with simple normal crossings in complex manifolds, but also to toric, toroidal
and logarithmic geometries.} Another reason of this choice is that we want to emphasize the
following tight relations between the three kinds of geometry:

\begin{itemize}
   \item Every complex toric variety is canonically a toroidal variety and  toric morphisms
      are always toroidal morphisms. That is, the category of complex toric varieties is a
      subcategory of that of toroidal varieties.
   \item  Every toroidal variety has a canonical divisorial log structure and toroidal morphisms
      induce canonically log morphisms relative to these structures.
      That is, the category of toroidal varieties is a subcategory of that of complex log spaces.
    \item When one studies a toroidal variety $X$, one is also interested in the structure of its
        algebro-geometric boundary $\partial X$. This boundary is  not a toroidal subvariety
        of $X$, but it is a log subspace of the associated complex log space of $X$.
   \item The rounding of a toroidal log variety is canonically isomorphic to the real oriented
      blowup of the corresponding toroidal variety. That is, the rounding operation is a
      generalization of a real oriented blowup.
    \item A subvariety of a toric variety which intersects transversally its orbits inherits a
      canonical toroidal structure from the ambient variety. This applies to modifications of
      Newton non-degenerate singularities obtained by subdividing
      their local tropicalizations (see \cite[Sections 3.3 and 4.1]{CPPS 23}).
      This illustrates the fact that, in singularity theory, toric and toroidal geometries are intimately related.
\end{itemize}

\medskip
\subsection{The structure of the text} $\:$  \label{ssec:structext}
\medskip

This text is subdivided into sections and subsections.
Each section begins by a general description of its content and of the literature related to it. In turn, the first paragraph of each subsection gives a brief description of its content. Whenever a notation is introduced, it is presented inside a \boxed{\textrm{box}}. For instance, the continuous map
       \[ \tau_{\CC} : \R_{\geq 0}  \times \bS^1  \to  \CC \]
of {\em passage to polar coordinates}
is emphasized as $\boxed{\tau_{\CC}}$ in formula \eqref{eq:changepolarcomplexmap}.

This map will be a recurrent example in this paper. Successive reinterpretations of it will
serve to introduce new concepts and to highlight different aspects, namely:
\begin{itemize}
   \item  the need to work with {\em monoids} and their morphisms (see \autoref{rem:twoaspectsrobu});
   \item  a definition of {\em real oriented blowup} of $\CC^n$ (see \autoref{def:robuCn})
        and more generally
       of any complex affine toric variety (see formula  \eqref{eq:robarbtoricmon});
   \item the definition of the {\em divisorial log structure} on the complex plane $\CC$ induced by
          its origin (see \autoref{def:robRiemsurf}) and, by extension, of {\em log structures};
    \item the definition of {\em prelog structures} (see \autoref{ssec:prelogassoc}).
\end{itemize}

Three aspects of the map $\tau_{\CC}$ will appear in this text:
    \begin{itemize}
         \item  it replaces an algebro-geometric boundary
              by a topological boundary (see \autoref{rem:twoaspectsrobu});
         \item  it is a morphism of monoids (see \autoref{rem:twoaspectsrobu});
         \item  it defines a complex log structure on a point (see \autoref{rem:newvpolcoord}).
    \end{itemize}

\autoref{sec:introtoric} introduces the reader to toric and toroidal geometries.
\autoref{sec:torictoroidrob} presents the real oriented blowup operations in toric and
toroidal geometries and compares them to the original definition by A'Campo of
a real oriented blowup of a complex manifold along a divisor with simple normal crossings.
Finally, \autoref{sec:poltoround} introduces the reader to log geometry and to Kato and
Nakayama's version of real oriented blowup, namely, the {\em rounding} of complex log spaces.
Its final \autoref{ssec:approundsing}  gives a synoptic view of applications of rounding to singularity theory.

\medskip
\subsection{Acknowledgments}   $\ $ \label{ssec:acknow}
\medskip

This work was supported by the ANR SINTROP (ANR-22-CE40-0014) and the Labex CEMPI (ANR-11-LABX-0007-01).

The author is very grateful to Mar\'{\i}a Ang\'elica Cueto and Dmitry Stepanov
with whom he collaborates on a long-term project of study of Milnor fibers of smoothings
using tropical and logarithmic techniques.
It is during this collaboration that he learned  properties of complex log varieties and
of their roundings which are useful in singularity theory. He thanks warmly Mar\'{\i}a Angelica Cueto for her careful reading of a first version of this text. He is grateful to Hulya Arg\"uz, Luc Illusie and
Arthur Ogus for their remarks.

He is also grateful to the scientific committees  of the {\em Workshop on Singularities: topology, valuations and semigroups} (Universidad Complutense de Madrid, 28--31/10/2019),
of the research school {\em Milnor fibrations, degenerations and deformations from modern perspectives} (CIRM, Marseille, 06--09/09/2021, during Javier Fern\'andez de Bobadilla's Jean Morlet Chair stay), of the 2022 annual meeting of the {\em GDR singularities} (Aussois, 04--08/07/2022) and of the research school
{\em Logarithmic and non-archimedean methods in singularity theory} (CIRM, Marseille, 27--31/01/2025)
for having invited him to give mini-courses on themes related with log-geometry: {\em A tropical and logarithmic study of Milnor fibers}, {\em A proof of Neumann-Wahl Milnor fiber conjecture via logarithmic geometry}, {\em Introduction to logarithmic geometry} and {\em Tropical and logarithmic techniques for the study of Milnor fibers} respectively (see \cite{PP 25}). This text developed partly from his notes for those mini-courses. Many thanks go to Anne Pichon and Javier Fern\'andez de Bobadilla for the invitation to contribute to this volume.

\medskip
\section{\bf From polar coordinates to toric and toroidal geometries}  \label{sec:introtoric}

{\em Toric geometry} is the branch of algebraic geometry which studies {\em toric varieties}
and their toric morphisms. By definition, a complex algebraic variety is called {\em toric}
if it is the closure inside an ambient variety of a single orbit under the action of an {\em algebraic
torus} $((\CC^*)^n, \cdot)$. The first textbook \cite{KKMS 73} of toric geometry was published by Kempf, Knudson, Mumford and Saint-Donat in 1973. Let us quote from its introduction:

\begin{quote}
     {\em  When teaching algebraic geometry and illustrating simple singularities, varieties,
       and morphisms, one almost inevitably tends to choose examples of a ``monomial''  type [...].
       Moreover, even when a variety as a whole is quite general, either
       its singularities, or a certain blow-up may well be defined in local analytic coordinates
       by monomials.  When this happens, problems of algebraic geometry can sometimes
       be translated into purely combinatorial problems involving the lattice of exponents.
       This is the technique that we wish to systematize.}
\end{quote}

Other textbooks of toric geometry were published since then by Oda \cite{O 88}, Fulton \cite{F 93}, Ewald \cite{E 96}, and Cox, Little and Schenck \cite{CLS 11}. One may find shorter introductions to toric geometry in Khovanskii \cite{K 77}, Danilov \cite{D 78}, Teissier \cite{T 81}, Gelfand, Kapranov and Zelevinsky \cite[Chapter II.5]{GKZ 94},  Cox \cite{C 03}, Barthel, Kaup and Fieseler \cite{BKF 07}, Brasselet \cite{B 08}, Garc\'{\i}a Barroso, Gonz\'alez P\'erez and the author \cite[Section 1.3]{GBGPPP 20}. The development of toric geometry was described in \cite[Appendix A]{CLS 11}.

In this section we give an introduction to toric varieties
adapted to our purpose of presenting in the rest of the text the more general notions
of {\em toroidal variety}, {\em complex log variety} and their {\em real oriented blowups}
and {\em roundings}, respectively. As {\em commutative monoids}
are the central algebraic objects of all three geometries, basic terminology about them is
introduced already in the next subsection (see \autoref{def:monoid}).

\medskip
\subsection{The classical passage to polar coordinates} $\:$  \label{ssec:polcoord}
\medskip

In this subsection we explain both the topological and the algebraic aspects of the classical
passage from cartesian to polar coordinates, first for the
complex affine line $\CC$ (see \autoref{def:orbu}),
then for the complex affine space $\CC^n$ (see \autoref{def:robuCn}).  We will see that particular
{\em commutative monoids} play important roles in these constructions:
$(\CC, \cdot), (\R_{\geq 0} \times \bS^1, \cdot)$ and
$(\N^n, +)$ for $n \geq 1$, where $\boxed{\N} := \Z_{\geq 0}$. Namely,
we will see in formula \eqref{eq:realorbudimnreint} that the passage to polar coordinates
of $\CC^n$ may be described as a morphism  of monoids
    \[  \Hom(\N^n,      \R_{\geq 0}  \times \bS^1)  \to   \Hom(\N^n,    \CC).   \]
It is then immediate to construct an analogous morphism of monoids by replacing
the monoid $(\N^n, +)$ by any commutative  monoid $\Gamma$. The natural context in which
the sets $\Hom(\N^n,  \CC)$ appear is that of {\em toric geometry}. This will lead
to our introduction to toric geometry, which spans Subsections \ref{ssec:basictoric} through
\ref{ssec:normaltoric}.
\medskip

The change of variables formulae from cartesian coordinates
$(x,y)$ to polar coordinates $(r, \theta)$ are most of the time written as
  \[ \left\{ \begin{array}{l}
                    x = r \cos \theta \\
                    y = r \sin \theta
               \end{array} \right. \]
or as
     \[ z = r \cdot e^{i \theta} \]
if $\boxed{z} := x + iy$. In what follows we will not use any of the transcendental functions
$\cos, \sin, e^{\bullet}$,  but we will rather write this change of variables as:
   \begin{equation}   \label{eq:changecomplexpolar}
        z = | z |  \cdot \mbox{sign}(z),
   \end{equation}
using the algebraic functions $| z | := \sqrt{x^2+y^2}$ and
$ \mbox{sign}(z) :=  \displaystyle{\frac{1}{\sqrt{x^2+y^2}} (x,y)}$ in the variables $x$ and $y$.
We will use the following terminology for the second algebraic function:

 \begin{definition}\label{def:logCoordinatesPolar}
     Let $\boxed{z} : \CC \to \CC$ be the {\bf standard coordinate function} on $\CC$.
     We denote by $\boxed{\bS^1}$ the unit
     circle in $\CC$, defined by the equation $| z | =1$.
     The \textbf{sign function}
     $\boxed{\mbox{sign}} \colon \CC^*\to \bS^1$ is the morphism of multiplicative abelian groups
     defined by:
         \[ \mbox{sign}(z):=z/|z|. \]
 \end{definition}

 \begin{remark}
     The standard coordinate function $z: \CC \to \CC$ connects the topological and algebraic aspects
     of the complex plane. Its source is to be thought of as a topological space and its target as a field.
  \end{remark}

\begin{remark}
    The name {\em sign function} is motivated by the fact that $\mbox{sign}$
    is an extension of the sign function
   $\R^* \to \{-1, 1\}$ seen as a morphism of multiplicative abelian groups.
   It is obtained by extending the
   source $\R^*$ to $\CC^*$ and the target $ \{-1, 1\}$ to $\bS^1$.
   This function  is a variant
   of the standard notion of argument of a
   non-zero complex number, which takes values in the abelian group $(\R / 2 \pi \Z, +)$
   and is defined by $ r e^{i \theta} \mapsto \theta  \mod 2 \pi$.
   Indeed, one passes from one function to the other one through the isomorphism of groups
      \[ \begin{array}{ccc}
              \bS^1 & \to & \R / 2 \pi \Z  \\
              e^{i \theta}  & \mapsto & \theta
          \end{array}\]
   In fact,  in~\cite[Section V.1.2]{O 18} Ogus even calls {\em argument} the  {\em sign function}
   of \autoref{def:logCoordinatesPolar}.
\end{remark}

The change of variable (\ref{eq:changecomplexpolar}) amounts to considering the map:
         \[     \begin{array}{cccc}
                 \boxed{\psi_{\CC}} : &    \CC  & \dashrightarrow  &  \R_{\geq 0} \times \bS^1  \\
                          &   z       & \mapsto  & \left(  | z |  ,\mbox{sign}(z) \right).
               \end{array}     \]
The dashed arrow indicates that this map is not defined everywhere; namely it is not defined
at the origin of $\CC$ and it cannot be extended by continuity to the whole complex
plane $\CC$. But, most importantly for us, the inverse
   \begin{equation}   \label{eq:changepolarcomplexmap}
              \begin{array}{cccc}
                 \boxed{\tau_{\CC}} : &   \R_{\geq 0}  \times \bS^1 & \to  &  \CC    \\
                          &   (r, u)        & \mapsto  & r \cdot u
               \end{array}
   \end{equation}
   of $\psi_{\CC} $ is everywhere defined and surjective. It is this inverse which is the prototype of
   {\em real oriented blowup} whose generalization we will study in this text.

   \begin{definition}  \label{def:orbu}
       The map $\tau_{\CC}$ defined by formula (\ref{eq:changepolarcomplexmap})
       is called the {\bf real oriented blowup of $\CC$ at the origin}.
   \end{definition}

   The terminology is motivated by the fact that $\tau_{\CC}$ is an analog of the usual
   {\bf real blowup}
     \[ \boxed{\pi_{\CC}}    : \boxed{\mathrm{Bl}_0 \CC}  \to \CC \]
   of $\CC$ at the origin, if we look at $\CC$ as a real vector space instead of as a complex one.
   Indeed, the real blowup
   of $\CC$ replaces the origin by the set of real lines through the origin, while the real
   {\em oriented} blowup replaces it by the set of {\em oriented} real lines: each point $(0,u)$ of the
   preimage $\tau_{\CC}^{-1}(0) = \{0\} \times \bS^1$ corresponds to the real line passing through
   $0$ and $u \in \bS^1 \hookrightarrow \CC$, oriented from $0$ to $u$ (see \autoref{fig:twobo}).

   The source surfaces of both kinds of blowups of $\CC$ at $0$ are traditionally
   also called {\em blowups of $\CC$}, like the associated morphisms.
   They may be obtained as closures of graphs of maps which are {\em undetermined} at the origin,
   from $\CC$ to particular circles:
    \begin{itemize}
       \item {\em for the real blowup}  $\mathrm{Bl}_0 \CC$, the map is the {\em real projectivization}
          $\boxed{\bP} : \CC= \R^2 \dashrightarrow \bP(\R^2)$,
           which sends each non-zero vector of $\R^2$ to the real line it generates in $\R^2$;
        \item {\em for the real oriented blowup} $ \R_{\geq 0}  \times \bS^1$,
             the map is the \emph{sign function}
            $\mbox{sign} \colon \CC  \dashrightarrow \bS^1$, which
           sends each non-zero vector of $\R^2$ to the {\em oriented} line it generates,
           labeled by the point of the trigonometric circle $\bS^1$ intercepted
           by the positive half-line of that oriented line.
    \end{itemize}
    The blowup maps $\tau_{\CC}$ and $\pi_{\CC}$ are parts of the following commutative diagram:
       \[  \xymatrix{
                                \bS^1
                                  \ar[d]_{\bP} &
                                                    \R_{\geq 0}  \times \bS^1   \ar[l]   \ar[d]   \ar[rd]^{\tau_{\CC}} &  \\
                                    \bP(\R^2)    & \mbox{Bl}_0 \CC  \ar[l] \ar[r]_{\pi_{\CC}}    &  \CC. }  \]
      The left vertical arrow $\boxed{\bP}$ is the restriction to $\bS^1$ of the real projectivisation map
      $\bP : \CC= \R^2 \dashrightarrow \bP(\R^2)$ above, while the horizontal arrows oriented to the left are
      induced by the projections on the second factors of the cartesian products
      $  \R_{\geq 0}  \times \bS^1$ and $\R^2 \times  \bP(\R^2)$.

 \begin{remark}   \label{rem:twoaspectsrobu}
    The real oriented blowup $\tau_{\CC}$ of formula (\ref{eq:changepolarcomplexmap})
 has two aspects:
   \begin{itemize}
       \item a {\bf topological aspect}: $\tau_{\CC}$ replaces the {\em algebro-geometric boundary}
           $0$ of $\CC$ by a
          {\em topological boundary}, that is, it is a homeomorphism outside them.
       \item an {\bf algebraic aspect}:  $\tau_{\CC}$ is a morphism of multiplicative {\em monoids}.
   \end{itemize}
      We will see a third {\em logarithmic aspect} in \autoref{rem:newvpolcoord}.
      In the successive generalizations of the real oriented blowup to be discussed in the sequel,
      {\em algebro-geometric boundaries will be again transformed into topological
      boundaries} and {\em monoids} will be crucial ingredients of such constructions.
  \end{remark}

   Let us recall the definitions of {\em monoids} and of their morphisms:

   \begin{definition} \label{def:monoid}
    A {\bf monoid} is a set endowed with an associative binary operation which
    has a neutral element. The monoid is {\bf commutative} if the operation is so.
    In this case, the neutral element is denoted by $0$ if the operation is written additively
    and by $1$ if it is written multiplicatively.
    A {\bf morphism of monoids} is a function from one monoid to a second one which
    respects neutral elements and the composition laws. A {\bf submonoid} of a given monoid is a
    subset which contains the neutral element and is closed under composition. A monoid is
    {\bf trivial} if it is reduced to its neutral element.
\end{definition}

   In the sequel we will be only interested in {\em commutative monoids}. Note that in the literature,
   monoids are also called {\em semigroups}, even if strictly speaking this last notion is
   more general, as a semigroup is a set endowed with an associative binary operation
   which may lack a neutral element. For instance, $(\Z_{> 0} , +)$ is a semigroup but not a monoid,
   by contrast with $(\N = \Z_{\geq 0} , +)$, which is a monoid. The main types of monoids to be
   considered in this paper are:
      \begin{itemize}
           \item  the multiplicative monoids $\CC^*$, $\CC$, $\R_{\geq 0}$, $\bS^1$,
           \item  {\bf lattices}, that is, free and finitely generated abelian groups,
           \item {\bf finite-dimensional real vector spaces},
           \item {\bf commutative rings} seen as monoids under multiplication,
      \end{itemize}
 as well as special kinds of {\em submonoids} thereof
 (see Definitions \ref{def:toricmonoid},  \ref{def:ratpolcone} and \ref{def:satmonoid}).

 \medskip
Let us come back to the real oriented blowup morphism $\tau_{\CC}$ of formula
\eqref{eq:changepolarcomplexmap}. One gets a first generalization of  it simply
by taking its finite powers:

\begin{definition}   \label{def:robuCn}
    The {\bf real oriented blowup of $\CC^n$ along its coordinate hyperplanes}
    is the $n$-th cartesian power of the real oriented blowup $\tau_{\CC}$ of
    formula \eqref{eq:changepolarcomplexmap}:
     \begin{equation}   \label{eq:realorbudimn}
              \begin{array}{cccc}
                \boxed{ \tau_{\CC^n}} : &   (\R_{\geq 0}  \times \bS^1)^n & \to  &  \CC^n    \\
                          &   (r_1, u_1, \dots, r_n, u_n)        & \mapsto  & (r_1 \cdot u_1, \dots, r_n   \cdot u_n)
               \end{array}      .
   \end{equation}
   \end{definition}

   The source space $(\R_{\geq 0}  \times \bS^1)^n$ is a manifold with corners, therefore
   it is a topological manifold-with-boundary. Denote by
   $\partial_{top} ((\R_{\geq 0}  \times \bS^1)^n)$ its boundary.
   The restriction of the map  $\tau_{\CC^n}$ to the complement of the boundary is a homeomorphism onto
    $( \CC^*)^n$:
         \[    (\R_{> 0}  \times \bS^1)^n  \xrightarrow{\sim}    (\CC^*)^n .  \]
    Therefore,  $ \tau_{\CC^n}$ may be seen topologically as a {\em replacement} of the
   {\bf algebro-geometric boundary} $\boxed{\partial \CC^n}$ of $\CC^n$,  equal to the
   union of coordinate hyperplanes, by the {\bf topological boundary}
   $\boxed{\partial_{top} ((\R_{\geq 0}  \times \bS^1)^n)}$ of $(\R_{\geq 0}  \times \bS^1)^n$.

   Let us look slightly differently at the map \eqref{eq:realorbudimn}, by relating both its source and
   its target to the additive monoid $(\N^n, +)$, which is a finitely generated submonoid
   of the lattice $(\Z^n, +)$. Namely, one may interpret the two spaces as
   {\em sets of morphisms in the category of monoids}:
        \begin{equation*}   \label{eq:monoidinterpr}
              \begin{array}{ccc}
                 (\R_{\geq 0}  \times \bS^1)^n    & =  & \Hom(\N^n,      \R_{\geq 0}  \times \bS^1),   \\
                  \CC^n &   =  &  \Hom(\N^n,    \CC).
               \end{array}
        \end{equation*}
   Here the targets of the two sets of morphisms
   are seen as multiplicative monoids. The advantage of this viewpoint is that
   it makes visible the fact that the complex $n$-dimensional real oriented blowup $ \tau_{\CC^n} $ of
   formula \eqref{eq:realorbudimn} is induced by the complex one-dimensional
   real-oriented blow-up $ \tau_{\CC} $ of formula \eqref{eq:changepolarcomplexmap} simply
   by post-composing with it:
        \begin{equation}   \label{eq:realorbudimnreint}
              \begin{array}{cccc}
                 \tau_{\CC^n} : & \Hom(\N^n,      \R_{\geq 0}  \times \bS^1)  & \to  &
                                              \Hom(\N^n,    \CC)   \\
                          &  \varphi     & \mapsto  & \tau_{\CC}    \circ \varphi
               \end{array}
     \end{equation}

        In this way, it becomes clear that the same construction may be performed if one replaces
        $\N^n$ by {\em any commutative monoid $\Gamma$}. In particular, if $\Gamma$
        is a finitely generated
        submonoid of a lattice, then it turns out that  $ \Hom(\Gamma,    \CC)$ is the set
        of complex-valued points of an {\em affine toric variety $\tv^{\Gamma}$}
        (see \autoref{def:afftoric}) and the associated morphism
             \begin{equation}   \label{eq:robarbtoricmon}
              \begin{array}{cccc}
                 \tau_{\tv^{\Gamma}} : & \Hom(\Gamma,      \R_{\geq 0}  \times \bS^1)  & \to  &
                                              \Hom(\Gamma,    \CC)   \\
                          &  \varphi     & \mapsto  & \tau_{\CC}    \circ \varphi
               \end{array}
             \end{equation}
         defined by analogy with \eqref{eq:realorbudimnreint} generalizes the real
         oriented blowup construction to those toric varieties (see \autoref{def:robuafftoric}).
         {\em Such morphisms may be glued together and they induce
         the operation of a real oriented blowup of an arbitrary toric variety which may be covered
         by affine toric varieties} (see \autoref{def:robarbtoric}).
         In order to explain these facts, we now turn to toric varieties.

\medskip
\subsection{General complex toric varieties and the lattices associated to their dense tori} $\:$  \label{ssec:basictoric}
\medskip

In this subsection we introduce {\em complex algebraic tori}, {\em affine and general toric varieties
and their boundaries}, {\em affinely covered toric varieties} and {\em morphisms of toric varieties}
 (see \autoref{def:toricvar}). We also introduce two dual lattices canonically associated
 to a complex algebraic torus: the lattice $M$ of
 {\em exponents of monomials} and the lattice $N$ of {\em weight vectors} (see \autoref{def:lattoric}).
\medskip

In the sequel, we will look at complex algebraic varieties only through the closed points of the
associated schemes, in order
to have a common underlying set for such a variety and for its associated complex analytic space.
That is, for us a {\em complex algebraic variety} will mean the ringed subspace of a
Noetherian scheme defined over $\CC$ obtained by restricting its structure sheaf to the subset
of closed points.

If $X$ is such a variety, we denote by $\boxed{\cO_X}$ its structure sheaf, which associates to each
Zariski open subset $U$ of $X$ the $\CC$-algebra $\boxed{\cO_X(U)}$ of regular functions on $U$.
In particular, $\cO_X(X)$ is the $\CC$-algebra of global regular functions on $X$.

{\em Complex toric geometry} is the branch of algebraic geometry which studies
{\em complex toric varieties} and their morphisms, defined as follows:

\begin{definition}   \label{def:toricvar}
    A {\bf complex algebraic torus} is a complex algebraic group isomorphic to
    the multiplicative group $(\CC^*)^n$, for some $n \in \N$.
   A complex algebraic variety $X$ is called {\bf toric} if it is endowed with an action of a complex
   algebraic torus $T$, called its {\bf dense torus},  and if  it has a base point $1$
   whose orbit is an open Zariski dense subset, identified with $T$
            by the multiplication map $T \to X$ which sends $t \in T$ to $t \cdot 1 \in X$.
    A toric variety is called {\bf affine} if it is an affine algebraic variety in the usual sense.
    If a toric variety may be covered by a finite set of affine Zariski open toric subvarieties,
    then it is called an {\bf affinely covered toric variety}.
     A {\bf toric morphism} between two  toric varieties is an algebraic morphism which is
    equivariant relative to the actions of their dense algebraic tori and which restricts to a
     morphism of groups between them. The {\bf toric boundary} $\boxed{\partial X}$
     of a toric variety $X$ is the complement of its dense torus.
\end{definition}

\begin{remark}
   Normal toric varieties (discussed in \autoref{ssec:normaltoric})  are affinely covered,
   as proved by Sumihiro \cite[Corollary 2]{S 74}
   (see also \cite[Section I.2, Theorem 5]{KKMS 73}).
   The terminology {\em affinely covered toric variety} is non-standard, as almost all the literature
   in toric geometry deals with such varieties, a fact which does not create the need to use a
   qualificative. The simplest example of toric
   variety which is not affinely covered is a nodal cubic curve {\huge $\alpha$} in a complex projective plane
   (see \cite[Example 3.A.1]{CLS 11}, noting that we chose this notation as an allusion to the topology
   of a particular real model in an affine chart):
    \begin{itemize}
        \item the smooth locus {\huge $\alpha^{\circ}$} of {\huge $\alpha$} is isomorphic to $\CC^*$ as a
                complex algebraic curve, which may be seen by projecting it from its nodal point;
         \item if we identify {\huge $\alpha^{\circ}$} with $\CC^*$ by a fixed isomorphism, then it acts on
                 itself by multiplication;
          \item this action extends to an algebraic action of $\CC^*$ on {\huge $\alpha$};
          \item any open set which contains both the nodal point and $\CC^*$
              is the whole curve {\huge $\alpha$},  which is projective, therefore not affine.
      \end{itemize}
      Note also  that sometimes, especially in symplectic topology, one drops from the
   definition of toric variety the condition that $X$ contains
   $T$, that is, no canonical base point is fixed (see Delzant \cite[Th\'eor\`eme 2.1]{D 88} and
   Cannas da Silva \cite[Chapter 28]{C 01}).
   Such toric varieties are to those of
   \autoref{def:toricvar} what affine spaces are to vector spaces.
\end{remark}

One works with an abstract complex algebraic torus $T$ and not only with the concrete one
$(\CC^*)^n$ in order to
be able to emphasize the intrinsic constructions, independent of the choice of coordinates. This
is analogous to reasoning in terms of abstract finite-dimensional vector spaces over a field $K$
 instead of in terms of just the vector spaces $K^n$.
 The name {\em torus} given to $((\CC^*)^n, \cdot)$ or to a complex algebraic
 group isomorphic to it comes from the fact that the usual torus $(\bS^1)^n$ of the category
 of manifolds embeds in $(\CC^*)^n$ as a deformation retract. More generally,
 the algebraic group $(K^*)^n$ is the analog of $(\bS^1)^n$ in the category of algebraic varieties
 over an algebraically closed field $K$: one says that an algebraic group defined over $K$
 and isomorphic to $(K^*)^n$ is an {\em algebraic $K$-torus}.

\begin{example}
      The basic examples of complex affine toric varieties are the affine spaces $\CC^n$, for $n \in \N$.
      The dense torus of $\CC^n$ is $(\CC^*)^n$, acting by:
      \[
              \begin{array}{ccc}
                      (\CC^*)^n \times \CC^n   &   \to   &   \CC^n   \\
                        (\underline{t}, \underline{z})  &   \mapsto   &     \underline{t} \cdot \underline{z}.
               \end{array}
      \]
       Here $\boxed{\underline{z}} := (z_1, \dots, z_n)$, $\boxed{\underline{t}} :=  (t_1, \dots, t_n)$
       and $ \boxed{\underline{t} \cdot \underline{z}} := (t_1 \cdot z_1, \dots, t_n \cdot z_n)$.
 \end{example}

{\bf In the sequel we will work only with affinely covered complex toric varieties.}
Most of the time, we will simply write {\em toric variety}, but the reader should interpret this syntagm as
meaning {\em affinely covered toric variety}. The reason why we considered more general
toric varieties in \autoref{def:toricvar} is to formulate an analogous
definition in more general categories than complex algebraic (see \autoref{def:toricvarS}).

Every toric variety $X$ may be {\em encoded combinatorially} by a finite set of objects
(a {\em fan of monoids}, see  definitions \ref{def:fanmonoids} and \ref{def:toricfromfanmon}),
which are defined using the following two lattices canonically associated with the dense torus
of $X$:

\begin{definition}  \label{def:lattoric}
   Let $T$ be a complex algebraic torus. Its {\bf exponent lattice}
   $M$ and its {\bf weight lattice} $N$  are defined by:
     \begin{equation} \label{eq:deflattices}
          \boxed{M}  := \Hom(T, \CC^*)   \mbox{ and } \ \boxed{N}  := \Hom(\CC^*, T),
     \end{equation}
   where the sets of morphisms are taken in the category of groups and their internal
   operation is defined by the group operation in the target.
\end{definition}

The names of the two lattices may be understood by looking at the special case
where $T = (\CC^*)^n$:
   \begin{itemize}
        \item  The elements of $M = \Hom((\CC^*)^n, \CC^*)$ are  the {\bf monomials}
            \[  \boxed{\underline{z}^{\underline{m}}} := z_1^{m_1} \cdots z_n^{m_n},   \]
            where $\boxed{\underline{m}} := (m_1, \dots, m_n) \in \Z^n$.
            Their multiplicative group is therefore canonically isomorphic to the additive
            group $(\Z^n, +)$ through the map
             $\underline{z}^{\underline{m}} \mapsto \underline{m}$. The convention is to always look at the
             group law of $M$ additively. It is for this reason that $M$ is seen as the lattice of
             {\em exponents of monomials}.
        \item  The elements of $ N = \Hom(\CC^*, (\CC^*)^n)$ are exactly the {\bf monomial curves},
          also called
        {\bf one-parameter subgroups} (even when the corresponding map is not an embedding)
             \begin{equation} \label{eq:moncurves}
              \begin{array}{ccc}
                      \CC^*  &   \to   &   (\CC^*)^n   \\
                         \lambda  &   \mapsto   &  \boxed{\lambda^{\underline{w}}} :=
                                                               (\lambda^{w_1}, \dots, \lambda^{w_n})
               \end{array}
             \end{equation}
           where $\boxed{\underline{w}} := (w_1, \dots, w_n) \in \Z^n$. By composing such a morphism
           with a monomial $\underline{z}^{\underline{m}} :  (\CC^*)^n \to \CC^*$, we get:
              \begin{equation}     \label{eq:expscalprod}
                 \underline{z}^{\underline{m}} \circ  \lambda^{\underline{w}} =
                           \lambda^{w_1 m_1 + \cdots + w_n m_n}.
              \end{equation}
            That is, a monomial curve may be seen as a way to attribute the weight $w_j$ to
            the variable $z_j$, which explains why $\underline{w}$ may be thought of as a
            {\bf weight vector}.
  \end{itemize}

    Looking at the exponents in formula \eqref{eq:expscalprod}, we get the following pairing:
         \begin{equation}   \label{eq:extrpair}
              \begin{array}{ccc}
                      \Z^n \times \Z^n &   \to   &   \Z   \\
                        (\underline{w}, \underline{m})  &   \mapsto   &
                          \boxed{ \underline{w} \cdot\underline{m}} := w_1 m_1 + \cdots + w_n m_n.
               \end{array}
        \end{equation}
      This pairing is {\bf unimodular}, that is, the determinant of its matrix relative to any couple of bases
      of the two lattices is $\pm 1$.

      We will interpret the pairing \eqref{eq:extrpair} intrinsically,
      starting from an arbitrary algebraic torus $T$.
      For this purpose we will use the notion of the {\bf degree} $\deg(\phi)$ of a morphism
      $\phi \in \Hom(\CC^*, \CC^*)$. Such a morphism $\phi$ is of the form
      $\phi(\lambda) = \lambda^d$, for some $d \in \Z$. We then set $\boxed{\deg(\phi)} := d$.
      Given $w \in N$ and $m \in M$, we may look at them as morphisms of groups
          \begin{equation} \label{eq:subgpmonomial}
              \boxed{\lambda^w} : \CC^* \to T,  \  \   \boxed{\chi^m} : T \to \CC^*.
          \end{equation}
      The composed morphism $\chi^m \circ \lambda^w: \CC^* \to \CC^*$ of groups
      allows to get the promised intrinsic description of the bilinear pairing \eqref{eq:extrpair}:
       \begin{equation}  \label{eq:canpairing}
              \begin{array}{ccc}
                      N \times M &   \to   &   \Z   \\
                        (w, m)  &   \mapsto   &    \deg(\chi^m \circ \lambda^w).
               \end{array}
        \end{equation}

    Since the pairing \eqref{eq:extrpair} is unimodular, it follows that the pairing \eqref{eq:canpairing}
    is also unimodular, for every complex algebraic torus
    $T$. Therefore, {\em the associated lattices $M$ and $N$ of} \eqref{eq:deflattices}
    {\em are naturally dual to each other}:
       \begin{equation} \label{eq:mutual}
             N = \Hom(M, \Z), \  M = \Hom(N, \Z).
       \end{equation}

   \begin{remark}
      The decision to denote the two lattices by the capital letters $M$ and $N$ goes back to
      \cite[Chapter I.1]{KKMS 73}. If $M$ is the initial of ``monomial'', the letter $N$
      is not the initial of any object connected to the context. For this reason, in his thesis the author
      denoted the weight lattice by the capital letter $W$ instead of $M$
      (see \cite[Section 4]{PP 04}). This notation has also the
      advantage that one exchanges the letters $M$ and $W$ by a simple plane involution,
      which points to the duality of the two lattices.  Nevertheless, in this text we keep the
      standard notations.
  \end{remark}

       \begin{definition}
            Let $w \in N$ and $m \in M$. Then, the corresponding morphisms of groups
            $\boxed{\lambda^w}$ and $\boxed{\chi^m}$
            of equation \eqref{eq:subgpmonomial}  are called the {\bf one-parameter subgroup
            with exponent $w$} and the {\bf monomial with exponent $m$} respectively.
       \end{definition}

    In formula \eqref{eq:deflattices}, we defined the lattice $M$ starting from the algebraic torus $T$.
    Conversely, $T$ may be reconstructed from $M$ as follows:
       \begin{equation}   \label{eq:reconstrtorus}
            T = \Hom(M, \CC^*).
       \end{equation}
    Therefore, each one of the objects $M, N, T$ determines canonically the other two
    as a group of morphisms, through the formulae \eqref{eq:deflattices},
    \eqref{eq:mutual} and \eqref{eq:reconstrtorus}. In particular,
    objects associated to $T$ may be also seen as objects associated to $M$
    or to $N$.

    Note that the group structure of a complex algebraic torus $T$ is determined
    by its structure of an affine algebraic variety and by its base point $1 \in T$. Indeed
    (we leave the proof as an exercise):

    \begin{proposition}
         The monomials $\chi^m : T \to \CC^*$ are the only functions $f \in \cO_T(T)$ such that:
            \begin{itemize}
                \item $f$ is nowhere vanishing;
                \item $f(1) = 1$.
            \end{itemize}
    \end{proposition}

    Therefore, the complex algebra $\cO_T(T)$ of global regular functions on $T$
    may be also seen as the {\bf algebra
     $\boxed{\CC[M]}$ of the lattice $M$} (see also \autoref{def:groupalg}),
     whose elements are the formal finite linear
     combinations with complex coefficients of elements of $M$. More precisely,
     the following map is an isomorphism of algebras:
        \[  \begin{array}{ccc}
                  \cO_T(T)   &   \to    &   \CC[M]   \\
                   \sum_m c_m \chi^m  &   \mapsto    &    \sum_m c_m m
            \end{array} .  \]
     In particular, $T$ may be seen as the set of complex-valued points
     of the spectrum of this algebra:
         \begin{equation}   \label{eq:reconstrspec}
               T = \Hom(\Spec \CC,   \Spec \CC[M]).
          \end{equation}
    If the morphisms  in equality \eqref{eq:reconstrtorus} were taken in the category of abelian groups,
    in equality \eqref{eq:reconstrspec} they are taken instead in the category of complex affine varieties.
    The first viewpoint endows the set $T$ with its group structure and the second one
    with its structure of a complex algebraic variety.

    \medskip
    \subsection{Affine toric varieties and their monoids of exponents}   $\  $ \label{ssec:}
    \medskip

    In this subsection we introduce {\em toric monoids} (see \autoref{def:toricmonoid})
    and their associated {\em affine toric varieties} (see \autoref{def:afftoric}). We explain
    that all affine toric varieties are determined by a corresponding toric monoid
    (see \autoref{prop:allafftoric}).
    \medskip

Let now $X$ be an arbitrary {\em affine} complex toric variety. As its dense torus $T$ is
a dense Zariski open subset of $X$, the associated restriction morphism
  \begin{equation}  \label{eq:embafftoric}
       \cO_X(X) \to \cO_X(T) = \cO_T(T)
  \end{equation}
of complex algebras is injective and birational. It may be shown by looking at the
eigenvectors of the action of $T$ on it that the subalgebra $\cO_X(X) $ of $\cO_T(T)$
is generated by monomials (see \cite[Proposition 1, page 4]{KKMS 73}). The exponents
of the monomials contained in $\cO_X(X) $ form obviously a {\em submonoid} $\Gamma$
of the lattice $(M, +)$. As $X$ is assumed to be Noetherian, this monoid is finitely generated.
As the morphism \eqref{eq:embafftoric} is birational,  the embedding $\Gamma \hookrightarrow M$
is isomorphic to the embedding of $\Gamma$ in the {\em group generated by it}. This is
an opportunity to introduce more vocabulary related to monoids (see \cite[Section I.1.3]{O 18}):

\begin{definition}   \label{def:intmon}
    Let $\Gamma$ be a  monoid. Its {\bf edge} $\boxed{\Gamma^{\star}}$ is its subgroup of invertible
    elements. A monoid is called {\bf sharp} if its edge is the trivial submonoid.
    The {\bf group generated by $\Gamma$} is denoted by $\boxed{\Gamma^{\gp}}$.
    The monoid $\Gamma$ is called {\bf integral} if the natural morphism of monoids
    $\Gamma \to \Gamma^{\gp}$ is injective and {\bf unit-integral} if its restriction
    $\Gamma^{\star} \to \Gamma^{\gp}$ to the edge of $\Gamma$ is injective.
\end{definition}

\begin{remark}  \label{rem:terminedge}
      The terminology {\em sharp monoid} is taken from  \cite[Section I.1.3]{O 18}.
         Instead, the terminology {\em edge of a monoid} is non-standard (for instance,
         in \cite[Section I.1.3]{O 18} $\Gamma^{\star}$ is called the {\em group of units}
         of $\Gamma$, pointing to a
         multiplicative interpretation of the operation of $\Gamma$). We chose this terminology
         by thinking about the special case of a monoid equal to a dihedral angle $D$ in a
         $3$-dimensional real vector space, that
         is, at the intersection of two closed half-spaces whose boundary planes contain the origin. Then,
         $D$ is a monoid with respect to the usual addition of vectors
         and $D^{\star}$ is the usual edge of the dihedral angle,
         obtained as the intersection of the two boundary planes.
\end{remark}

\begin{remark}
     The natural morphism $\Gamma \to \Gamma^{\gp}$ is the universal morphism of $\Gamma$
     to an abelian group. Its injectivity is equivalent to the cancellation property:
        \[   \gamma_1 + \gamma = \gamma_2 + \gamma \   \implies \  \gamma_1  = \gamma_2, \
              \mbox{ for every } \gamma_1,  \gamma_2, \gamma  \in \Gamma. \]
      For this reason, integral monoids are also called {\em cancellative monoids}.
      In turn, a monoid is unit-integral if and only if one has the weaker cancellation property:
         \[   \gamma_1 + \gamma = \gamma_2 + \gamma \   \implies \  \gamma_1  = \gamma_2, \
              \mbox{ for every } \gamma_1,  \gamma_2 \in \Gamma \mbox{ and } \gamma \in \Gamma^{\star}.\]
\end{remark}

The fact that finitely generated submonoids of lattices determine affine toric varieties
explains the following terminology, where such monoids are described intrinsically, without
an a priori reference to an ambient lattice:

\begin{definition}  \label{def:toricmonoid}
     A commutative monoid is called {\bf toric} if it is finitely generated, integral and if
     the group generated by it is a lattice.
\end{definition}

Let us come back to an affine complex toric variety $X$ with dense torus $T$.
The considerations of the paragraph preceding \autoref{def:intmon} lead to the following
identification of $\CC$-algebras:
    \begin{equation} \label{eq:eqalg}
          \cO_X(X) = \CC[\Gamma],
    \end{equation}
where $\Gamma$ is a toric submonoid of the lattice $M$ of exponents of monomials of $T$
such that $\Gamma^{\gp} = M$. Here, $ \CC[\Gamma]$ denotes the {\em complex monoid algebra}
of the monoid $\Gamma$, which we next define.

\begin{definition}   \label{def:groupalg}
   Let $(\Gamma, +)$ be a monoid. Its {\bf complex monoid algebra} $ \CC[\Gamma] $
   is the complex vector space freely generated by $\Gamma$.
   Conventionally, the element of $\CC[\Gamma]$ associated to $m \in \Gamma$
   is denoted $\boxed{\chi^m}$ and is called the {\bf monomial with exponent $m$}.
    The product in the ring $\CC[\Gamma]$  is defined by:
      \[  ( \sum_{m \in \Gamma} a_m  \chi^m ) ( \sum_{p \in \Gamma} b_p  \chi^p ) :=
            \sum_{m, p \in \Gamma } a_m b_p \chi^{m+p} .\]
\end{definition}

Note that the complex algebra $\CC[M]$ of the lattice $M$, introduced in the last paragraph of
\autoref{ssec:basictoric}, is a particular example of a complex monoid algebra.

Similarly to the way in which the dense torus $T$ may be
reconstructed from the group $(M, +)$ (see the equalities \eqref{eq:reconstrtorus} and
\eqref{eq:reconstrspec}), the affine toric
variety $X$ can be reconstructed from the monoid $(\Gamma, +)$ of monomials of
$T$ which are regular on $X$ as:
        \begin{equation*}   \label{eq:reconstrafftoricvar}
            X = \Hom(\Gamma, \CC).
       \end{equation*}
 This time, {\em the morphisms are taken in the category of monoids} and $\CC$ is seen
 as a {\em multiplicative monoid}. Note that if we perform this construction starting
 for the monoid $\Gamma = (M, +)$, we get the complex torus $T$, as
    \[ \Hom(M, \CC) = \Hom(M, \CC^*) \]
 in the category of monoids. Indeed, as $M$ contains only invertible elements and as
 the neutral element $0$ of $M$ must be sent to the neutral element $1$ of $\CC$, no element
 of $M$ can be sent to $0 \in \CC$ by an element of $\Hom(M, \CC)$.

 Let us introduce the following notation for the toric variety defined by a toric monoid $\Gamma$:

 \begin{definition}  \label{def:afftoric}
     Let $(\Gamma, +)$ be a toric monoid. The {\bf associated complex affine
     toric variety $\boxed{\tv^\Gamma}$} is the variety of closed points of the scheme
     $\Spec \CC [\Gamma]$, whose underlying set is  $\Hom(\Gamma, \CC)$.
     Then, $(\Gamma, +)$ is called the {\bf exponent monoid of the variety $\tv^\Gamma$}.
 \end{definition}

 The reason why we write $\Gamma$ as an exponent in $\tv^\Gamma$ is explained in \autoref{rem:positionmon}.

\begin{example}  \label{ex:smoothafftoric}
    Assume that $\Gamma = \N^n$. Then, $\tv^{\N^n} =  \Hom(\N^n, \CC) = \CC^n$. The last equality
    is a consequence of the fact that giving a morphism of monoids
    from $(\N^n, +)$ to $(\CC, \cdot)$
    is the same as giving $n$ morphisms of monoids from $(\N, +)$ to $(\CC, \cdot)$, that is,
    $n$ complex numbers, the images of the generators $1$ in each of the $n$ copies
    of the monoid $(\N, +)$.
\end{example}

Our considerations leading to the equality \eqref{eq:eqalg} show that:

\begin{proposition}  \label{prop:allafftoric}
     Each complex affine toric variety is isomorphic to a variety of the form $\tv^{\Gamma}$,
     for a well-defined toric monoid $\Gamma$.
\end{proposition}

\medskip
\subsection{Morphisms of affine toric varieties}   $\  $   \label{ssec:morafftoric}
\medskip

In this subsection we explain how to describe all the toric morphisms between affine toric varieties using
morphisms between their associated toric monoids (see \autoref{prop:toricmorph}).
As a special case, we describe all the toric-invariant affine Zariski-open subsets of affine
toric varieties (see \autoref{prop:chartoraffopen}). Such open subsets are important because
all affinely covered toric varieties are obtained from affine ones by gluing them along such open subsets,
a fact which will be explained in \autoref{ssec:combdescrtorvar}.
\medskip

Let us now turn to discussing the toric morphisms between affine toric varieties, as introduced in
\autoref{def:toricvar}.

\begin{example}  \label{ex:monmorph}
    Toric morphisms $\psi : \CC^n \to \CC^l$ are {\em monomial} morphisms,
    that is, morphisms of the form
         $\psi(\underline{z}) = (\underline{z}^{m_1}, \dots, \underline{z}^{m_l})$, with
         $m_1, \dots, m_l \in \N^n$ (this explains the use of the term {\em monomials} in
         the citation from \cite{KKMS 73} at the beginning of \autoref{sec:introtoric}).
          For instance, the toric morphisms from $\CC^2$ to $\CC^3$
         are those of the form:
            \[
              \begin{array}{cccc}
                   \psi :  & \CC^2 &   \to   &   \CC^3   \\
                      & (z_1, z_2)   &   \mapsto   &  (z_1^{a_1} z_2^{a_2} ,
                          z_1^{b_1} z_2^{b_2}, z_1^{c_1} z_2^{c_2})
               \end{array}
             \]
          where all exponents are non-negative integers.
\end{example}

Using \autoref{prop:allafftoric} and the fact that a toric morphism must send monomials to monomials,
one can prove that toric morphisms between arbitrary
affine toric varieties may be described using morphisms between their exponent monoids
in the sense of \autoref{def:afftoric}:

\begin{proposition}   \label{prop:toricmorph}
    Let $(\Gamma_1, +)$ and $(\Gamma_2, +)$ be two toric monoids. The toric morphisms from
    the affine toric variety $\tv^{\Gamma_1}$ to $\tv^{\Gamma_2}$ are those of the form:
          \begin{equation*}  \label{eq:morphtoricgeom}
              \begin{array}{cccc}
                   \boxed{\chi^{\mu}} :  & \tv^{\Gamma_1} &   \to   &   \tv^{\Gamma_2}   \\
                      & x_1   &   \mapsto   &  x_1 \circ \mu ,
               \end{array}
        \end{equation*}
     where $\mu$ varies in the set $\Hom(\Gamma_2, \Gamma_1)$ of morphisms of monoids
     and where $x_1 \circ \mu$ denotes the composition:
        \[ \Gamma_2  \xrightarrow{\mu} \Gamma_1  \xrightarrow{x_1} \CC. \]
     The associated morphism of monoid algebras is:
          \begin{equation*}  \label{eq:morphtoricalg}
              \begin{array}{cccc}
                   \boxed{\chi_{\mu}} :  & \CC[\Gamma_2] &   \to   &   \CC[\Gamma_1]   \\
                      & \chi^m   &   \mapsto   &  \chi^{\mu(m)}.
               \end{array}
        \end{equation*}
\end{proposition}

\begin{remark}    \label{rem:positionmon}
    If $\Gamma_2 = \N$, then the commutative monoid $(\Hom(\Gamma_2, \Gamma_1), +)$
    gets canonically  identified with $(\Gamma_1, +)$ by sending each element
    $ \mu \in \Hom(\N, \Gamma_1)$ to $\mu(1) \in \Gamma_1$.
    In this way, the toric morphism
     $ \chi^{\mu} :  \tv^{\Gamma_1}   \to    \tv^{\N} = \CC$ gets identified with the monomial
     $\chi^{\mu(1)} \in \CC[\Gamma_1]$. This shows that
     the notations $\chi^m$ of \autoref{def:afftoric} and $\chi^{\mu}$ of
     \autoref{prop:toricmorph} are compatible. Similarly, note that for $\Gamma_1 = \Z$ and
     $\Gamma_2 = M$, the commutative monoid  $(\Hom(\Gamma_2, \Gamma_1), +)$
     gets canonically identified with
     $N$, by the identity $\Hom(M, \Z) = N$ induced by the pairing \eqref{eq:canpairing}.
     With this in mind, $ \mu \in \Hom(M, \Z)$ becomes a weight vector in $N$,  and
     in the notations of formula \eqref{eq:moncurves}, one has
     $\chi^{\mu} = (\lambda \to \lambda^{\mu})$ when $M = \Z^n$.

     Note also that the positions of $\mu$ as
     exponent and index in $\chi^{\mu}$  and $\chi_{\mu}$ respectively, are compatible with the conventions
     of writing the contravariant/covariant functors in homology theory ($H^{\bullet}$ is
     contravariant, while $H_{\bullet}$ is covariant). For the same reason, in the notation
     ``$\tv^{\Gamma}$'', we write the monoid $\Gamma$ as an exponent, as the map
     of objects  $\Gamma \mapsto \tv^{\Gamma}$ is induced by a contravariant functor
     from the category of toric monoids to that of complex affine toric varieties (see also
     \autoref{rem:positionfan}).
\end{remark}

\begin{example}
      Let us consider again a monomial morphism  $\psi : \CC^n \to \CC^l$ defined by
      $\psi(\underline{z}) = (\underline{z}^{m_1}, \dots, \underline{z}^{m_l})$,
      as in \autoref{ex:monmorph}. Then, we have  $\psi = \chi^{\mu}$, where
      $\mu : \N^l \to \N^n$ is such that $\mu(e_i) = m_i$ for every $i \in \{1, \dots, l\}$.
      Here $(e_1 := (1, 0, \dots, 0), \dots, e_l := (0, \dots, 0, 1))$ is the canonical
      freely generating sequence of the monoid $\N^l$. That is, the matrix of $\mu$
      is $(m_1, \dots, m_l)$, where $m_1, \dots, m_l \in \N^n$ are written as column vectors.
      For instance, the morphism
            \[
              \begin{array}{cccc}
                   \psi :  & \CC^2 &   \to   &   \CC^3   \\
                      & (z_1, z_2)   &   \mapsto   &  (z_1^{a_1} z_2^{a_2} ,
                          z_1^{b_1} z_2^{b_2}, z_1^{c_1} z_2^{c_2})
               \end{array}
             \]
        is the morphism $\chi^{\mu}$, where the matrix of $\mu : \N^3 \to \N^2$ is
           \[    \left(  \begin{array}{cccc}
                                a_1 &  b_1   &   c_1 \\
                                a_2 &  b_2   &   c_2
               \end{array}   \right).   \]
\end{example}

By definition, all affinely covered complex toric varieties are obtained by gluing together several affine ones.
Those gluings are done along torus-invariant affine Zariski open
subsets of affine toric varieties. Let us describe those subsets combinatorially. This description depends
on the notion of {\em face} of a monoid (see \cite[Definition 1.4.1]{O 18}):

\begin{definition}   \label{def:facemon}
     Let $(\Gamma, +)$ be a monoid. A {\bf face} of it is a submonoid $\Phi$ of $\Gamma$
     such that:
           \[  \forall \ p, q \in \Gamma, \ (p + q \in \Phi \ \implies \  p, q \in \Phi). \]
\end{definition}

\begin{example}
   If $\Gamma = (\R_{\geq 0})^n$ or $\Gamma = \N^n$
   for some $n \in \N^*$, then the faces of $\Gamma$ are all its intersections
   with the coordinate subspaces of $\R^n$. Therefore, there are $2^n$ of them. Ordered
   by inclusion, the smallest one is the trivial submonoid and the largest one is the whole monoid $\Gamma$.
\end{example}

The following proposition is the announced characterization of the toric-invariant
affine Zariski open subsets of the complex affine toric variety $\tv^{\Gamma}$
(see Gonz\'alez P\'erez and Teissier's  \cite[Lemma 19]{GPT 14}):

\begin{proposition}  \label{prop:chartoraffopen}
      Let $(\Gamma, +)$ be a toric monoid, seen as a submonoid of the lattice $M = \Gamma^{\gp}$.
      Then the inclusion morphisms
      $\chi^{\mu} : \tv^{\Lambda} \hookrightarrow \tv^\Gamma$
      of the affine Zariski-open torus-invariant subsets of $\tv^{\Gamma}$ are given by the
      inclusion morphisms $\mu : \Gamma \hookrightarrow \Lambda$, with $\Lambda$ varying among
      the submonoids of $M$ of the form $\Gamma + \Phi^{\gp}$, where $\Phi$ is a face of $\Gamma$.
\end{proposition}

\begin{example}   \label{ex:toropenC2}
     Let $\Gamma := \N^2$. Then, we have that $\tv^\Gamma = \CC^2$. Denote $x := \chi^{(1,0)}$
     and $y := \chi^{(0,1)}$.
     Then we have that $\CC[\Gamma] = \CC[x,y]$. The faces of $\Gamma$ are:
        \begin{itemize}
            \item  $\Gamma = \N^2$.
            \item  $ \Phi_1 = \N (1,0)$.
            \item  $ \Phi_2 = \N (0,1)$.
            \item  $O := \{(0,0)\}$.
        \end{itemize}
   The corresponding submonoids of $M = \Z^2$ are:
      \begin{itemize}
            \item  $\Gamma +  \Gamma^{\gp} = \Gamma + M = M = \Z^2$.
            \item  $ \Gamma + \Phi_1^{\gp} = \Gamma + \Z (1,0) = \Z \times \N$.
            \item  $ \Gamma + \Phi_2^{\gp} = \Gamma + \Z (0,1) = \N \times \Z$.
            \item  $\Gamma + O^{\gp} = \Gamma = \N^2$.
        \end{itemize}
    Therefore, the corresponding affine Zariski-open torus-invariant  subsets of $\tv^{\N^2} = \CC^2$ are:
       \begin{itemize}
            \item  $\tv^{\Z^2} = (\CC^*)^2$.
            \item  $\tv^{\Z \times \N} = \CC^* \times \CC$.
            \item  $\tv^{\N \times \Z} = \CC \times \CC^*$.
            \item  $\tv^{\N^2} = \CC^2$.
        \end{itemize}
    This example will be continued in \autoref{ex:dualconesplane}.
\end{example}

     Let $(\Gamma, +)$ be a monoid. One may characterize its faces among its submonoids
     by looking at their complements (see \cite[Section 1.4]{O 18}). Indeed, the defining property
           \[  \forall \ p, q \in \Gamma, \ (p + q \in \Phi \ \implies \  p, q \in \Phi) \]
      of a face $\Phi$ is equivalent to:
             \[  \forall \ a, b \in \Gamma, \ (b  \in \Gamma \smallsetminus \Phi \ \implies \  a +
                       b \in \Gamma \smallsetminus \Phi). \]
      This means, by definition, that $\Gamma \smallsetminus \Phi$ is an {\bf ideal} of the monoid
      $\Gamma$. That is, {\em the faces of $\Gamma$ are exactly the submonoids
      of $\Gamma$ whose complements are ideals of $\Gamma$}.
      Note that only proper ideals may appear as complements of faces, because
      submonoids are necessarily non-empty, as they contain the neutral element $0$ of $(\Gamma, +)$.

      In turn, the ideals $\Pi$ of $\Gamma$ obtained in this way are exactly the ideals
      whose complement $\Gamma \setminus \Pi$ is a submonoid of $\Gamma$, a property
      which may be written in terms of $\Pi$ as:
          \[ \Pi \neq \Gamma \mbox{ and }
               \forall \ p, q \in \Gamma, \ (p + q \in \Pi \ \implies \  p \in \Pi \mbox{ or } q \in \Pi). \]
      By definition, the ideals verifying this conditions are called {\bf prime}. We conclude that:

      \begin{proposition}
         The operation
         \[  \Phi  \  \mapsto \  \Gamma \smallsetminus \Phi \]
         establishes a bijection between the set of faces of $\Gamma$ and the set of
         prime ideals of $\Gamma$.
      \end{proposition}

       \begin{remark} \label{rem:specmonoid}
         The set of prime ideals of a monoid is called, by analogy with the case of rings,
         the {\bf spectrum of the monoid}. In \cite[Sections 5 and 9]{K 94}, Kato introduced a structure
         of {\em monoidal space} on the spectrum of a monoid, that is, of topological space
         endowed with a sheaf of monoids.
      \end{remark}

         We leave it as an exercise to prove that the edge of a monoid in the sense of
         \autoref{def:intmon} may be also characterized
         either as its maximal subgroup or as its minimal face. Note that the complement
         of the edge of a monoid is its unique maximal ideal. Therefore,
         a monoid is an analog of a local ring rather than of an arbitrary ring.

\medskip
\subsection{A combinatorial description of toric varieties using fans of monoids}
$ \ $   \label{ssec:combdescrtorvar}
\medskip

In this subsection we explain how any complex affinely covered toric variety may be seen as the
direct limit of its affine toric Zariski-open subsets (see \autoref{prop:identifdirlimit}). This allows to
encode it by a {\em fan of monoids} (see \autoref{prop:charfinitesetmon},
\autoref{def:fanmonoids} and \autoref{def:toricfromfanmon}).
Finally, we explain how morphisms of such varieties may be combinatorially encoded
by morphisms of fans of monoids (see \autoref{prop:morphtoricgen}).

\medskip

Recall that for us a {\em toric variety} means a {\em complex affinely covered toric variety}
in the sense of \autoref{def:toricvar}. The combinatorial description of such a variety
passes through the following result (folklore), which amounts to saying that it may be obtained
by gluing all its toric affine Zariski open subsets using their natural inclusion morphisms:

\begin{proposition}  \label{prop:identifdirlimit}
     Let $X$ be a complex toric variety. Then,
     the inclusion morphisms $\tv^{\Gamma} \hookrightarrow X$ of all toric affine Zariski open subsets
     of $X$ identify $X$ with the direct limit of those toric affine varieties $\tv^{\Gamma}$,
     endowed with their inclusion morphisms.
\end{proposition}

In the same way as a complex affine toric variety may be described using its associated monoid
(see \autoref{def:afftoric}), one may describe an arbitrary
toric variety $X$ using the finite set of monoids associated to all its toric affine Zariski open subsets.
It turns out that those are not arbitrary finite sets of submonoids of the lattice $M$ of exponents
of the dense algebraic torus $T$ of $X$ (see \autoref{def:lattoric}). In order to characterize such sets
in \autoref{prop:charfinitesetmon} below, we need to introduce more terminology concerning {\em cones}
relative to lattices:

\begin{definition}  \label{def:ratpolcone}
   Let $N$ be a lattice. The {\bf associated real vector space} $\boxed{N_{\R}}$ is
     $N \otimes_{\Z} \R \hookleftarrow N$.  A {\bf polyhedral cone} $\sigma$ in $N_{\R}$ is a
     submonoid of $N_{\R}$ of the form $\boxed{\R_{\geq 0}\langle w_1, \dots, w_r\rangle}$,
     consisting of the linear combinations of finitely many vectors $w_1, \dots, w_r \in N_{\R}$,
     with {\em non-negative} real coefficients.
     The polyhedral cone $\sigma$ is called {\bf rational} if those vectors may be chosen in $N$
     and it is called {\bf sharp} if it is a sharp monoid in the sense of \autoref{def:intmon}.
\end{definition}

In the following definition, the toric monoid $\Gamma$ is to be thought of as the exponent monoid
of an affine toric variety (see \autoref{def:afftoric}):

\begin{definition}  \label{def:twocones}
     Let $\Gamma$ be a toric monoid, $M:=  \Gamma^{\gp} \hookleftarrow \Gamma$
     be the lattice generated by it and $N := \Hom(M, \Z)$ be its dual lattice.
     One associates to $\Gamma$ the following rational polyhedral cones:
         \[  \begin{array}{ccc}
               \boxed{\sigma(\Gamma)}              &  :=
                   & \{ w \in N_{\R}, \  w \cdot m \geq 0, \forall \ m \in \Gamma \}  \subseteq N_{\R}; \\
               \boxed{\sigma^{\vee}(\Gamma)}  &  :=
                   &  \{ m \in M_{\R}, \  w \cdot m \geq 0, \forall \ w \in \sigma(\Gamma) \}
                        \subseteq M_{\R}.
             \end{array}     \]
     $\sigma(\Gamma)$ is called the {\bf weight cone} of $\Gamma$ and
     $\sigma^{\vee}(\Gamma) \hookleftarrow \Gamma$ is called the {\bf exponent cone} of $\Gamma$.
\end{definition}

By construction, $\sigma^{\vee}(\Gamma)$ is the {\bf dual cone of $\sigma(\Gamma)$}.
The duality operator is an involution on polyhedral cones: conversely, $\sigma(\Gamma)$
is the dual cone of $\sigma^{\vee}(\Gamma)$.

A basic property of the cones $\sigma(\Gamma)$ and $\sigma^{\vee}(\Gamma)$ is the following:

\begin{lemma}  \label{lemma:dualcones}
   If $\Gamma$ is a toric monoid, then both cones $\sigma(\Gamma)$ and $\sigma^{\vee}(\Gamma)$
   are rational and polyhedral in the sense of \autoref{def:ratpolcone}. Moreover:
        \begin{enumerate}
              \item $\sigma(\Gamma)$ is sharp.
              \item  \label{secondefsigmagamma}
                    $\sigma^{\vee}(\Gamma) = \R_{\geq 0} \Gamma$.
             \item $\sigma^{\vee}(\Gamma)$ has non-empty interior in $M_{\R}$.
             \item The edge of the cone $\sigma^{\vee}(\Gamma)$ is
                 the orthogonal $\R \sigma(\Gamma)^{\perp}$ of the
                 real vector subspace $\boxed{\R \sigma(\Gamma)}$ of $N_{\R}$
                 generated by $\sigma(\Gamma)$.
             \item The cone $\sigma(\Gamma)$ and the edge of $\sigma^{\vee}(\Gamma)$
                 have complementary dimensions:
                     \[  \dim \sigma(\Gamma) + \dim \R \sigma(\Gamma)^{\perp} = \dim M_{\R}.  \]
        \end{enumerate}
\end{lemma}

\begin{remark}
      Lemma \ref{lemma:dualcones} (\ref{secondefsigmagamma}) shows that the cone
      $\sigma^{\vee}(\Gamma)$ may be directly defined from $\Gamma$, as well as its dual
      $\sigma(\Gamma)$. Why have we privileged $\sigma(\Gamma)$ by using
      the duality notation $\sigma^{\vee}(\Gamma)$
      for the cone $\R_{\geq 0} \Gamma$ and not the other way round? The reason
      is that {\em normal} toric varieties are defined using nice families of sharp rational
      polyhedral cones contained in
      the weight space $N_{\R}$, called {\em fans} (see \autoref{def:fan}),
      and that one associates to such a cone
      $\sigma$ a monoid by considering first its dual $\sigma^{\vee}$, then the intersection
      $M \cap \sigma^{\vee}$ (see \autoref{thm:fanormalvar}).
  \end{remark}

\begin{example}  \label{ex:dualconesplane}
   Let us consider the monoids $ \Z^2, \Z \times \N,  \N \times \Z, \N^2$ corresponding
   to the toric Zariski-open subsets of $\CC^2$ (see  \autoref{ex:toropenC2}).
   The associated weight cones, which are subcones of $N_{\R}= \R^2$, are:
        \begin{itemize}
            \item  $\sigma( \Z^2) = \{0\}$,
            \item  $ \sigma(\Z \times \N) = \{0\} \times \R_{\geq 0}$,
            \item  $ \sigma(\N \times \Z) =  \R_{\geq 0} \times \{0\}$,
            \item  $ \sigma(\N^2) =  \R_{\geq 0} \times  \R_{\geq 0}$,
        \end{itemize}
   while their dual exponent cones, which are subcones of $M_{\R} = \R^2$, are:
      \begin{itemize}
            \item  $\sigma^{\vee}( \Z^2) = \R^2$,
            \item  $ \sigma^{\vee}(\Z \times \N) = \R \times \R_{\geq 0}$,
            \item  $ \sigma^{\vee}(\N \times \Z) =  \R_{\geq 0} \times \R$,
            \item  $ \sigma^{\vee}(\N^2) =  \R_{\geq 0} \times  \R_{\geq 0}$.
        \end{itemize}
    We see that, unlike the set $\{ \sigma^{\vee}( \Z^2), \sigma^{\vee}(\Z \times \N),
        \sigma^{\vee}(\N \times \Z),  \sigma^{\vee}(\N^2) \}$ of exponent cones,
    the finite set $\{ \sigma( \Z^2), \sigma(\Z \times \N), $
       $\sigma(\N \times \Z),  \sigma(\N^2) \}$ of weight cones consists exclusively of sharp
       cones and is closed under taking faces and intersections.
     This property holds for every affine toric variety and it motivates \autoref{def:fan} below.
\end{example}

 \begin{definition}  \label{def:fan}
     A finite set $\cF$ of sharp rational polyhedral
     cones of $N_{\R}$ is called a {\bf fan} if it satisfies simultaneously:
        \begin{itemize}
             \item it is closed under taking faces;
             \item it is closed under taking intersections.
        \end{itemize}
\end{definition}

We are now ready to formulate the promised characterization of the finite sets of monoids corresponding
to toric varieties:

\begin{proposition}  \label{prop:charfinitesetmon}
      Let $T$ be a complex algebraic torus with exponent lattice $M$ and weight lattice $N$.
      Then, the finite sets $I$ of toric submonoids of $M$ corresponding to affinely covered toric
      varieties with dense torus $T$ may be characterized in the following way:
        \begin{enumerate}
             \item  For every $\Gamma \in I$, one has $\Gamma^{\gp} = M$.
             \item The map which associates to each $\Gamma \in I$ its weight cone
                  $\sigma(\Gamma)$ is injective.
             \item  The set
                 \[   \{ \sigma(\Gamma), \Gamma \in I \} \]
                 forms a  fan in $N_{\R}$, in the sense of \autoref{def:fan}.
             \item \label{pointfaces}
                 For every $\Gamma \in I$, the elements of $I$ containing it
                 are the monoids of the form
                  \[ \Gamma + (\Gamma \cap \tau^{\perp})^{\gp} \]
                where $\tau$ varies among the faces of the cone $\sigma(\Gamma)$ and
                $\tau^{\perp} \subseteq N_{\R}$ denotes the orthogonal of $\tau$.
                In this case, we have $\sigma( \Gamma + (\Gamma \cap \tau^{\perp})^{\gp}) = \tau$.
        \end{enumerate}
\end{proposition}

\begin{remark}
    If $\tau$ is a face of $\sigma(\Gamma)$, then the orthogonal $\tau^{\perp}$ is
    the edge of the rational polyhedral cone
    $\tau^{\vee}$ and $\Gamma \cap \tau^{\perp}$ is a face of $\Gamma$.
    The map $\tau \mapsto \Gamma \cap \tau^{\perp}$ is an inclusion-reversing bijection from the set
    of faces of $\sigma(\Gamma)$ to the set of faces of $\Gamma$. Point \eqref{pointfaces} of
    \autoref{prop:charfinitesetmon} results from \autoref{prop:chartoraffopen}.
\end{remark}

Let us turn the properties of the monoids of $I$ and the fan $ \{ \sigma(\Gamma), \Gamma \in I \}$
appearing in \autoref{prop:charfinitesetmon}
into the following definition, equivalent to part of Gonz\'alez P\'erez and Teissier's
\cite[Definition 21]{GPT 14}:

\begin{definition}  \label{def:fanmonoids}
     Let $M$ and $N$ be dual lattices, $N$ being interpreted as the weight lattice of an algebraic
     torus. A {\bf fan of monoids} is the datum $\boxed{\hat{\cF}}$ of a fan $\cF$ of cones of
     $N_{\R}$ in the sense of \autoref{def:fan}, called the {\bf underlying fan} of $\hat{\cF}$,
     and of a toric submonoid $\boxed{\Gamma_{\sigma}}$ of $M$ for every $\sigma \in \cF$,
     such that the following conditions are simultaneously satisfied:
        \begin{itemize}
             \item $(\Gamma_{\sigma})^{\gp} = M$ for every $\sigma \in \cF$.
             \item $ \Gamma_{\tau} = \Gamma_{\sigma} + (\Gamma_{\sigma} \cap \tau^{\perp})^{\gp}$,
             for every  $\tau, \sigma \in \cF$ such that $\tau$ is a face of $\sigma$.
        \end{itemize}
\end{definition}

The author and Stepanov had already introduced this notion in \cite[Definition 3.4]{PPS 13}
under the name ``{\em fan of semigroups}''. The terminology ``{\em fan of monoids}'' seems to be new.

\autoref{prop:charfinitesetmon} states that every affinely covered toric variety
is encoded by a fan of monoids.
By \autoref{prop:identifdirlimit}, we see that the toric variety associated to such a fan of monoids
may be obtained as follows:

\begin{definition}  \label{def:toricfromfanmon}
    Let $\hat{\cF}$ be a fan of monoids, with underlying fan $\cF$.
    Its {\bf associated toric variety} is defined as the direct limit:
     \[  \boxed{\tv_{\hat{\cF}} }\  := \lim_{\substack{\longrightarrow \\ \sigma \in \cF}}
             \tv^{\Gamma_{\sigma}} .
     \]
    Here, the connecting morphisms between the affine toric varieties $\tv^{\Gamma_{\sigma}}$
    are the morphisms $\tv^{\Gamma_{\tau}} \to \tv^{\Gamma_{\sigma}}$
     induced by the inclusion morphisms $\Gamma_{\sigma} \hookrightarrow \Gamma_{\tau}$,
    whenever $\tau \subseteq \sigma$.
\end{definition}

Note that, as a consequence of the fact that $\cF$ is a fan, one has an inclusion
$\tau \subseteq \sigma$ as above if and only if $\tau$ is a face of $\sigma$.

Toric morphisms between toric varieties in the sense of \autoref{def:toricvar}
may be characterized as the morphisms between the associated
weight lattices which are {\em compatible with the two corresponding fans of monoids},
in the following sense:

\begin{proposition}  \label{prop:morphtoricgen}
     Let $\hat{\cF_1}$ and $\hat{\cF_2}$ be two fans of monoids in two lattices $N_1$ and
     $N_2$. Then, the toric morphisms $\tv_{\hat{\cF_1}} \to \tv_{\hat{\cF_2}}$ are
     induced by the morphisms $\nu : N_1 \to N_2$ of lattices such that:
        \begin{itemize}
           \item  For every $\sigma_1 \in \cF_1$, there exists some cone $\sigma_2 \in \cF_2$ such that
              $\nu(\sigma_1) \subseteq \sigma_2$.
           \item If $\sigma_1 \in \cF_1$ and $\sigma_2 \in \cF_2$ are as above
              and if their associated monoids are $\Gamma_1 \subseteq M_1$ and
              $\Gamma_2 \subseteq M_2$,
              then the dual morphism $\nu^{\vee} : M_2 \to M_1$ of $\nu$ satisfies the inclusion relation:
                  \[  \nu^{\vee}(\Gamma_2) \subseteq \Gamma_1.  \]
        \end{itemize}
\end{proposition}

\begin{remark}  \label{rem:positionfan}
   This is a sequel to \autoref{rem:positionmon}. In the notation ``$\tv_{\hat{\cF}}$'',
   the fan of monoids is written as a subscript, as the map
     of objects  $\hat{\cF} \mapsto \tv_{\hat{\cF}}$ is induced by a covariant functor
     from the category of fans of monoids to that of complex affinely covered toric varieties.
\end{remark}

\medskip
\subsection{Normal toric varieties}   \label{ssec:normaltoric}
$\  $
\medskip

In this subsection we introduce the notion of {\em saturated monoid} (see \autoref{def:satmonoid})
and we explain that {\em normal} complex toric varieties are exactly those whose
associated fan of monoids contains only
saturated monoids (see \autoref{def:satmonoid}).
\medskip

In most of the textbooks or introductory papers about toric geometry mentioned at the beginning
of \autoref{sec:introtoric},  only {\em normal} affine toric varieties are considered
(exceptions are \cite[Chapter 5]{GKZ 94}, \cite{C 03} and \cite[Appendix of Chapter 3]{CLS 11}). For this
reason (even if we do not restrict to them in the sequel, because
real oriented blowups may be performed
on arbitrary toric varieties), we will recall now how to characterize them in terms of
the corresponding fans of monoids. Since normality is a local property of an algebraic variety,
it suffices to describe the normal affine toric varieties in terms of their monoids.
It turns out that  their corresponding monoids are exactly the {\em saturated} ones, in the following sense:

\begin{definition} \label{def:satmonoid}
     Let $(\Gamma, +)$ be a toric monoid in the sense of \autoref{def:toricmonoid}.
     Identify it with its image inside the
     group $M := \Gamma^{\gp}$ generated by it. Its {\bf saturation} $\boxed{\overline{\Gamma}}$ is the set
     of elements $m$ of $M$ such that there exists $a \in \N^*$ with $a \cdot m \in \Gamma$.
     The toric monoid $(\Gamma, +)$ is called {\bf saturated} if it is equal to its saturation.
\end{definition}

Recall that any morphism of toric monoids induces a morphism of affine toric varieties
(see \autoref{prop:toricmorph}).
Here is the announced characterization of the normal affine toric varieties
(see \cite[Lemma 1, page 5]{KKMS 73}):

\begin{proposition}  \label{prop:charnormaltoric}
    Let $(\Gamma, +)$ be a toric monoid.  Then the toric morphism
    $\chi^{\nu} : \tv^{\overline{\Gamma}} \to \tv^{\Gamma}$ associated to the inclusion
    morphism $\nu : \Gamma \hookrightarrow \overline{\Gamma}$ of $\Gamma$ in its saturation
    $\overline{\Gamma}$ is a normalization morphism
    of $\tv^\Gamma$. Therefore, the toric variety $\tv^\Gamma$ is normal if and only if
    $\Gamma$ is saturated.
\end{proposition}

\begin{example}  \label{ex:normaliz}
    Assume that $\Gamma = \N \langle 2, 3 \rangle$, that is,
    $\Gamma$ is the submonoid of $(\N, +)$ generated by
    $2$ and $3$. Therefore, $\tv^\Gamma = \CC[x^2, x^3] = \CC[u,v]/(u^3 - v^2)$,
    where $u = x^2, v = x^3$.
    The group generated by $\Gamma$ is $\Gamma^{\gp} = \Z$,
    and its saturation is $\overline{\Gamma} = \N$. Denote by $\nu : \Gamma \hookrightarrow \N$
    the associated inclusion morphism.
    The corresponding toric morphism $\chi^{\nu} : \tv^{\N} = \CC \to \tv^\Gamma$ is given
    at the level of the corresponding algebras by:
       \[   \begin{array}{cccc}
                    & \CC[x^2, x^3] = \CC[u,v]/(u^3 - v^2) &   \to   &   \CC[\N] = \CC[x]   \\
                      & (u,v)   &   \mapsto   &  (x^2, x^3).
               \end{array}     \]
     We get the usual normalization morphism $x \mapsto (x^2, x^3)$ of the cuspidal affine plane cubic
     curve defined by the equation $u^3 - v^2 =0$.
\end{example}

\begin{remark}   \label{rem:closedim}
      In \autoref{ex:normaliz}, the non-normal affine toric variety $\tv^{ \N \langle 2, 3 \rangle}$
     appears as the image of a finite toric morphism between normal affine toric varieties. Therefore,
    if one wants the category of toric varieties to be closed under taking images of finite toric
    morphisms, one is forced to admit inside it also non-normal toric varieties.
\end{remark}

It turns out that the saturation $\overline{\Gamma}$ of a toric monoid $\Gamma$ is determined by
its exponent cone $\sigma^{\vee}(\Gamma)$ of \autoref{def:twocones}
(see \cite[Proposition 1.3.8]{CLS 11}):

\begin{lemma}
    Let $\Gamma$ be a toric monoid. Then, we have that
       $\overline{\Gamma} = M \cap \sigma^{\vee}(\Gamma)$.
\end{lemma}

Combining this with \autoref{prop:identifdirlimit} and \autoref{prop:charfinitesetmon},
we see that {\em normal} toric varieties can be described in terms of fans alone, without
the need to decorate them with monoids:

\begin{theorem}  \label{thm:fanormalvar}
     Let $T$ be a complex algebraic torus, with weight lattice $N$. Then, the normal
     toric varieties with dense torus $T$ are the direct limits of the normal affine
     varieties $\tv^{M \cap \tau^{\vee}}$, for $\tau$ varying among the cones of a fan
     $\cF$ contained in $N_{\R}$, where the connecting morphisms
     $\tv^{M \cap \tau^{\vee}}\to \tv^{M \cap \sigma^{\vee}}$ are those determined
     by the inclusions $\tau \hookrightarrow \sigma$ of cones of $\cF$.
\end{theorem}

\medskip
\subsection{Toroidal geometry}   \label{ssec:toroidalgeom}
$ \ $
\medskip

In this subsection we introduce {\em toroidal varieties} (see \autoref{def:toroidal}) and their
{\em morphisms} (see \autoref{def:toroidalmorph}). These complex analytic
varieties, locally modeled on toric ones, were introduced
by Kempf, Knudson, Mumford and Saint-Donat in \cite[Chapter II.1]{KKMS 73}.
\medskip

Toric varieties are highly structured complex algebraic varieties: they are endowed with a group action, therefore  with a partition into orbits. {\em Toroidal varieties} are less structured varieties, which imitate in some respect the toric ones. Namely, they are first of all complex analytic varieties locally isomorphic to toric varieties. But this is not the whole definition. One also fixes a hypersurface in them, and one asks those local isomorphisms to send the hypersurface onto the toric boundary:

\begin{definition}   \label{def:toroidal}
       A {\bf toroidal variety} is a pair $(X, \partial X)$ consisting of a complex analytic variety
       $X$ (the {\bf total space}) and a reduced divisor $\partial X$ in it, such that every point $x \in X$
       is contained in an open set $U$ with the property that
       the pair $(U, U \cap \partial X)$ is analytically isomorphic to the pair
       $(V, V \cap \partial Y)$ where $V$ is an open subset of an affine toric variety $Y$
       with toric boundary  $\partial Y$ in the sense of \autoref{def:toricvar}.
       Such an isomorphism $(U, U \cap \partial X) \xrightarrow{\sim} (V, V \cap \partial Y)$ is called
       a {\bf toric chart} of $(X, \partial X)$. The divisor
       $\boxed{\partial X}$ is called the {\bf toroidal boundary} of the toroidal variety  $(X, \partial X)$.
       In order not to charge notations, we will denote a toroidal variety $(X, \partial X)$ simply by $X$.
\end{definition}

Note that in the previous definition the point $x$ is not constrained to lie in $\partial X$.
If one takes it in the complement $X\  \setminus \ \partial X$ of the toroidal boundary,
then one deduces that {\em this complement is necessarily smooth}. Note also that
a toric variety is canonically toroidal if one defines its toroidal boundary to be equal to its
toric boundary (see \autoref{def:toricvar}). What is lost when one passes from
a toric variety to the associated toroidal
variety is the action of the torus, as well as the notion of a monomial. But not everything is
lost about them: the union of the supports of the divisors of all the monomials
is exactly the toroidal boundary.

The notion of a toroidal variety was introduced by Kempf, Knudsen, Mumford and Saint-Donat
in \cite{KKMS 73} under a slightly different form.
Namely, they considered the pairs $(X, X\  \setminus \ \partial X)$ instead of $(X, \partial X)$,
as the emphasis in that book was on {\em partial compactifications} of complex {\em manifolds}, and
they considered compact toroidal varieties $X$ as especially nice partial compactifications
of $X\  \setminus \ \partial X$. Moreover,
they constrained the resulting space $X$ to be normal.
We drop this condition for the same reason we dropped it in the
definition of toric varieties (see \autoref{rem:closedim}).

The simplest toroidal varieties are defined by divisors with normal crossings:

\begin{definition}  \label{def:nctorvar}
    A complex variety $X$ is called a {\bf normal crossings toroidal variety} if $X$ is a
    manifold and $\partial X$ is a {\bf divisor with normal crossings} in it,
    that is, a divisor such that each point
    of $\partial X$ has a neighborhood $U$ in $X$ with the property that the pair $(U, U \cap \partial X)$
    is analytically isomorphic to a pair of the form $(V, V \cap Z(z_1 \cdots z_n))$,
    where $V$ is an open set in $\CC^n$ and $Z(z_1 \cdots z_n)$ is the zero locus of the
    product $z_1 \cdots z_n$.
\end{definition}

Normal crossings toroidal varieties are exactly the toroidal varieties whose total spaces are
smooth. Indeed, the only smooth toric varieties of a fixed dimension $n$
are $\CC^n = \Hom(\N^n, \CC)$ and its torus-invariant open subsets.

In the same way as toroidal varieties are modeled on toric varieties, toroidal morphisms are
modeled on toric morphisms, in the sense of \autoref{def:toricvar}:

\begin{definition}   \label{def:toroidalmorph}
    Let $X_1$ and $X_2$ be two toroidal varieties.
    A {\bf toroidal morphism}
      \[   \phi : X_1 \to X_2 \]
    is a complex analytic morphism such that for every $x \in X_1$, there are toric charts
    $\chi_1 : (U_1,  U_1 \cap \partial X_1) \xrightarrow{\sim} (V_1, V_1 \cap \partial Y_1)$  and
    $\chi_2 : (U_2, U_2 \cap \partial X_2) \xrightarrow{\sim} (V_2, V_2 \cap \partial Y_2)$
    with the property that  $x \in U_1$, $\phi(x) \in U_2$ and the composition
    $\chi_2 \circ \phi \circ \chi_1^{-1} : (V_1, V_1 \cap \partial Y_1)
       \to (V_2, V_2 \cap \partial Y_2)$ making the following diagram  commutative
            \[  \xymatrix{
                       (U_1, U_1 \cap \partial X_1)
                             \ar[rr]^{\phi}   \ar[d]_{\chi_1} &
                            & (U_2, U_2 \cap \partial X_2)     \ar[d]^{\chi_2} \\
                         (V_1, V_1 \cap \partial Y_1)  \ar[rr]_{\chi_2 \circ \phi \circ \chi_1^{-1}}&  &
                          (V_2, V_2 \cap \partial Y_2)   } \]
     is the restriction to $V_1$ of a toric morphism from $Y_1$ to $Y_2$.
\end{definition}

\begin{example}   \label{ex:toroidmorph}  $\ $

  \begin{enumerate}
     \item Any toric morphism is automatically toroidal if one endows its source and target
           with their canonical toroidal structures.
      \item  If $X$ is a manifold and $f : X \to \CC$ is a holomorphic function whose divisor
        $Z(f) \hookrightarrow X$ has normal crossings, then $f$ becomes a toroidal
        morphism if one defines $\partial X$ to be the reduction of $Z(f)$ and $\partial \CC$
        to be the origin.
    \end{enumerate}
\end{example}

\section{\bf Toric, toroidal and A'Campo's real oriented blowups}   \label{sec:torictoroidrob}

In this section we explain how the classical passage to polar coordinates
$ \R_{\geq 0}  \times \bS^1  \to    \CC$
recalled in \autoref{ssec:polcoord} may be generalized into an operation of a {\em real oriented
blowup} of any complex affinely covered toric variety (see \autoref{ssec:polcoordtoric}),
then of any toroidal variety (see \autoref{ssec:realbutoroidal}). A'Campo's definition
of the real oriented blowup of a divisor with simple normal crossings in a complex manifold
looks different, but it turns out to be a special case of real oriented blowup of a toroidal
variety (see \autoref{ssec:acamporobu}). Note that Hubbard, Papadopol and Veselov
introduced in \cite[Section 5]{HPV 00} a notion of a real oriented blowup of any analytic
subset of a real analytic manifold. We do not discuss it in this text.

\medskip
\subsection{Real oriented blowups of toric varieties} $\:$  \label{ssec:polcoordtoric}
\medskip

 In this subsection we explain how to perform the real oriented blowup of any complex affine
 toric variety (see \autoref{def:robuafftoric}), of any toric morphism between
 such varieties (see \autoref{def:robuafftoricmorph}) and finally of any complex
 affinely covered toric variety (see \autoref{def:robarbtoric}).
 \medskip

{\em We continue assuming that toric varieties are affinely covered in the sense
of \autoref{def:toricvar}.}

Let us define first the real oriented blowups of complex affine toric varieties. One simply replaces
the monoid $(\CC, \cdot)$ in \autoref{def:afftoric} of complex affine toric varieties by
the monoid $( \R_{\geq 0} \times \bS^1, \cdot)$, as seen already in formula
(\ref{eq:robarbtoricmon}):

\begin{definition}   \label{def:robuafftoric}
     Let $\Gamma$ be a toric monoid and let $\tv^\Gamma$ be the associated complex
     affine toric variety.
     The {\bf real oriented blowup of $\tv^\Gamma$}
     is the real semi-algebraic variety defined by
     \[   \boxed{(\tv^\Gamma)^{\rob}} := \Hom(\Gamma,  \R_{\geq 0} \times \bS^1), \]
     where the morphisms are taken in the category of monoids.
    The {\bf real oriented blowup morphism of $\tv^\Gamma$}  is defined by:
        \begin{equation*}   \label{eq:robutoric}
              \begin{array}{cccc}
                 \boxed{\tau_{\tv^\Gamma}} : & (\tv^\Gamma)^{\rob}   & \to  &    \tv^\Gamma   \\
                                     &     x   & \mapsto  &  \tau_{\CC} \circ x
               \end{array}
          \end{equation*}
      In this formula,  $ \tau_{\CC} \circ x$ denotes the composition of the
      morphisms of monoids $x : \Gamma \to  \R_{\geq 0} \times \bS^1$ and
      $\tau_{\CC} :  \R_{\geq 0} \times \bS^1 \to \CC$
      (see formula (\ref{eq:changepolarcomplexmap})).
\end{definition}

\begin{example}   \label{ex:smoothafftoricrob}
   Consider $\Gamma := \N^n$, as in \autoref{ex:smoothafftoric}. Then $\tv^{\N^n} = \CC^n$ and
   $ (\tv^{\N^n})^{\rob} = (\R_{\geq 0} \times \bS^1)^n$. The real oriented blowup morphism
   $\tau_{\tv^{\N^n}} : (\R_{\geq 0} \times \bS^1)^n \to \CC^n$ is the cartesian product
   of $n$ copies of the prototypical real oriented blowup morphism $\tau_{\CC}$ of
   formula (\ref{eq:changepolarcomplexmap}). Therefore, it coincides with the
   real oriented blowup map $\tau_{\CC^n}$ from \autoref{def:robuCn}. This shows that
   \autoref{def:robuafftoric} is a generalization of  \autoref{def:robuCn} to all complex affine
   toric varieties.
\end{example}

The real oriented blowup of a complex affine toric variety is not a toric variety in the sense
of \autoref{def:toricvar}, because it is not a complex algebraic variety. Nevertheless, it satisfies the
remaining properties of that definition. This motivates us to introduce the following concepts:

\begin{definition}   \label{def:toricvarS}
    Let $\mathcal{C}$ be a subcategory of the category of topological spaces,
    containing the category of complex algebraic varieties.
    An object $X$ of $\mathcal{C}$ is called {\bf $\mathcal{C}$-toric}
    if it is endowed with an action of a complex
    algebraic torus $T$, called its {\bf dense torus}  and if
    it has a base point $1$ whose orbit is an open dense subset, identified with $T$
            by the multiplication map $T \to X$ which sends $t \in T$ to $t \cdot 1 \in X$.
    A {\bf $\mathcal{C}$-toric morphism} is a morphism in the category $\mathcal{C}$
    between two $\mathcal{C}$-toric spaces, which is
    equivariant relative to the actions of their dense algebraic tori and which restricts to a
     morphism of groups between them.
\end{definition}

One may check that:

\begin{proposition}
      Real oriented blowups of complex affine toric varieties are {\em semi-algebraic toric
     varieties}, that is $\mathcal{SA}$-toric varieties for the category $\boxed{\mathcal{SA}}$ of
     semi-algebraic spaces and morphisms.
\end{proposition}

One may also define the real oriented blowup of a toric morphism between affine toric
varieties (recall \autoref{prop:toricmorph}):

\begin{definition}   \label{def:robuafftoricmorph}
        Let $\Gamma_1$ and $\Gamma_2$ be two toric monoids and
        $\mu : \Gamma_2 \to  \Gamma_1$ a morphism of monoids. Let
                     $ \chi^{\mu} :   \tv^{\Gamma_1}   \to     \tv^{\Gamma_2} $
     be the associated toric morphism.
     Its {\bf real oriented blowup} is the map
         \begin{equation*}  \label{eq:morphtoricgeomrob}
              \begin{array}{cccc}
                   \boxed{(\chi^{\mu})^{\rob}} :  & (\tv^{\Gamma_1})^{\rob} &   \to   &   (\tv^{\Gamma_2})^{\rob}   \\
                      & x   &   \mapsto   &  x \circ \mu ,
               \end{array}
        \end{equation*}
     where  $x \circ \mu$ denotes the composition
        $ \Gamma_2  \xrightarrow{\mu} \Gamma_1  \xrightarrow{x} \R_{\geq 0} \times \bS^1. $
\end{definition}

It is immediate to check from the definition that:

\begin{proposition}  \label{prop:functrob}
   The real oriented blowup is a functor from the category of
    complex affine toric varieties to the category of semi-algebraic toric varieties.
\end{proposition}

Moreover, this functor is compatible with the real oriented blowup morphisms of
\autoref{def:robuafftoric}:

\begin{proposition}
    The following diagram is commutative:
       \[  \xymatrix{
                                  (\tv^{\Gamma_1})^{\rob}
                                  \ar[rr]^{(\chi^{\mu})^{\rob}}
                                  \ar[d]_{\tau_{\tv^{\Gamma_1}}} &
                                                &  ( \tv^{\Gamma_2})^{\rob}   \ar[d]^{\tau_{\tv^{\Gamma_2}}} \\
                                    \tv^{\Gamma_1}  \ar[rr]_{\chi^{\mu}}&  & \tv^{\Gamma_2}. }  \]
\end{proposition}

Indeed, if
   $x \in (\tv^{\Gamma_1})^{\rob} = \Hom(\Gamma_1,  \R_{\geq 0} \times \bS^1)$,
then its images by the arrows of the previous diagram are:
     \[  \xymatrix{
                                x
                                  \ar[rr]^{(\chi^{\mu})^{\rob}}
                                  \ar[dd]_{\tau_{\tv^{\Gamma_1}}} &
                                                &   x \circ \mu   \ar[dd]^{\tau_{\tv^{\Gamma_2}}} \\
                                            &     \circlearrowleft  &  \\
                                    \tau_{\CC} \circ x  \ar[rr]_{\chi^{\mu}}&  & \tau_{\CC} \circ x \circ \mu. }  \]

Consider now an arbitrary toric variety. One may define its real oriented blowup
analogously to \autoref{def:toricfromfanmon}, by gluing
the real oriented blowups of its affine Zariski open invariant subvarieties through the
real oriented blowups of their gluing morphisms:

\begin{definition}  \label{def:robarbtoric}
    Let $\hat{\cF}$ be a fan of monoids in the sense of \autoref{def:fanmonoids}, with underlying fan $\cF$.
    The {\bf real oriented blowup} of its associated toric variety $\tv_{\hat{\cF}}$
    is defined as the direct limit:
     \[  \boxed{\tv_{\hat{\cF}}^{\rob} }\  := \lim_{\substack{\longrightarrow \\ \sigma \in \cF}}
            (\tv^{\Gamma_{\sigma}})^{\rob}.
           \]
    Here, the connecting morphisms between the real algebraic toric varieties
    $(\tv^{\Gamma_{\sigma}})^{\rob}$
    are those induced by the inclusion morphisms $\Gamma_{\tau} \hookrightarrow \Gamma_{\sigma}$,
    whenever $\tau \subseteq \sigma$. The real oriented blowup morphisms
      $\tau_{\tv^\Gamma} :  (\tv^\Gamma)^{\rob}    \to      \tv^\Gamma$ glue
     into a corresponding global {\bf real oriented blowup morphism}
       \[ \boxed{\tau_{ \tv_{\hat{\cF}}}} :  \tv_{\hat{\cF}}^{\rob}    \to      \tv_{\hat{\cF}}. \]
\end{definition}

As in the complex algebraic case, the real oriented real blowup of a toric variety
is a semi-algebraic toric variety, in the sense of \autoref{def:toricvarS}.

\medskip
\subsection{Real oriented blowups of toroidal varieties} $\:$  \label{ssec:realbutoroidal}
\medskip

In this subsection we define the real oriented blowup of a toroidal variety by gluing
the real oriented blowups of its toric charts (see \autoref{def:toroidrob}).

\medskip

Once one knows how to define real oriented blowups of toric varieties, an analogous notion
may be defined for toroidal varieties using their {\em toric charts} in the sense of
\autoref{def:toroidal}. One has to ensure that this notion is well-defined on overlaps
of charts. Indeed:

\begin{proposition}   \label{prop:toroidcompar}
    Let $X$ be a toroidal variety and let $\phi_i : (U, U \cap \partial X) \to (V_i, V_i \cap \partial Y_i)$
    for $i \in \{1, 2\}$ be two charts of it, defined on the same open set $U$. Consider the
    real oriented blowup morphisms $\tau_{Y_i}: Y_i^{\rob} \to Y_i$  of the toric varieties $Y_i$
    and their restrictions $\tau_{V_i}: V_i^{\rob} \to V_i$ to their open sets $V_i$, where
    $V_i^{\rob} := \tau_{Y_i}^{-1}(V_i)$.  Then,  the toroidal isomorphism
    $\phi_2 \circ \phi_1^{-1} : V_1 \to V_2$ lifts in a unique way to a homeomorphism
    $(\phi_2 \circ \phi_1^{-1})^{\rob}: V_1^{\rob} \to V_2^{\rob}$ making the following
    diagram commute:
         \[  \xymatrix{
                                V_1^{\rob}
                                  \ar[rr]^{(\phi_2 \circ \phi_1^{-1})^{\rob}}
                                  \ar[d]_{\tau_{V_1}} &
                                                &    V_2^{\rob}   \ar[d]^{\tau_{V_2}} \\
                                  V_1  \ar[rr]_{\phi_2 \circ \phi_1^{-1}}  &  & V_2 \\
                                       &  U  \ar[lu]^{\phi_1}  \ar[ru]_{\phi_2}& }.  \]
\end{proposition}

\begin{proof}
     One could prove the proposition directly, following the arguments of \cite[Page 37]{M 84},
     where only the case of smooth toroidal varieties was considered. Another method is
     to use \autoref{prop:robisround}, which identifies the real oriented blowup of a complex affine
     toric variety with the rounding of its natural toroidal log structure. This implies the desired
     statements because the analytic isomorphism $\phi_2 \circ \phi_1^{-1}$ induces an
     isomorphism of the associated toroidal log structures.
\end{proof}

The uniqueness statement in \autoref{prop:toroidcompar} ensures that whenever
one has three toric charts $(\phi_i)_{i \in \{1, 2, 3\}}$ as above on an open set $U$
of the toroidal variety, then one has automatically the cocycle relation:
    \[ (\phi_3 \circ \phi_2^{-1})^{\rob} \circ (\phi_2 \circ \phi_1^{-1})^{\rob} = (\phi_3 \circ \phi_1^{-1})^{\rob}. \]
Therefore, in the presence of an atlas $(\phi_j : U_j \to V_j)_{j \in J}$
of $X$ by toric charts, the lifts $(\phi_k \circ \phi_j^{-1})^{\rob}$ of its transition
maps $\phi_k \circ \phi_j^{-1}$ glue in a compatible way. This leads to the following
definition:

\begin{definition}   \label{def:toroidrob}
   Let $X$ be a toroidal variety. Its {\bf real oriented blowup} is obtained by gluing the real
   oriented blowups of the sources of its toric charts using the lifts of their transition maps
   described in \autoref{prop:toroidcompar}.
\end{definition}

 \begin{remark}\label{rem:Kawrob}
         In  \cite[Definition 2.1]{K 02}, Kawamata defined a notion of real oriented blowups for
         {\em quasi-smooth toroidal varieties}, whose toric charts use only
         \emph{simplicial} normal affine toric varieties. His notion is a particular instance of
         \autoref{def:toroidrob}.
    \end{remark}

\medskip
\subsection{A'Campo's definition of a real oriented blowup} $\:$  \label{ssec:acamporobu}
\medskip

In this subsection we explain A'Campo's definition of a real oriented blowup of a complex
manifold along a divisor with simple normal crossings (see \autoref{def:acampodef}).
Introduced by A'Campo \cite{A 75}
in 1975, it was chronologically the first appearance of such operation. It became classical
very quickly, as shown by the fact that it was used without
reference by Persson \cite[Section 2.2]{P 77} and Majima \cite[Section I.3]{M 84}. Those
are the two papers mentioned by Kato and Nakayama as references for real oriented blowups
in their paper \cite{KN 99}, in which they introduced {\em rounding}, their generalization
of real oriented blowup to complex log spaces.
\medskip

Once one has a definition of a real oriented blowup for toroidal varieties (see \autoref{def:toroidrob}),
one may specialize it to normal crossings toroidal varieties (see \autoref{def:nctorvar}).
This leads to a notion of {\em real oriented blowup of a complex manifold along a reduced
divisor with normal crossings}.

The first definition of such an operation seems to have been formulated in 1975 by A'Campo,
for divisors with {\em simple} normal crossings, that is, whose irreducible components are {\em smooth}.
He described it as follows in \cite[Page 238]{A 75} (here  $X_0$
is a divisor with simple normal crossings inside a complex manifold $X$
and $(C_j)_{1 \leq j \leq n}$ are its irreducible components, which are taken as initial
{\em centers} of real oriented blow up):

\begin{quote}
   {\em   For $j =1, \dots , n$ let $ \pi_j=  Z_j \to X$ be the real oriented blowup with center $C_j$.
   Therefore, above $x\in C_j$ lie the real oriented normal directions to $C_j$ at $x$ and $\pi_j$ is a
   diffeomorphism outside $C_j$. Thus $Z_j$ is a differentiable manifold-with-boundary and its
   boundary $\partial Z_j = \pi_j^{-1}(C_j)$  is diffeomorphic to the boundary of a tubular
   neighborhood of $C_j$ in $X$. Let $\pi :Z \to X$ be the fibered product of the various $\pi_j$ above $X$.
   Then $Z$ is a differentiable manifold with corners and $\pi$ is a diffeomorphism outside $X_0$.
   The boundary $\partial Z= \pi^{-1}(X_0)=N$ is a differentiable manifold with corners.
   The restriction of $\pi$ to $\partial Z$ is the map $\rho : N \to X_0$. The manifold $N$
   is homeomorphic to the boundary of every regular neighborhood of $X_0$ in $X$
   and $\rho$ is a retraction.}
\end{quote}

One may extract from this description the following definition:

\begin{definition}   \label{def:acampodef}
     Let $X$ be a complex manifold and $C$ a reduced divisor with simple normal crossings inside $X$.
     Denote by $(C_j)_j$ the irreducible components of $C$.
     The {\bf real oriented blowup}
           $\boxed{\tau^{\rob}_{X, C}} : \boxed{X^{\rob}_C} \to X$
     {\bf of $X$ with center $C$} is the fiber
    product of the real oriented blowups $\tau^{\rob}_{X, C_j} : X^{\rob}_{C_j} \to X$ of $X$
    with centers $C_j$.
\end{definition}

One sees that this construction needs the preliminary definition of a real oriented blowup
in the case of {\em smooth} divisors. Even if A'Campo does not define it formally,
one sees from the rest of his paper
that he describes it in local charts as a pullback of the real oriented blowup $\tau_{\CC}$
by one of the canonical projections $\CC^n_{\underline{z}} \to \CC_{z_j}$  of formula
(\ref{eq:changepolarcomplexmap}), that is, of the passage to polar coordinates in one
complex variable. The fact that one gets the same result as in \autoref{def:robuCn} is
guaranteed by the following statement, whose proof is left to the reader:

\begin{proposition}   \label{prop:altrobtoric}
     For every $j \in \{1, \dots, n\}$, denote by $\pi_j : Z_j \to \CC^n$ the map obtained
     by replacing the $j$-th component of the identity $\CC^n \to \CC^n$ by the real
     oriented blowup $\tau_{\CC}: \R_{\geq 0}  \times \bS^1  \to    \CC$. Then, the fiber
     product of the maps $(\pi_j)_{j \in \{1, \dots, n\}}$ identifies canonically with the
     real oriented blowup $\tau_{\tv^{\N^n}} : (\R_{\geq 0} \times \bS^1)^n \to \CC^n$
     of $\CC^n$ seen as the toric variety $\tv^{\N^n}$.
\end{proposition}

\begin{remark}
  Gillam \cite{Gi 11} presented differently the real oriented blowups of complex manifolds
  along simple normal crossing divisors. This alternative perspective is described in a nutshell in
 \cite[Section 8.2]{ACGHOSS 13}.
\end{remark}

\medskip
\section{\bf From polar coordinates to the rounding of complex log spaces}  \label{sec:poltoround}

In this section we present the logarithmic viewpoint on real oriented blowups: the
operation of {\em rounding} of a complex log structure introduced by Kato and Nakayama
in \cite{KN 99}, also called {\em Betti realization} or {\em Kato-Nakayama morphism}.

We give a first example of a {\em complex log structure} by looking at
the passage to polar coordinates in a more intrinsic way than in \autoref{ssec:polcoord}
(see \autoref{ssec:polcoordintr}). Then, we formulate the general definition of a {\em log structure} on a ringed space, giving as a main class of examples the {\em divisorial log structures}
on complex analytic spaces, which are induced by a reduced divisor (see \autoref{ssec:defdivlog}).
A space endowed with a log structure is called a {\em log space}.
We explain what are morphisms between log spaces and how to pull back a log structure
by a morphism of complex spaces (see \autoref{ssec:categlog}).
This pullback operation uses a notion weaker than that of a log structure,  namely, that of a
{\em prelog structure} (see \autoref{ssec:prelogassoc}).
Looking at the divisorial log structure induced by the
algebro-geometric boundary of a toric variety, we explain the notion of a {\em chart} of a log structure,
useful in order to introduce convenient finiteness properties of log structures
(see \autoref{ssec:chartsls}). Then, we introduce the operation of {\em rounding of
complex log structures} (see \autoref{ssec:KNrounding})
and we reinterpret it using a convenient kind of {\em log point} (see \autoref{ssec:reintrounding}).
Passing to the {\em rounding of log morphisms}, we state an important local triviality theorem of Nakayama
and Ogus and we explain an application of it to the study of Milnor fibrations of
smoothings of singularities
(see \autoref{ssec:specaseNO}). Finally, we give a synoptic view of the way rounding may be used in singularity theory (see \autoref{ssec:approundsing}).

Kato's foundational paper \cite{K 88} in logarithmic geometry
develops logarithmic structures in the category
of schemes, following ideas of Fontaine and Illusie (see also~\cite[Definition III.1.1.1]{O 18}).
Logarithmic structures in the complex analytic setting are discussed in~\cite[Section 1]{KN 99},
the paper in which the operation of rounding was first introduced. Nakayama and Ogus'
article \cite{NO 10} studies rounding of morphisms of complex log spaces and proves
the most general local triviality theorem known until now for roundings of special types
of log morphisms. In \cite{AO 20}, Achinger and Ogus study the topology of degenerating families
of complex manifolds using the tools and results of  \cite{NO 10}.
Ogus' book \cite{O 18} is the only existing textbook about log geometry.
Shorter introductions to it may be found in Illusie  \cite{I 94, I 02},
Gross \cite[Chapter 3]{G 11}, Abramovich, Chen,  Gillam,  Huang, Olsson, Satriano and Sun
\cite{ACGHOSS 13}, Arg\"uz \cite{A 20} and Cueto, the author and Stepanov \cite[Section 4]{CPPS 23}.

In order to understand this chapter, the reader should be confortable with the language of sheaves.
If a reader needs a reminder of it, we recommend Eisenbud and Harris \cite[Section I.1.3]{EH 00}
and Tu \cite[Lecture 1]{T 22}. If $\cS$ is a sheaf of sets on a topological space $X$, $U$ is an open
subset of $X$ and $x$ is a point of $X$, we denote by $\boxed{\cS(U)}$ the set of sections of $\cS$
on $U$ and by $\boxed{\cS_x}$ the stalk of $\cS$ at $x$.

\medskip
\subsection{A logarithmic way to make the classical passage to polar coordinates intrinsic} $\:$  \label{ssec:polcoordintr}
\medskip

In this subsection, which follows the presentation of \cite[Section 4.2]{CPPS 23},
we revisit the prototypical real oriented blowup map
$ \tau_{\CC} :    \R_{\geq 0}  \times \bS^1  \to    \CC $ examined in \autoref{ssec:polcoord}
(see \autoref{def:orbu}),
redefining it more intrinsically,  in terms of the local ring $\cO_{\CC, 0}$ of $\CC$ at $0$.
This will allow us to give an intrisic definition of the real oriented blowup of a Riemann surface
at a point, using what will turn out to be our first example of a {\em logarithmic structure}
(see \autoref{def:robRiemsurf}).
\medskip

Let us consider again the classical passage to polar coordinates from \autoref{ssec:polcoord}.
There we saw the set $\CC$ both as a topological space -- the {\em complex plane} --
and as a field, identifying these two viewpoints via the coordinate
function $z : \CC \to \CC$.  Instead of being thought simply as an identity map, $z$ is
to be considered here as a map from a topological space to a field.
The function $z$ was crucial to define
the real oriented blowup function $ \tau_{\CC} :    \R_{\geq 0}  \times \bS^1  \to    \CC $
of formula (\ref{eq:changepolarcomplexmap})
and to see it both as a morphism in the category of topological spaces and in that of monoids.

Now we want to give a more intrinsic, {\em coordinate-free} definition of the map $\tau_{\CC}$,
without using the special function $z$. For this reason, we will see the complex plane $\CC$
as a Riemann surface, endowed with its sheaf $\boxed{\cO_{\CC}}$ of holomorphic functions.
We denote by $ \boxed{\cO_{\CC, 0}}$ its stalk at $0$, that is,
the local ring of germs of holomorphic functions at the origin of $\CC$.

Consider a germ $h \in \cO_{\CC, 0} \setminus \{0\}$. Post-composing it with the function
$\mbox{sign} : \CC^* \to \bS^1$ introduced in \autoref{def:logCoordinatesPolar}, one gets a
germ at $0$ of $\bS^1$-valued function  $\boxed{\mbox{sign} (h)}$ defined in a pointed neighborhood of
$0$ in $\CC$. The germ $h$ may be  written in a unique way as a product
      \[ h=z^m\cdot v\]
 for some $m \in \N$ and $v \in \cO_{\CC, 0}^{\star}$, where $\boxed{\cO_{\CC, 0}^{\star}}$
 denotes the edge of the monoid  $(\cO_{\CC, 0}, \cdot)$ in the sense of \autoref{def:intmon}.
Thus, the following  relation holds in a sufficiently small punctured neighborhood $\D_r \setminus \{0\}$
of $0$ in $\CC$, where $\boxed{\D_r}$ denotes the disc of radius $r > 0$ centered
at the origin of $\CC$:
     \begin{equation}   \label{eq:argh}
          \mbox{sign}(h) = \frac{h}{|  h |} =  \left(\frac{z}{|  z |} \right)^m \cdot   \frac{v}{|  v |} =
        \mbox{sign}(z)^m \cdot \mbox{sign}(v).
    \end{equation}
As a consequence of the fact that the pullback function
    \[ \begin{array}{cccc}
               \tau_{\CC}^*( \mbox{sign}(z)) : &  \R_{> 0} \times \bS^1 &   \to     &   \bS^1  \\
                             & (r, u)  &  \mapsto & u
        \end{array} \]
extends by continuity to the cylinder $\R_{\geq 0} \times \bS^1$, formula  \eqref{eq:argh}
shows that the  lift $\tau_{\CC}^* (\mbox{sign}(h)) : (0, r) \times \bS^1 \to \bS^1$
extends by continuity to $[0, r) \times \bS^1$.
By abusing notation, we denote this extension also by  $\tau_{\CC}^* (\mbox{sign}(h))$:
         \begin{equation*}   \label{eq:liftargholonevar}
                           \xymatrix{
                               [0, r) \times \bS^1
                                  \ar[rrd]^{\:\:\:  \tau_{\CC}^* (\mbox{sign}(h))}
                                  \ar[d]_{\tau_{\CC} }&
                                                & \\
                                  \CC \ar@{-->}[rr]_{\mbox{sign}(h)}  &  & \bS^1. }
           \end{equation*}

If $h_1, h_2 \in \cO_{\CC, 0} \setminus \{0\}$, then on any punctured neighborhood of the origin
on which they are both defined and non-zero, we have:
   \[  \mbox{sign}(h_1) \cdot  \mbox{sign}(h_2) =  \mbox{sign}(h_1 \cdot h_2). \]
 As a consequence, the relation
     \[ \tau_{\CC}^*( \mbox{sign}(h_1)) \cdot \tau_{\CC}^*( \mbox{sign}(h_2)) =
        \tau_{\CC}^*(\mbox{sign}(h_1 \cdot h_2)) \]
 is true over a neighborhood of the topological boundary
 $\bS^1= \partial_{\topo}(\R_{\geq 0} \times \bS^1)$ of the cylinder $\R_{\geq 0} \times \bS^1$.
We get:

 \begin{proposition}\label{pr:monoidsFromPoints}
   Consider a point $P \in \bS^1= \partial_{\topo}(\R_{\geq 0} \times \bS^1)$.
    Then, the map
      \[  \begin{array}{ccc}
              (\cO_{\CC, 0} \setminus \{0\}, \cdot) &   \to   & (\bS^1, \cdot) \\
                  h   &  \mapsto & \tau_{\CC}^* (\mathrm{sign}(h))(P)
           \end{array} \]
     is a \emph{morphism of multiplicative monoids}
     extending the standard morphism of groups $(\cO_{\CC, 0}^{\star}, \cdot)  \to    (\bS^1, \cdot)$
     defined by  $h\mapsto   \mathrm{sign}(h(0))$. Moreover, one gets in this way an isomorphism
     from the group $\bS^1$ to the group of such morphisms.
 \end{proposition}

 That is, each point of the topological boundary of the real oriented blowup
    $\R_{\geq 0} \times \bS^1$ of $\CC$ at $0$ may be seen as a morphism of a special
    type from the  monoid
    $(\cO_{\CC, 0} \setminus \{0\}, \cdot)$ to  the group $(\bS^1, \cdot)$:
    {\em it extends the standard morphism of groups $(\cO_{\CC, 0}^{\star}, \cdot)  \to    (\bS^1, \cdot)$
    defined by $h\mapsto   \mathrm{sign}(h(0))$.}

    Note also that the monoid $(\cO_{\CC, 0} \setminus \{0\}, \cdot)$ may be seen as the stalk
    at the origin $0$ of the complex plane $\CC$
    of the sheaf of submonoids of the (multiplicative) structure sheaf $(\cO_{\CC}, \cdot)$,
    consisting of the holomorphic functions {\em which do not vanish outside $0$}.
   This interpretation yields the promised intrinsic, {\em coordinate-free},
   extension of the  map $\tau_{\CC}$ to any Riemann surface $S$ and point $s$ of it:

  \begin{definition}   \label{def:robRiemsurf}
       Let $S$ be a Riemann surface and $s$ be a point on it.
       Denote by $\boxed{\cO_S^{\star}(-s)}$ the subsheaf of monoids of $(\cO_S, \cdot)$ consisting
       of the holomorphic functions which do not vanish outside $s$. Define
       the {\bf real oriented blowup of $S$ at $s$} by:
          \[  \boxed{S_s^{\rob}} := \{(x,u), x \in S,    u \in \mathrm{Hom}(\cO_{S,x}^{\star}(-s), \bS^1),
                  u(f) = \mbox{sign}(f(x)) \  \mbox{ for any } f \in \cO_{S,x}^{\star} \} \]
        where the morphisms are taken in the category of monoids.
       Define the associated {\bf real oriented blowup morphism} by:
          \[  \begin{array}{cccc}
                      \boxed{\tau_{S,s}} : &  S_s^{\rob}  &  \to &  S  \\
                         & (x,u)   & \mapsto & x
               \end{array}  .  \]
  \end{definition}

  Note that if the point $x \in S$ is different from $s$, then the stalk $\cO_{S,x}^{\star}(-s)$
  is equal to $\cO_{S,x}^{\star}$. Therefore, there is a unique point in the fiber
  $\tau_{S,s}^{-1}(x)$, the morphism $\mbox{sign} \in \mathrm{Hom}(\cO_{S,x}^{\star}(-s), \bS^1)$.
  This shows that the morphism $\tau_{S,s}$ is, as expected, a bijection above $S \setminus \{s\}$.
  This bijection is in fact a homeomorphism relative to a natural topology on $S_s^{\rob}$ which
  will be introduced in \autoref{def:rounding}.

 \autoref{def:robRiemsurf} uses a special subsheaf of monoids $\cO_S^{\star}(-s)$
  of the structure sheaf $\cO_S$ of the ambient Riemann surface $S$, seen as a
  sheaf of multiplicative monoids.
  This sheaf of monoids has the  property that {\em the inclusion morphism of sheaves
        \[ \cO_S^{\star}(-s)  \to  \cO_S \]
  realizes an isomorphism of the corresponding sheaves of edges}, in the sense of
  \autoref{def:intmon}. This is precisely the defining property of a {\em logarithmic structure}
  in the sense of Fontaine and Illusie.  It is now the time to discuss this notion.

   \medskip
\subsection{The definitions of log spaces and of divisorial log spaces}
 \label{ssec:defdivlog}  $\:$
\medskip

In this subsection we explain what is a {\em log structure} on a ringed space,
turning this space into a {\em log space} (see \autoref{def:logspace}).
We explain that any ringed space has two canonical log structures, an {\em empty} one
and a {\em trivial} one (see \autoref{def:tauttriv}).
We pass then to the complex analytic setting, defining the
{\em divisorial} log structures (see \autoref{def:divlogstruct}), with the special cases of
{\em toroidal} and {\em toric} log structures (see \autoref{def:logtoricgen}).
\medskip

  At the end of \autoref{ssec:polcoordintr} we associated to each couple $(S,s)$ consisting
  of a Riemann surface and a point on it a morphism
  $ \cO_S^{\star}(-s)  \to  \cO_S$ of sheaves of monoids  (the inclusion morphism)
  {\em which realizes an isomorphism of the corresponding sheaves of edges}
  (see \autoref{def:robRiemsurf}).
   Generalizing this property to arbitrary ringed spaces, we get the definitions of a {\em log space}
   as a ringed space endowed with a {\em log structure}:

   \begin{definition}   \label{def:logspace}
        A \textbf{logarithmic space} $\boxed{W}$ or a {\bf log space} for short is a  ringed
        space $(|W|, \cO_W)$,  where $\boxed{|W|}$ is a topological space
        called the \textbf{underlying topological space} of the logarithmic space, endowed with a
        sheaf of monoids $\boxed{\cM_W}$ and an {\bf evaluation morphism} of sheaves of monoids
        \[ \boxed{\alpha_W} \colon \cM_W \to (\cO_{W}, \cdot) \]
        which restricts to an isomorphism between their subsheaves of edges $(\cM_W^{\star}, \cdot)$ and
        $(\cO_W^{\star}, \cdot)$. The pair $(\cM_W, \alpha_W)$ is called a \textbf{logarithmic
        structure} on the ringed space $W,$ or a \textbf{log structure}
        for short. The log space $W$ and its log structure are called \textbf{complex}
        if the structure sheaf $\cO_W$ is a sheaf of complex algebras.
    \end{definition}

   \begin{remark}\label{rem:logNotation}
          Fontaine and Illusie's main motivations for introducing the notion of
           a log space (in the context of schemes) are described in~\cite{I 94}.
           The terminology refers  to the fact that
           a log structure gives rise to a canonical notion of \emph{sheaf of
           differential forms with logarithmic poles}.
           The term ``logarithmic'' hints also to the fact that if the composition law in $\cM_W$
           is thought additively, then  $\alpha_W$ becomes an exponential map turning
           sums into products.
     \end{remark}

     \begin{remark}  \label{rem:dagger}
         We took the notation ``$|W|$'' for the
        underlying topological space of a ringed space from Eisenbud and Harris \cite[Section I.2]{EH 00}.
        Gross used the notation ``$W^{\dagger}$'' for a log space with underlying ringed space $W$
        in his book \cite{G 11} in which he surveyed the Gross-Siebert program
        for studying mirror symmetry with  log geometrical techniques. This convention is used
        in part of the literature. We will use it only for divisorial log spaces (see \autoref{def:divlogstruct}
        and \autoref{def:logtoricgen}), in \autoref{def:logenhancement} and in \autoref{rem:meanstrict}.
        The term {\em evaluation morphism} seems to be new, but
        the notation ``$\alpha_W$'' for it is widely used and
        already appears in Kato's foundational paper \cite{K 88}.
    \end{remark}

     Every ringed space can be endowed with two canonical log structures (see \cite[Page 271]{O 18}):

    \begin{definition} \label{def:tauttriv}
        Let  $(W, \cO_W)$ be a ringed space. Its \textbf{empty log structure}
        is given by the identity  morphism  $\cO_W  \hookrightarrow \cO_W$ and its
        \textbf{trivial log structure} by the embedding $\cO_W^{\star} \hookrightarrow \cO_W$.
    \end{definition}

 \begin{remark}   \label{rem:emptylog}
      Assume that $(W, \cO_W)$ is a complex analytic variety. The construction of divisorial log
      structures in \autoref{def:divlogstruct} below may be performed starting from any closed
      subvariety $D \hookrightarrow W$, instead of a divisor. If one starts from $D := W$, then one
      gets the log structure $\cM_W := \cO_W \hookrightarrow \cO_W$. In this case,
      $W \smallsetminus D = \emptyset$, which motivated Ogus to call this log structure
      {\em empty}, as he kindly informed the author in a message from August 12, 2024.
 \end{remark}

\begin{remark}  \label{rem:logCatFixedV}
   {\em Log  structures on a fixed ringed space $W$ form a category}.
   More precisely,  one may define a morphism $(\cM_W,\alpha_W) \to (\cN_W, \beta_W)$
   as a morphism of sheaves of monoids
   $\varphi \colon \cN_W \to \cM_W$ compatible with the evaluation morphisms
   $\alpha_W$ and $\beta_W$, i.e., satisfying the relation $\beta_W = \alpha_W \circ  \varphi$.
    The empty log structure is then the initial object in this category,
    whereas  the trivial log structure is its final object. Our apparently strange convention
    of making $\phi$ and $\varphi$ go in opposite directions, which is contrary to that
    of \cite[Definition III.1.1.1]{O 18}, is chosen in order to make this definition compatible
    with \autoref{def:morphprelog}
    of morphisms of log spaces. Indeed, according to that definition, a morphism from the log space
    $(W, \cM_W,\alpha_W)$ to the log space $(W, \cN_W,\beta_W)$ is a pair of morphisms
        \[(\phi \colon W \to W,   \:  \phi^{\flat}\colon \phi^{-1} (\cN_W) \to \cM_W),\]
    such that the following diagram is commutative:
        \[
           \xymatrix{
              \phi^{-1} (\cN_W)  \ar[r]^-{\phi^{\flat}} \ar[d]_{\phi^{-1}\beta_W}
                           & \cM_W \ar[d]^{\alpha_W} \\
               \phi^{-1} (\cO_W)   \ar[r]_{\phi^{\sharp}}
                          & \cO_W.}
         \]
    When $\phi$ is the identity morphism of $W$, this means that
    $\phi^{\flat}\colon \cN_W  \to \cM_W$ makes the following diagram commutative:
       \[
           \xymatrix{
              \cN_W  \ar[r]^-{\phi^{\flat}} \ar[d]_{\beta_W}
                           & \cM_W \ar[d]^{\alpha_W} \\
               \cO_W   \ar[r]_{\mathrm{id}}
                          & \cO_W.}
         \]
      That is, $\beta_W = \alpha_W \circ \phi^{\flat}$, which shows that $\phi^{\flat}$ is precisely
      a morphism from the log structure $(\cM_W,\alpha_W)$ to the log structure
      $(\cN_W,\beta_W)$ in the sense explained at the beginning of this remark.
\end{remark}

  In the same way as a point of a Riemann surface $S$ defines a complex log structure on $S$, an arbitrary
  effective divisor on a complex analytic variety $W$ also defines a complex log structure on $W$:

   \begin{definition}  \label{def:divlogstruct}
         Let $D$ be an effective divisor on a complex analytic variety $W$.
         The {\bf divisorial log structure}
                 $ (\boxed{\cO_W^{\star}(-  D)}, \cO_W^{\star}(-  D) \hookrightarrow \cO_W)$
         {\bf defined by $D$} consists of the holomorphic functions which vanish {\em at most} along
         the support $|D|$ of $D$.
         That is, if $U$ is an open subset of $W$, then $\cO_W^{\star}(-D)(U)$
         is the submonoid of the multiplicative monoid $(\cO_W(U), \cdot)$ consisting of the
         holomorphic functions defined on $U$ {\em which do not vanish outside $|D|$}.
         We denote by $\boxed{W_D^{\dagger}}$ the {\bf divisorial log space}
         $(W, \cO_W^{\star}(-  D), \cO_W^{\star}(-  D) \hookrightarrow \cO_W)$ {\bf defined by $D$}.
    \end{definition}

\begin{remark}\label{rem:NotationsDivLogStr}
   The notation ``$\cO_W^{\star}(-  D)$'' is not standard. We chose it in \cite{CPPS 23} by analogy
   with the classical notation
    ``$\cO_W(-D)$'' for the sheaf of ideals of holomorphic functions on $W$
    vanishing {\em at least} along $D$.
    Other notations used in the literature are ``$\cM_{(W \setminus D) | V}$''
    (see \cite[Section III.1.6]{O 18}) and ``$\cM_{(W,D)}$'' (see \cite[Example 3.8]{G 11}
    or \cite[Example 1.6]{A 21}).  Unlike what happens to the sheaf
    $\cO_W(-D)$, one has $\cO_W^{\star}(-  D) = \cO_W^{\star}(-  |D|)$,
    therefore it is enough to consider reduced divisors $D$ in \autoref{def:divlogstruct}.
    Note that by contrast with $\cO_W(-D)(U)$ which is an
    ideal of $\cO_W(U)$, therefore is closed under addition, this is
    never the case for $\cO_W^{\star}(-  D)(U)$. Indeed, for every $f \in \cO_W^{\star}(-  D)(U)$,
    one has also $-f \in \cO_W^{\star}(-  D)(U)$, but
    $0 = f + (-f)$ is never an element of $\cO_W^{\star}(-  D)(U)$.
\end{remark}

\begin{remark}   \label{rem:compactif}
    Divisorial log structures are sometimes called {\em compactifying log structures} (see for instance
    \cite[Section III.1.6]{O 18}). This terminology refers to the fact that the pairs $(W, D)$ with {\em compact}
    $W$ may be seen as {\em compactifications} of $W \setminus D$.
\end{remark}

Toroidal varieties (see~\autoref{def:toroidal}) and in particular toric varieties can be equipped with canonical divisorial log structures as follows:

   \begin{definition}  \label{def:logtoricgen}
        A \textbf{toroidal log structure} is a divisorial log structure of the form
        $\cO_W^{\star}(- \partial W)$, where $(W, \partial W)$ is a toroidal variety.
        A variety endowed with a toroidal log
        structure is called a \textbf{toroidal log variety}. If $W$ is a toroidal variety, we denote by
        $\boxed{W^{\dagger}}$ the associated toroidal log variety.  In particular,
        if $\hat{\cF}$ is a fan of monoids in the sense of \autoref{def:fanmonoids}, we denote by
        $\boxed{\tv_{\hat{\cF}}^{\dagger}}$ the
        toroidal log variety associated to the toric variety $\tv_{\hat{\cF}}$
        (see \autoref{def:toricfromfanmon}).
   \end{definition}

   {\em What is the advantage of looking at a toric or a toroidal variety as a log variety?}
   The main reason is flexibility. For instance, we may {\em restrict} the ambient log structure
   to the toroidal or toric boundary, getting a new log variety, which is in general
   no longer toroidal or toric. Therefore, this operation is impossible
   inside the toric or toroidal category. For instance, if the boundary has a point at which it is
   analytically reducible (which occurs already in the simple example of the affine toric surface
   $\CC^2$), then the boundary is not even a toroidal variety, which is by construction locally
   irreducible.

   \begin{remark}  \label{rem:restruse}
         Assume that $D$ is a {\em smooth} compact divisor inside a complex manifold $M$.
        It defines an associated line bundle $L_D$ endowed with a section whose divisor is $D$.
        If one {\em restricts} $L_D$ to $D$, then one gets a line bundle isomorphic to the normal
        bundle of $D$ in $M$. Therefore, the total space $S_D$ of its associated circle
        bundle is isomorphic to the topological boundaries of the tubular neighborhoods of $D$ in $M$.

       If $D$ is now a {\em singular} simple normal crossings divisor in $M$, then one has again an
       associated line bundle $L_D$, but now the total space $S_D$ of the associated circle bundle
       of its restriction to $D$ is no longer isomorphic to the topological
       boundaries of the tubular neighborhoods of
       $D$ in $M$. Indeed, $S_D$ is now singular above the singular locus of $D$,
       unlike the previous boundary.

       It turns out that one may nevertheless obtain the boundaries  of the tubular neighborhoods of
       $D$ in $M$ from a restriction to $D$ of a structure induced from $D$ on $M$, namely, from
       the associated divisorial log structure in the sense of \autoref{def:divlogstruct}. This is why
       the operation of restriction of log structures is important in the study of real oriented blowups.

       There is a slightly different notion of a log structure associated to a simple normal crossings
       divisor in a complex manifold, called a {\em DF log structure}. The name alludes to
       Deligne and Faltings, who introduced this notion more or less at the same time when
       Fontaine and Illusie imagined their notion of log structure. For details about
       DF log structures, one may consult \cite[Complement 1]{K 88} and \cite[Section III.1.7]{O 18}.
   \end{remark}

   The operation of {\em restriction} of a log structure to a subvariety is a special case of
   the operation of {\em pull back} of a log structure by a {\em morphism of ringed spaces}.
   Subsections \ref{ssec:prelogassoc} and \ref{ssec:categlog} are
   dedicated to this operation. In particular, we will understand why the morphisms $\alpha_W$
   appearing in \autoref{def:logspace} of log structures are not required to be {\em injective},
  by contrast with the empty or trivial log structures of
   \autoref{def:tauttriv} and the divisorial log structures of \autoref{def:divlogstruct},
   therefore also by contrast with the toroidal and toric log structures of \autoref{def:logtoricgen}.
   Indeed, the evaluation morphisms in the sense of
   \autoref{def:logspace} of log structures obtained by pullbacks
   are not necessarily injective (see \autoref{rem:notinj}).

  \medskip
\subsection{Prelog structures and their associated log structures} \label{ssec:prelogassoc}  $\:$
\medskip

 In this subsection, we examine heuristically how to {\em restrict log structures to subspaces}.
 We do this by looking at the example of the divisorial log structure defined
 by the origin of the complex plane $\CC$,
 considered in  \autoref{ssec:polcoordintr} (see \autoref{def:robRiemsurf}). This will lead us
 to the definition of a {\em prelog structure} (see \autoref{def:prelog}) and of
 {\em log structure associated to a
 prelog structure} (see \autoref{def:standardLogStructure}). We will end the section with the
 definitions of {\em log points} and of several particular kinds of log points
 (see Definitions \ref{def:logpoint} and \ref{def:varlogpoints}).
\medskip

Let us consider again our first example of log space, namely, that of a Riemann surface $S$
endowed with the divisorial log structure $\cO_S^{\star}(-s) \hookrightarrow \cO_S$
associated to a point $s \in S$ (see \autoref{def:robRiemsurf}). Our aim is to {\em restrict}
in a suitable sense this log structure to the point $s$. For simplicity, we consider again the special case
$S := \CC$ and $s := 0_{\CC}$ in which we explained the passage
to polar coordinates using monoids (see \autoref{ssec:polcoordintr}).  We denote by
  $\boxed{0_{\CC}}$ the origin of the complex plane $\CC$, stressing that it is
  thought of as a divisor and not as a number. For instance, we have $  - 0_{\CC} \neq 0_{\CC}$.

If we simply restrict the morphism of sheaves of monoids
$\cO_{\CC}^{\star}(-0_{\CC}) \hookrightarrow \cO_{\CC}$
to the point $0_{\CC}$, we get the associated morphism of stalks
   \begin{equation*} \label{eq:stalksRS}
         \boxed{\alpha_{\CC, 0_{\CC}}} :  \cO_{\CC, 0_{\CC}}^{\star}(-0_{\CC})
               \hookrightarrow \cO_{\CC,0_{\CC}}.
    \end{equation*}
  This is not a log structure on the point $0_{\CC}$, because the target $\cO_{\CC,0_{\CC}}$ is not
  the structure sheaf $\cO_{0_{\CC}}$ of $0_{\CC}$. Indeed, the structure sheaf $\cO_{0_{\CC}}$ consists
  of a single $\CC$-algebra,
  that of holomorphic functions on the single non-empty open set $\{0_{\CC}\}$ of $0_{\CC}$. This
  $\CC$-algebra is simply the field $\CC$. One has therefore a canonical surjective morphism of rings
     \begin{equation} \label{eq:rhorestrpoint}
         \boxed{\rho_{0_{\CC}}} :   \cO_{\CC,0_{\CC}}  \to \CC
    \end{equation}
   which restricts germs of holomorphic functions on $\CC$ at $0_{\CC}$
   to the point $0_{\CC}$: it consists
   in evaluating at the point $0_{\CC}$ a germ at $0_{\CC}$ of function defined on a neighborhood
   of $0_{\CC}$ in $\CC$.
   By composing the morphisms $\alpha_{\CC,0_{\CC}}$ and $\rho_{0_{\CC}}$ of \eqref{eq:stalksRS}
   and \eqref{eq:rhorestrpoint},  we get a morphism of multiplicative monoids
      \begin{equation} \label{eq:restrpoint}
         \boxed{\alpha_{0_{\CC}}} := \rho_{0_{\CC}}  \circ \alpha_{\CC,0_{\CC}} :
              \cO_{\CC,0_{\CC}}^{\star}(-0_{\CC})  \to \CC.
    \end{equation}
    This is still not a log structure, as it does not restrict to an isomorphism
   of their edges (or subgroups of invertible elements, see \autoref{def:intmon}). Indeed, the edge of
   the monoid $(\cO_{\CC, 0_{\CC}}^{\star}(- 0_{\CC}), \cdot)$ is
   the group $(\cO_{\CC, 0_{\CC}}^{\star}, \cdot)$ of germs at $0_{\CC}$ of holomorphic functions
   which do not vanish at $0_{\CC}$
   and the edge of $(\CC, \cdot)$ is simply the group $(\CC^*, \cdot)$.

   \medskip
   This phenomenon appears every time one wants to restrict a log structure to a subspace. It
   motivates the introduction of a special name for morphisms of sheaves of monoids which do not
   necessarily identify the subgroups of edges, as asked in \autoref{def:logspace} of log structures:

   \begin{definition}   \label{def:prelog}
        A \textbf{prelogarithmic space} $W$  is a  ringed space $(|W|, \cO_W)$,
        called the {\bf underlying ringed space} of the prelogarithmic space,
        endowed with a sheaf of monoids $\cM_W$ and a morphism of sheaves of monoids
                 \[ \alpha_W \colon \cM_W \to (\cO_W, \cdot). \]
        The pair $(\cM_W, \alpha_W)$ is called a \textbf{prelogarithmic
        structure} on $W$, or \textbf{prelog structure} for short.
        The prelogarithmic space $W$ is called
        \textbf{complex} if the structure sheaf $\cO_W$ is a sheaf of complex algebras.
    \end{definition}

 Note that log structures are particular prelog structures, as sheaves are
 particular presheaves. This explains the use of the prefix ``{\em pre}''.
 We have a natural inclusion functor
\begin{equation}\label{eq:forgetful}
     \boxed{\iota}   \colon \{\text{log structures on } W\} \to \{\text{prelog structures on } W\}
\end{equation}
   from the category of log structures on $W$ to that of prelog structures on $W$. This last category is
   defined analogously to that of log structures on $W$ (see \autoref{rem:logCatFixedV}).

   There is a natural way to transform a prelog structure into a log structure, parallel to the
   passage from a presheaf to a sheaf. A category-theoretic way to describe this transformation
   is to say that {\em the functor  $\iota$ of formula \eqref{eq:forgetful} admits a left adjoint $a$}
   (see \cite[Proposition III.1.1.3]{O 18}). More precisely, given a prelog structure $(\cM_W,\alpha_W)$
   on $W$, its image $ \boxed{\cM_W^a} $ under the functor $a$
   (``$a$'' being the initial of ``{\em associated}'', see \autoref{def:standardLogStructure} below)
   is the pushout of a diagram of sheaves on the topological space $|W|$:
\begin{equation}   \label{eq:pushoutprelog}
\begin{tikzcd}
  \alpha_W^{-1}(\cO^{\star}_W) \arrow[r] \arrow[d, "\alpha_W |" '] &  \cM_W\\
        \cO_W^{\star} &
  \end{tikzcd}
         \leadsto
  \begin{tikzcd}
  \alpha_W^{-1}(\cO^{\star}_W) \arrow[r] \arrow[d,  "\alpha_W |" ']
         \arrow[dr, phantom, "\lrcorner", very near end]
  &  \cM_W   \arrow[d] \\
        \cO_W^{\star}   \arrow[r]    &      \cM_W^a
  \end{tikzcd}
\end{equation}
where $\boxed{\alpha_W^{-1}(\cO_W^{\star})}$ is the inverse image sheaf of $\cO_W^{\star}$
by the morphism $\alpha_W$
and $\boxed{\alpha_W |}$ denotes the restriction
of the morphism $\alpha_W$ to $\alpha_W^{-1}(\cO_W^{\star})$.

The pushout sheaf $\cM_W^a$ is associated to the presheaf which attributes to each open
subset $U$ of $W$ the {\em pushout} of the following diagram of morphisms of monoids:
\[
\begin{tikzcd}
  \alpha_W^{-1}(\cO^{\star}_W)(U) \arrow[r] \arrow[d,  "\alpha_W |_U" '] &  \cM_W(U)\\
        \cO_W^{\star}(U) &
  \end{tikzcd}
\]
The operation of {\em pushout of morphisms of monoids} is an analog of that of {\em amalgamated sum} of morphisms of groups appearing
in the Seifert-Van Kampen theorem. One may find a concrete description of it using congruence
relations on monoids in \cite[Proposition I.1.1.5]{O 18} (see also \cite[Section (1.3)]{K 88}).

Note that the diagram on the left of formula \eqref{eq:pushoutprelog} is part of the
following commutative diagram:
\begin{equation}   \label{eq:partcomm}
  \begin{tikzcd}
  \alpha_W^{-1}(\cO^{\star}_W) \arrow[r] \arrow[d,  "\alpha_W |" ']
  &  \cM_W   \arrow[d,  "\alpha_W" ] \\
        \cO_W^{\star}   \arrow[r]    &      \cO_W
  \end{tikzcd}
\end{equation}

Therefore, by the universal property of the pushout, there is a unique morphism
$\boxed{\alpha_W^a} \colon \cM_W^a\to \cO_W$ of sheaves of monoids
making the following diagram commute:
\begin{equation}   \label{eq:compldiaggen}
\begin{tikzcd}
  \alpha_W^{-1}(\cO^{\star}_W) \arrow[r] \arrow[d,  "\alpha_W |" ']
         \arrow[dr, phantom, "\lrcorner", very near end]
  &  \cM_W   \arrow[d]   \arrow[ddr, "\alpha_W" ]  & \\
        \cO_W^{\star}   \arrow[r]   \arrow[drr] &      \cM_W^a \arrow[dr, "\alpha_W^a" description]     & \\
          &    &  \cO_W
  \end{tikzcd}
\end{equation}
The map $\boxed{\alpha_W^a}$  sends
$(s,t)$ to $\alpha_W(s)t \in \cO_W(U)$ for each $s\in \cM_W(U)$ and $t\in \cO^{\star}_W(U)$, where
$U$ is an open set in $W$.
One may show that the morphism $\alpha_W^a :  \cM_W^a \to  \cO_W$ is a log structure,
which is exactly the image of the prelog structure $\alpha_W :  \cM_W \to  \cO_W$ by
the functor $a$ left adjoint to $\iota$  (see  \cite[(1.3)]{K 88} and
\cite[Proposition III.1.1.3]{O 18} for details).

\begin{definition}\label{def:standardLogStructure}
     We say that the pair $(\cM_W^a,\alpha_W^a)$ of diagram \eqref{eq:compldiaggen}
     is the {\bf log structure on the ringed space $(|W|, \cO_W)$
     associated  to the prelog structure} $(\cM_W, \alpha_W)$.
\end{definition}
\smallskip

Consider for instance the prelog structure
    $\alpha_{0_{\CC}}  :   \cO_{\CC,0_{\CC}}^{\star}(-0_{\CC})  \to \CC$
  of formula \eqref{eq:restrpoint}.  The corresponding diagram \eqref{eq:partcomm} is:
  \begin{equation}  \label{eq:partcommspec}
  \begin{tikzcd}
       \cO^{\star}_{\CC, 0_{\CC}} \arrow[r] \arrow[d,  "\alpha_{0_{\CC}}  |" ']
        &   \cO^{\star}_{\CC, 0_{\CC}}(-0_{\CC})    \arrow[d,  "\alpha_{0_{\CC}}" ] \\
        \CC^*   \arrow[r]    &      \CC
  \end{tikzcd}
             =\joinrel=
      \begin{tikzcd}
        \CC\{z\}^{\star} \arrow[r] \arrow[d,  "\alpha_{0_{\CC}}  |" ']
        &    \CC\{z\} \setminus \{0\}  \arrow[d,  "\alpha_{0_{\CC}}" ] \\
        \CC^*   \arrow[r]    &      \CC
  \end{tikzcd}
\end{equation}
  and the corresponding diagram \eqref{eq:compldiaggen} is:
\begin{equation*}   \label{eq:compldiagpart}
\begin{tikzcd}
     \cO^{\star}_{\CC, 0_{\CC}} \arrow[r] \arrow[d,  "\alpha_{0_{\CC}}  |" ']
         \arrow[dr, phantom, "\lrcorner", very near end]
  &  \cO^{\star}_{\CC, 0_{\CC}}(-0_{\CC})   \arrow[d]   \arrow[ddr, "\alpha_{0_{\CC}}" ]  & \\
        \CC^*   \arrow[r]   \arrow[drr] &      \CC^* \oplus \N \arrow[dr, "\alpha_{0_{\CC}}^a" description]     & \\
          &    &  \CC
  \end{tikzcd}
                             =\joinrel=
     \begin{tikzcd}
           \CC\{z\}^{\star} \arrow[r] \arrow[d,  "\alpha_{0_{\CC}}  |" ']
         \arrow[dr, phantom, "\lrcorner", very near end]
         & \CC\{z\}  \setminus \{0\} \arrow[d]   \arrow[ddr, "\alpha_{0_{\CC}}" ]  & \\
        \CC^*  \arrow[r]   \arrow[drr] &      \CC^* \oplus \N \arrow[dr, "\alpha_{0_{\CC}}^a" description]     & \\
          &    &  \CC
  \end{tikzcd}
\end{equation*}
    Let us describe the three morphisms of monoids which are interior to this diagram, that is,
    which form the complement of the diagram \eqref{eq:partcommspec}:
   \begin{equation}   \label{eq:newarrows}
       \xymatrix{
            &   & \CC\{z\}  \setminus \{0\}  \ar[d]^{z^m v \  \mapsto \  (v(0), m)}   & \\
           \CC^* \ar[rr]_{\lambda \  \mapsto \  (\lambda, 0)}  &  &
               \CC^* \oplus \N \ar[dr]^{(\lambda, m) \  \mapsto  \  \lambda \cdot \delta^m_0}     & \\
          &  &  &  \CC
           }
    \end{equation}
As in \autoref{ssec:polcoordintr}, we write the elements of the monoid $(\CC\{z\} \setminus \{0\}, \cdot)$
in the form $z^m v$, with $m \in \N$ and $v \in \CC\{z\}^{\star}$. By $\delta^m_0$ we denoted
{\em Kronecker's symbol}:
   \[ \boxed{\delta^m_0} := \left\{   \begin{array}{ccc}
                                                        1 & \mbox{if}  & m = 0, \\
                                                        0  & \mbox{if}  & m > 0.
                                           \end{array} \right.   \]
        The diagonal morphism of diagram \eqref{eq:newarrows} is a log structure,
        called the {\em standard complex log point} (see \autoref{def:varlogpoints}). It is by definition
        {\em the restriction to the point $0_{\CC}$
        of the log structure $\alpha_{\CC}: \cO_{\CC}^{\star}(-0_{\CC}) \hookrightarrow \cO_{\CC}$}.

       We arrived at our first example of {\em log point}, in the sense of \cite[Section III.1.5]{O 18}:

       \begin{definition}   \label{def:logpoint}
            Let $K$ be a field. A {\bf $K$-log point} is a log structure on the spectrum of $K$.
            Equivalently, it is a singleton topological space, endowed with a morphism
            of monoids $\Gamma \to (K, \cdot)$ which restricts to an isomorphism
            from the edge $\Gamma^{\star}$ of $\Gamma$ to the edge $K^*$ of $K$.
            If $K = \CC$, we speak about a {\bf complex log point}.
       \end{definition}

       In the sequel we will be interested in the following types of complex log points:

          \begin{definition}   \label{def:varlogpoints}
              One distinguishes the following kinds of complex log points, defined through
              their associated morphisms $\alpha: \Gamma \to (\CC, \cdot)$ of monoids:
                 \begin{itemize}
                      \item  The {\bf empty complex log point} $\boxed{\Pi^{\mathrm{\emptyset}}_{\CC}}$
                            is defined by the identity
                            $(\CC, \cdot) \to (\CC, \cdot)$.
                      \item  The {\bf trivial complex log point}  $\boxed{\Pi^{\mathrm{triv}}_{\CC}}$
                             is defined by the inclusion
                           $(\CC^{\star}, \cdot) \hookrightarrow (\CC, \cdot)$.
                      \item  The {\bf standard complex log point} $\boxed{\Pi^{\mathrm{std}}_{\CC}}$
                           is defined by:
                         \[   \begin{array}{ccc}
                                   (\CC^*, \cdot) \oplus (\N, +) & \to & (\CC, \cdot) \\
                                     (\lambda, m) & \mapsto   & \lambda  \cdot \delta^m_0
                               \end{array}   \]
                      \item  The {\bf polar log point} $\boxed{\Pi^{\mathrm{pol}}_{\CC}}$
                           is defined by the real oriented blowup morphism
                               \[
                                 \begin{array}{cccc}
                                     \tau_{\CC} : &  ( \R_{\geq 0}  \times \bS^1, \cdot) & \to  &  (\CC, \cdot)    \\
                                      &   (r, u)        & \mapsto  & r \cdot u
                                 \end{array}
                               \]
                            of formula \eqref{eq:changepolarcomplexmap}.
                 \end{itemize}
         \end{definition}

       The empty, trivial and standard log points may be defined over the spectrum of any field.
       By contrast, the polar log point is specific to the field $\CC$. For this reason, we do not
       use the attribute {\em complex} when referring to it.

       \begin{remark}  \label{rem:originames}
             The name {\em empty complex log point} is due to the fact that it is the spectrum of the field $\CC$
             endowed with the empty log structure in the sense of \autoref{def:tauttriv}. The name
             {\em polar log point} is taken from Arg\"uz \cite[Section 1.1]{A 21}, whereas the term
             {\em standard log point} is taken from \cite[Example 3.8(2)]{G 11}
             (see also \cite[Example III.1.5.2]{O 18}).
       \end{remark}

       \begin{remark}  \label{rem:newvpolcoord}
            We got a third, {\bf logarithmic aspect} of the real oriented blowup morphism $\tau_{\CC}$,
             in addition to the {\em topological} and {\em algebraic aspects}
             (see \autoref{rem:twoaspectsrobu}):  $\tau_{\CC}$ also describes a complex log structure on a
             point, turning it into the polar log point.
       \end{remark}

       \begin{remark}   \label{rem:notinj}
           Note that the evaluation morphism
           $ \alpha: (\CC^*, \cdot) \oplus (\N, +) \to (\CC, \cdot)$
           defining the standard complex log point $\Pi^{\mathrm{std}}_{\CC}$ is not injective. For instance:
              \[   \alpha^{-1}(0) =  (\CC^*, \cdot) \oplus (\Z_{> 0}, +). \]
           This explains why one does not ask evaluation morphisms to be injective
           when defining log structures: otherwise one could not, for instance, define
           their restrictions to subspaces of the underlying topological space.
       \end{remark}

   \medskip
   \subsection{Morphisms of log spaces and pullbacks of log structures} \label{ssec:categlog} $\:$
   \medskip

  In this subsection we define the operation of {\em pullback of a log structure} by a morphism of ringed
  spaces (see \autoref{def:pb}) and the {\em categories of prelog and log spaces}
  (see \autoref{def:morphprelog} and \autoref{def:morlogspaces}). An important type of
  log morphism is that of {\em log enhancement of a morphism of complex analytic spaces},
  associated to compatible pairs of divisors (see \autoref{def:logenhancement}).
  We conclude by introducing {\em strict morphisms}, which are the morphisms of log spaces
  corresponding to pullbacks (see \autoref{def:strict}).
  \medskip

  In \autoref{ssec:prelogassoc} we saw on the example of the divisorial log structure
  defined by the origin $0_{\CC}$ of the complex plane $\CC$ that the restriction of a
  log structure to a subspace is performed in two steps. The first one produces a
  prelog structure and the second one passes to the associated log structure, in the sense
  of \autoref{def:standardLogStructure}.
  The following definition describes this two step process in full generality, as an
  operation of {\em pullback} by a morphism of ringed spaces (see \cite[Section 1.4]{K 88}
    and \cite[Definition III.1.1.5]{O 18}):

   \begin{definition}   \label{def:pb}
    Let $\phi \colon V \to W$ be a {\bf morphism of ringed spaces}, that is, a pair
        \[(\boxed{|\phi|} \colon |V| \to |W|,   \
           \boxed{\phi^{\sharp}}  \colon |\phi|^{-1} (\cO_W) \to \cO_V) \]
      where $|\phi|$ is a continuous map and $\boxed{\phi^{-1}(\cS_W)}$ denotes the
      {\bf inverse image} of a sheaf $\cS_W$ on $|W|$ by $|\phi|$. This inverse image is the sheaf associated
      to the presheaf $U\mapsto \lim_{U'\supseteq f(U)} \cS_W(U')$ on $V$, where $U'\subset W$
      and $U\subset V$ are open subsets.
     Fix a log structure $(\cM_W,\alpha_W)$  on $W.$
     The \textbf{pullback\index{log!structure!pullback} $\boxed{\phi^* \cM_W}$  of $\cM_W$ by $\phi$}
      is the log structure on $V$ associated to the prelog structure obtained
      as the composition $\phi^{-1}(\cM_W) \xrightarrow{\phi^{-1}(\alpha_W)} \phi^{-1}(\cO_W)
      \xrightarrow{\phi^{\sharp}} \cO_V$.

        If $\phi\colon V\hookrightarrow W$ is a closed immersion of analytic spaces,
        we say that $\phi^*\cM_W$ is  \textbf{the restriction of $\cM_W$ to $V$}.
        We often denote it simply by $\boxed{\cM_{W | V}}$.
   \end{definition}

   There is also a notion of {\em pushforward} or {\em direct image} of log structures
   (see \cite[Section 1.4]{K 88} and \cite[Definition III.1.1.5]{O 18}), which we will not use in this text.

   In order to turn prelog and log spaces into objects of categories, {\em morphisms} between
   such spaces must be appropriately
   defined. We start with {\em morphisms between prelog spaces}, which are defined
   using inverse image sheaves (see \autoref{def:pb}):

  \begin{definition}  \label{def:morphprelog}
        A \textbf{morphism $\phi \colon V \to W$ of prelog spaces}  is a triple
      \[(\boxed{|\phi|} \colon |V| \to |W|,   \
           \boxed{\phi^{\sharp}}  \colon |\phi|^{-1} (\cO_W) \to \cO_V, \
               \boxed{\phi^{\flat}}    \colon |\phi|^{-1} (\cM_W) \to \cM_V),\]
   where $(|\phi|, \phi^{\sharp})$ is a morphism of ringed spaces and
   $\phi^{\flat}$ is a morphism of sheaves of monoids on $V,$ making the
    following diagram commute:
      \begin{equation*}\label{eq:logMorphisms}
           \xymatrix{
              |\phi|^{-1} (\cM_W)  \ar[r]^-{\phi^{\flat}} \ar[d]_{|\phi|^{-1}\alpha_W}
                           & \cM_V \ar[d]^{\alpha_V} \\
               |\phi|^{-1} (\cO_W)   \ar[r]_{\phi^{\sharp}}
                          & \cO_V.}
      \end{equation*}
   \end{definition}

   Next, we turn to the definition of {\em morphisms between log spaces}:

\begin{definition} \label{def:morlogspaces}
   A \textbf{morphism of log spaces}, or \textbf{log morphism} for short,
    is simply a morphism between the underlying prelog spaces.
   That is, the \textbf{category of log spaces}  is the full subcategory of
   the category of prelog spaces whose objects are log spaces.
\end{definition}

\begin{example}\label{ex:trivialStructures}
     If two ringed spaces $V$ and $W$ are endowed with their trivial log structures
   in the sense of~\autoref{def:tauttriv}, then a log morphism between the resulting log
   spaces is determined by the associated morphism of ringed spaces. Therefore, there
   is a functor from the category of ringed spaces to the category of log spaces,
   which at the level of objects sends each ringed space to the same ringed space endowed
   with the trivial log structure.
 \end{example}

Similarly to the way one associates a log structure to a reduced divisor on a complex variety
 (see \autoref{def:divlogstruct}), one may
 associate a log morphism to special kinds of morphisms between complex varieties endowed
 with reduced divisors (see \cite[Definition 4.36]{CPPS 23}):

 \begin{definition}   \label{def:logenhancement}
      Let $\varphi: V \to W$ be a morphism of complex analytic spaces. Let $E \hookrightarrow V$ and
      $D \hookrightarrow W$ be reduced divisors on them. Assume that $|\varphi|^{-1}(D) \subset E$.
      The log morphism $\boxed{\varphi^{\dagger}} : V_E^{\dagger} \to W_D^{\dagger}$ obtained by pulling
      back the sections of the sheaf of monoids $\cO^{\star}_W(-D)$ by the morphism $\varphi$ is
      called the {\bf log enhancement of $\varphi$ associated to the divisors $E$ and $D$}.
   \end{definition}

   Note that $\varphi^{\dagger}$ is well defined because the hypothesis that $|\varphi|^{-1}(D) \subset E$
   ensures that the pullbak $\varphi^*(f)$ of a section $f \in \cO^{\star}_W(-D)(U)$ belongs to
   $\cO^{\star}_V(-E)(|\varphi|^{-1}(U))$. Here,  $U \subset |W|$ denotes an open subset.

Next, we define special morphisms of log spaces, namely, those that can be obtained by
taking pullbacks of log structures in the sense of \autoref{def:pb} (see~\cite[Section III.1.2]{O 18}). They play an important role in the analysis of the structure of {\em roundings},
which are the logarithmic version of
real oriented blowups, as we will see in~\autoref{thm:torsorfibre} \eqref{cartdiag} below.

     \begin{definition}  \label{def:strict}
         A morphism  $\phi \colon V \to W$ of log spaces is called
         {\bf strict} if it establishes an isomorphism $\phi^* \cM_W \simeq \cM_V$.
     \end{definition}

    \begin{remark}   \label{rem:meanstrict}
       Assume that $\phi : V \to W$ is a morphism of ringed spaces and that $\alpha_W : \cM_W \to \cO_W$
       is a log structure on $W$. Then, by the construction of the log structure associated to a prelog
       structure (see \autoref{def:standardLogStructure}), there is a natural morphism
       $\phi^{\flat} : |\phi|^{-1}(\cM_W) \to \phi^{\star}(\cM_W)$ of sheaves of monoids
       making the  following diagram commute:
      \[
           \xymatrix{
              |\phi|^{-1} (\cM_W)  \ar[r]^-{\phi^{\flat}} \ar[d]_{|\phi|^{-1}\alpha_W}
                           & \phi^{\star}(\cM_W) \ar[d]^{|\phi|^{-1} \alpha_W} \\
               |\phi|^{-1} (\cO_W)   \ar[r]_{\phi^{\sharp}}
                          & \cO_V.}
      \]
       The triple $(|\phi|, \phi^{\sharp}, \phi^{\flat})$  defines a morphism
       $\phi^{\dagger} \colon V^{\dagger} \to W^{\dagger}$ between the log spaces
       $V^{\dagger} := (V, \phi^{\star}(\cM_W))$ and $W^{\dagger} := (V, \cM_W)$
       associated to the pullback operation.
       To say that a morphism is {\em strict} means that it is isomorphic in the category of
       log spaces to a morphism of this type.
    \end{remark}

  \medskip
\subsection{Charts of log structures}   \label{ssec:chartsls}  $\:$
\medskip

In this subsection we explain how the natural divisorial log structure on an affine toric variety
of \autoref{def:logtoricgen} may
be defined in a second way. Namely, as the log structure associated to a particular prelog structure
(see \autoref{prop:identiflog}). This
will lead us to the notions of {\em charts} and {\em atlas} of a log space, which are similar to the
homonymous notions from differential geometry (see Definitions \ref{def:chart} and
\ref{def:finelogspace}). Using the notion of atlas, we will define {\em coherent} and
{\em fine log spaces} (see again \autoref{def:finelogspace}).
\medskip

  In \autoref{def:logtoricgen}, we saw that toroidal varieties are endowed with canonical divisorial
  log structures, which we called {\em toroidal}. Now we will see that the toroidal log structure
  of an {\em affine} toric variety may be alternatively described as the log structure
  associated to a prelog structure, in the sense of \autoref{def:standardLogStructure}.

  Let $(\Gamma, +)$ be a toric monoid in the sense of \autoref{def:toricmonoid}.
  Consider the corresponding complex affine
     toric variety $\tv^\Gamma$, in the sense of \autoref{def:afftoric}. The map which
     associates to each element $m$ of $\Gamma$ the monomial $\chi^m$ (see \autoref{def:groupalg})
     is a morphism of monoids:
        \begin{equation} \label{eq:morphmontoric}
            (\Gamma, +) \to (\CC[\Gamma], \cdot).
        \end{equation}
    This map may be {\em sheafified} as follows:
       \begin{itemize}
          \item $\Gamma$ is the monoid of global
              sections of the {\em constant sheaf of monoids $\boxed{\underline{\Gamma}_{\tv^{\Gamma}}}$
              on $\tv^{\Gamma}$ with values in $\Gamma$} (see \cite[Section 3]{T 22}).
              This sheaf is associated to the
              presheaf of monoids whose monoid of sections on an open set $U$ of $\tv^{\Gamma}$ is
              $\Gamma$ and whose restriction morphisms are identities. Therefore, the
              sections of the sheaf $\underline{\Gamma}_{\tv^{\Gamma}}$ on $U$
              are the locally constant functions from $U$ to $\Gamma$. In particular, the monoid
              of sections of the sheaf $\underline{\Gamma}_{\tv^{\Gamma}}$
              on $U$ is again $\Gamma$ if and only if $U$ is connected.
          \item $\CC[\Gamma]$ is the $\CC$-algebra generated by the image of $\Gamma$ inside the ring
            $\cO_{\tv^{\Gamma}}(\tv^{\Gamma})$ of global sections of the sheaf $\cO_{\tv^{\Gamma}}$
            of holomorphic functions on the toric variety $\tv^{\Gamma}$.
          \item the map \eqref{eq:morphmontoric} is induced at the level of global sections by
            the morphism of sheaves of monoids
                  \begin{equation} \label{eq:sheafmorphmontoric}
                       (\underline{\Gamma}_{\tv^{\Gamma}}, +) \to (\cO_{\tv^{\Gamma}}, \cdot).
                  \end{equation}
       \end{itemize}

       According to \autoref{def:prelog}, the morphism \eqref{eq:sheafmorphmontoric} defines a
       prelog structure on the toric variety $\tv^{\Gamma}$. It turns out that its associated log structure
       in the sense of \autoref{def:standardLogStructure}
       is its toroidal  log structure of \autoref{def:logtoricgen}
       (see \cite[Example 3.19]{G 11} or \cite[Corollary III.1.9.5]{O 18}):

       \begin{proposition}   \label{prop:identiflog}
           The log structure on the affine toric variety $\tv^{\Gamma}$ associated to the
           prelog structure \eqref{eq:sheafmorphmontoric} may be canonically identified with the
           divisorial log structure determined by the toric boundary $\partial \tv^{\Gamma}$.
       \end{proposition}

       The maps \eqref{eq:sheafmorphmontoric} are particular cases of {\em charts}
       of log varieties in the following sense
       (see \cite[Example (1.5)]{K 88} and \cite[Proposition III.1.2.4]{O 18}):

      \begin{definition}  \label{def:chart}
         Let $(W, \cM_W,\alpha_W)$ be a log variety and $U$ be an open subset of $W$.
         Let $\Gamma$ be a monoid.
         A {\bf chart of $W$ on $U$ subordinate to the monoid $\Gamma$} is the datum of:
            \begin{itemize}
               \item a prelog structure  of the form $\underline{\Gamma}_U \to \cO_{W| U}$ on $U$;
               \item an identification of the log structure associated to this prelog structure with the restriction
                    of $(\cM_W,\alpha_W)$ to $U$.
            \end{itemize}
      \end{definition}

      \begin{remark}  \label{rem:similardiffg}
          Similarly to what we explained about the passage from the map
          \eqref{eq:morphmontoric} to the map \eqref{eq:sheafmorphmontoric}, giving
          a prelog structure $\underline{\Gamma}_U \to \cO_{W| U}$ is equivalent to giving
          a morphism of monoids $\Gamma \to \cO_{W| U}(U)$. This allows to see that
         the terminology of \autoref{def:chart} is similar to the analogous one
         used in differential geometry, where a chart of a smooth manifold $M$ is usually
         defined as a diffeomorphism $U \to V$ from an open subset $U$ of $M$ to an open subset
         $V$ of $\R^n$.  Such a diffeomorphism is given by a particular smooth map from
         $U$ to $\R^n$, which in turn is equivalent to the datum of $n$ smooth
         functions $(f_j : U \to \R)_{1 \leq j \leq n}$. Finally, to give such functions
         amounts to give a morphism of monoids
            \[  (\N^n, +) \to (\mathcal{C}^{\infty}(U), \cdot),   \]
            which is analogous to the morphism $\Gamma \to \cO_{W| U}(U)$ from the
            beginning of this remark.
      \end{remark}

      The notion of chart of a log space allows one to define {\em atlases} and special types of
      log spaces, satisfying convenient local properties:

  \begin{definition}   \label{def:finelogspace}
         An {\bf atlas} of a log space is a collection of charts on open sets which cover the
         underlying topological space. A log space is called {\bf coherent}
         if it admits an atlas by charts subordinate to finitely generated monoids.
         A log space is called {\bf fine} if it admits an atlas by charts subordinate to finitely
         generated and integral monoids (see \autoref{def:intmon}).
  \end{definition}

  \begin{remark}  \label{rem:meancoh}
      The terminology {\em coherent log space} reminds us of the notion of {\em coherent
      sheaf of modules} on a ringed space (see \cite[Section I.2.1]{EH 00}).
      Similarly to that one, it is a finiteness condition
      which may be formulated using finitely generated algebraic objects.
  \end{remark}

  Every fine log space is obviously coherent.
  The notion of coherent log space will be used in \autoref{prop:opentrivlocus} and that of
  fine log space  in \autoref{thm:torsorfibre}.  An important example of
  fine log spaces is provided by:

  \begin{proposition}   \label{prop:toroidfine}
       Toroidal log varieties are fine.
  \end{proposition}

  \begin{proof}
      Let $W$ be a toroidal variety. It admits by definition an atlas of toric charts,
      in the sense of \autoref{def:toroidal}. Let $\phi_j : U_j \to V_j$ be a chart of this
      atlas. Therefore, $V_j$ is an open set of an affine toric variety $\tv^{\Gamma_j}$, where
      $\Gamma_j$ is a toric monoid. Consider the canonical prelog structure
      $\underline{\Gamma}_j  \to \cO_{V_j}$ on $V_j$ obtained by restriction of
      the prelog structure \eqref{eq:sheafmorphmontoric}. Its pullback on $U_j$ by the
      morphism $\phi_j$ is a chart of the canonical toroidal log structure on $W$
      in the sense of \autoref{def:logtoricgen}. This chart is
      subordinate to the monoid $\Gamma_j$ in the sense of \autoref{def:chart}.
      Therefore, the atlas $(\phi_j)_j$ of the toroidal variety $W$ gives rise canonically
      to an atlas of the associated toroidal log variety. Since every toric monoid is
      finitely generated and integral, we conclude that toroidal log varieties are fine.
  \end{proof}

  \begin{remark}
      Toroidal log varieties are also {\em log smooth} in the sense of Kato \cite[Paragraph (3.3)]{K 88}.
      This is a consequence of Kato's \cite[Theorem 3.5]{K 88}
      (see also \cite[Theorem 4.1]{K 96}, \cite[Theorem 3.11]{ACGHOSS 13}
      and \cite[Section IV.3]{O 18}). There is a related
      notion of {\em log regularity} (see \cite[Section III.1.11]{O 18} and \cite[Section 3.6]{BN 20}),
      which coincides with log smoothness for complex log spaces.
  \end{remark}

  Recall that the notion of trivial log structure was introduced in \autoref{def:tauttriv}. Every
  log space has a {\em triviality locus} (see \cite[Corollary II.2.1.6, Proposition III.1.2.8]{O 18}):

  \begin{definition}  \label{def:trivlocus}
     Let $W$ be a log space. Its {\bf triviality locus} is the subset $| W^{\mathrm{triv}} |$ of
     $| W|$ consisting of the points $x \in W$ such that $\cM_{W, x} = \cM_{W, x}^{\star}$.
     The {\bf triviality log subspace} $\boxed{W^{\mathrm{triv}}}$ of $W$ is the topological space
     $| W^{\mathrm{triv}} |$ endowed with the restriction of the log structure of $W$.
  \end{definition}

  Coherent log spaces have the following property, whose proof is left to the reader:

  \begin{proposition}  \label{prop:opentrivlocus}
      Let $W$ be a coherent log space in the sense of \autoref{def:finelogspace}. Then, its triviality
      locus is an open subset of $| W|$.
  \end{proposition}

\medskip
\subsection{Kato and Nakayama's rounding operation}   \label{ssec:KNrounding}  $\:$
\medskip

 In this subsection we define the {\em rounding map} of a complex log space (see \autoref{def:rounding})
 and we discuss its topological and functorial properties (see \autoref{thm:torsorfibre}).  The
 rounding map generalizes the real oriented blowup of toroidal varieties of \autoref{def:toroidrob}
 (see \autoref{prop:topmantoroid}).
 The rounding map is the most general version of a real oriented blowup discussed in this text.
    \medskip

      Recall that the {\em sign function} was introduced in  \autoref{def:logCoordinatesPolar}
      as $\mbox{sign}(z) := z/|z|$ for each $z\in \CC^*$.

    The following definition of the {\em rounding} of a log space is a slight reformulation
    of Kato and Nakayama's generalization of the
    real oriented blowup operation
    given in~\cite[Section 1]{KN 99} for log complex analytic spaces
    (see also~\cite[Definition V.1.2.4]{O 18}).     Alternative descriptions of this operation
    can be found in~\cite[Section 1.2]{IKN 05} and~\cite[Section 1.1]{A 21}.

    \begin{definition} \label{def:rounding}
        Let $W$ be a complex log space in the sense of~\autoref{def:logspace}.
        We identify  the sheaves  $\cM_W^{\star}$ and $\cO_W^{\star}$ via the evaluation morphism
        $\alpha_W$. The \textbf{rounding  of $W$} is the set
                         \[ \boxed{W^{\odot}}:= \left\{ (x,u), x \in |W|,   u \in \Hom(\cM_{W,x}, \bS^1)  ,
                               u(f) = \mbox{sign}(f(x)), \: \forall \: f \in
                                   \cM_{W,x}^{\star}  = \cO_{W,x}^{\star}
                               \right\}. \]
       The \textbf{rounding map} is the function
                        \[ \begin{array}{cccc}
                              \boxed{\tau_W} \colon &  W^{\odot} & \to & |W|  \\
                                & (x,u) & \mapsto & x.
                           \end{array}   \]
     The set $W^{\odot}$ is endowed with the weakest topology making continuous
     the rounding map $\tau_W$ and the elements of the set of functions
            \[ \{\mbox{sign}(m)\in \Hom(\tau_W^{-1}(U), \bS^1):
         U\subset |W| \text{ is open},  \ m\in \cM_W(U)\}, \]
     where:
  \[  \begin{array}{cccc}
                   \boxed{\mbox{sign}(m)}\ \colon &   \tau_W^{-1}(U) & \to & \bS^1  \\
                                & (x,u) & \mapsto & u(m_x).
       \end{array}
   \]
    \end{definition}

     \begin{remark}  \label{rem:terminbettirealiz}
          The terminology ``\emph{rounding}'' was introduced in \cite{NO 10} and
          used also in \cite{ACGHOSS 13}. It refers to the fact that whenever $W$
        is a  \emph{fine} log space in the sense  of~\autoref{def:chart}, the fibers of the rounding
        map $\tau_W$ are finite disjoint unions of compact tori, which are product
        of circles, and thus, prototypical ``round'' geometric objects (see~\autoref{thm:torsorfibre}).
        Alternative names are ``\emph{Kato-Nakayama space}''
        (see \cite{C 04},  \cite{TV 18}, \cite{A 21}) or ``\emph{Betti realization}'',
        a terminology due to Ogus (see~\cite{O 18}).
        We preferred to use the most geometric of the three terminologies. The notation ``$W^{\odot}$''
         is new. We chose it in order to allude to the passage to polar coordinates, which
         is the prototypical rounding, that of the complex plane $\CC$ endowed with
         the divisorial log structure induced by its origin: the symbol ``$\odot$'' is intended to
         recall the fact that in this way the origin of $\CC$ is replaced by a circle.
         In the literature one may also find the
         notations $W^{\log}$ (see \cite{KN 99}, \cite{KN 08}, \cite{ACGHOSS 13}), $W_{\log}$
         (see \cite{NO 10}, \cite{O 18}, \cite{AO 20}), $W^{KN}$
         (see \cite{A 20}, \cite{A 21}).
     \end{remark}

     Recall that a {\em cartesian diagram} of topological spaces denotes a pullback
     or fiber product diagram in  the topological category.
    The next result describes the topology of the rounding of a complex log space,
    as well as functoriality properties of the rounding operation. For a proof when
    $W$ is a complex analytic space endowed with a log structure, we refer to
    \cite[Proposition V.1.2.5]{O 18}. The same proof is valid for arbitrary complex log spaces.

    \begin{theorem}  \label{thm:torsorfibre}
         Assume that $W$ is a complex log space.
             \begin{enumerate}
                   \item The rounding map $\tau_W$ is a
                       homeomorphism whenever the log structure of $W$ is trivial.

                   \item   \label{fibrdescr}
                       Let $x$ be a point of the underlying topological space $|W|$.
                       Consider the abelian group:
                     \begin{equation}\label{eq:Tx}
                       \boxed{T_x}:= \Hom({\cM}_{W,x}/{\cM}_{W,x}^{\star}, \bS^1),
                     \end{equation}
                     where morphisms are considered in the category of monoids.
                     Then, $T_x$ acts naturally on the fiber $\tau_W^{-1}(x)$ by extending the natural
                     action on $\Hom(\cM_{W,x}, \bS^1)$, i.e.:
                     \[ (\beta \cdot u)(m) = \beta(\overline{m}) u(m) \quad \text{ for all }
                             \beta\in T_x, \;u\in \Hom(\cM_{W,x}, \bS^1), \;m\in {\cM}_{W,x},
                     \]
                     where $\overline{m}$ denotes the coset of $m$ in $\cM_{W,x}/\cM_{W,x}^{\star}$.
                     This action
                         defines a torsor if the monoid $\cM_{W,x}$ is unit-integral in the sense
                         of~\autoref{def:intmon}.
                         In particular, $\tau_W$ is surjective if $\cM_W$ has only unit-integral stalks.
                         This occurs, for instance, if $W$ is a fine log space in the sense of
                         \autoref{def:finelogspace}.

                   \item \label{functbr}
                       The construction of the rounding $W^{\odot}$ is functorial and the morphism
                      $\tau_W$ is natural. More precisely, a morphism $\phi \colon V \to W$
                      of complex log spaces induces a morphism of topological spaces
                      $\boxed{\phi^{\odot}} \colon V^{\odot} \to W^{\odot}$,
                      called the \textbf{rounding of $\phi$}, which fits in a commutative diagram:
                          \begin{equation}\label{eq:roundingOfF}
                               \xymatrix{
                                    V^{\odot}     \ar[r]^{\phi^{\odot}}   \ar[d]_{\tau_V} &   W^{\odot} \ar[d]^{\tau_W} \\
                                     |V|  \ar[r]_{|\phi|}                     &   |W|.}
                          \end{equation}
                          Thus, the rounding operation is a \emph{covariant functor} from the category
                          of complex log spaces to the category of topological spaces.

                     \item   \label{cartdiag}
                         The diagram~\eqref{eq:roundingOfF} is cartesian in the topological category
                         whenever the log morphism $\phi$ is strict in the sense of~\autoref{def:strict}.
             \end{enumerate}
    \end{theorem}

 \begin{remark}   \label{rem:ConnectedFibers}
   Whenever $W$  is a fine log space in the sense  of~\autoref{def:chart}, the  monoid
   ${\cM}_{W,x}/{\cM}_{W,x}^{\star}$ appearing in~\autoref{thm:torsorfibre}  (\ref{fibrdescr}) is fine.
   Consequently, the group $({\cM}_{W,x}/{\cM}_{W,x}^{\star})^{\gp}$
    generated by it in the sense of \autoref{def:intmon}
    is finitely generated. It is thus a direct sum of a finite abelian group and of a lattice.
    Therefore,  the group $T_x$ from~\eqref{eq:Tx}
    is a finite disjoint union of compact tori (that is, of groups isomorphic to $(\bS^1)^n$ for some $n \in \N$).
    As a consequence of~\autoref{thm:torsorfibre}, the fiber $\tau_W^{-1}(x)$
    is connected (that is, it is a single torus) if, and only if,
    the group $({\cM}_{W,x}/{\cM}_{W,x}^{\star})^{\gp}$ is a lattice. This is always the case
    when $(W,\cM_W)$ is a toroidal log space
    in the sense  of~\autoref{def:chart} (see \cite[Proposition II.2.3.7]{O 18}).
  \end{remark}

 \begin{remark}  \label{rem:ghostsheaf}
     The quotient sheaf of monoids ${\cM}_{W}/{\cM}_{W}^{\star}$ whose stalks
     ${\cM}_{W,x}/{\cM}_{W,x}^{\star}$ appear in \autoref{thm:torsorfibre} is called
     the {\em ghost sheaf} of the log structure by Gross \cite[Page 101]{G 11},
     and denoted there $\overline{\cM}_W$.
     The notation $\overline{\cM}_W$ is also used in Ogus' textbook \cite{O 18}.
 \end{remark}

\autoref{thm:torsorfibre}~\eqref{cartdiag}  has an important consequence:

    \begin{corollary}    \label{cor:restrround}
       Let $W$ be a complex log space and let $|V| \hookrightarrow  |W|$
       be a subspace of the underlying topological space.
        Endow $|V|$ with the log structure obtained by restricting the log structure of $W$.
        Then, the associated rounding map $\tau_V$ is the restriction to $|V|$ of
        the rounding map $\tau_W$ of $W$.
    \end{corollary}

   The roundings of the natural toroidal log structures on complex affine toric varieties
    in the sense of~\autoref{def:logtoricgen} are the same as their real oriented blowups
    in the sense of \autoref{def:toroidrob} (see \cite[Proposition A.1]{KN 08} and \cite[Proposition 2.4]{A 21}):

     \begin{proposition}   \label{prop:robisround}
         Let $X$ be a complex  affine  toric variety. Consider its real oriented blowup
         morphism $\tau_{X} : X^{\rob}  \to  X$  in the sense of   \autoref{def:robuafftoric}.
         Let $X^{\dagger}$ be the complex log variety
         obtained by endowing $X$ with its canonical toroidal structure in the sense of
         \autoref{def:logtoricgen}. Consider its rounding map
         $\tau_{X^{\dagger}} : (X^{\dagger})^{\odot}  \to  X$ in the sense of \autoref{def:rounding}.
         Then there exists a canonical homeomorphism
         $\boxed{I_X} : X^{\rob} \to (X^{\dagger})^{\odot}$  making the following diagram commutative:
           \[
           \xymatrix{
              X^{\rob}  \ar[r]^-{I_X} \ar[d]_{\tau_X}
                           & (X^{\dagger})^{\odot} \ar[d]^{\tau_{X^{\dagger}}} \\
               X   \ar[r]_{\mathrm{id}}
                          & X.}
         \]
     \end{proposition}

    Using toric charts, \autoref{prop:robisround} may be used to prove that
    the roundings of toroidal log varieties
    in the sense of~\autoref{def:logtoricgen} are the same as their real oriented blowups
    in the sense of \autoref{def:toroidrob}:

    \begin{proposition}  \label{prop:topmantoroid}
        Let $W$ be a toroidal log variety in the sense of \autoref{def:logtoricgen}.
        Then its rounding map is
        isomorphic to the real oriented blowup of the underlying toroidal variety.
        The topological space $W^{\odot}$ has a natural structure of
        real semi-analytic variety homeomorphic to a topological manifold-with-boundary.
        Its topological boundary $\partial_{top} (W^{\odot})$ is the preimage
        of the toroidal boundary $\partial W$ of $W$ under the rounding map $\tau_W$.
    \end{proposition}

    Furthermore, it can be shown that $W^{\odot}$ is a ``manifold with generalized corners''
    in the sense of Joyce~\cite{J 16} (see also~\cite{GM 15, KM 15, FJ 24}).  According to
    A'Campo (private communication), classical {\em manifolds with corners} modeled
    on open subsets of $(\R_{\geq 0})^n$ were first introduced by Cerf \cite[Section I.1]{C 61}.

    \autoref{prop:topmantoroid} can be proven locally since open sets of affine toric varieties serve
    as local models for toroidal varieties.  The topological part of the statement can be found
    in~\cite[Lemma 1.2]{KN 08}. Its extension to the semi-analytic category
    may be proved similarly to the proof of
    \cite[Proposition 4.3]{CHH 25}.  \autoref{thm:logehresm} in the next subsection complements
    this result by extending it to morphisms.

 \medskip
\subsection{A reinterpretation of rounding using log points}   \label{ssec:reintrounding}  $\:$
\medskip

In this subsection we introduce the {\em functor of log points} of a log space
(see \autoref{def:functpoints}) and we show that the rounding of a complex log space
is the set of its log points with values in the polar log point of
\autoref{def:varlogpoints} (see \autoref{prop:threetypespointsets}). This viewpoint
on the rounding operation allows to construct variants of it by replacing
the polar log point by other types of complex log points.
\medskip

  Recall that the notion of {\em complex log point} was introduced in \autoref{def:logpoint} and
  that several types of such points were listed in \autoref{def:varlogpoints}. The complex log points
  form a  full subcategory of the category of complex log spaces.

  Grothendieck defined a functor of points associated to each scheme $X$ (see Eisenbud and Harris
  \cite[Chapters I.4 and  VI.1]{EH 00}). It is the contravariant functor from the category of schemes
  to the category of sets, which at the level of objects transforms each scheme
  $\Pi$ into the set $\Hom(\Pi, X)$ of morphisms of schemes from $\Pi$ to $X$, called the
  {\em set of $\Pi$-valued points of $X$}, and each morphism of schemes $\psi : \Pi_1 \to \Pi_2$
  into the function
      \[ \begin{array}{cccc}
                 W(\psi) : & W(\Pi_2)  & \to & W(\Pi_1) \\
                  &   \nu  & \mapsto &  \nu \circ \psi.
            \end{array} \]

  There is an analogous notion for log spaces:

\begin{definition}   \label{def:functpoints}
   Let $W$ be a log space. The {\bf functor $\boxed{W(\cdot)}$ of log points}
   is the functor from the category of log spaces to that of sets defined by:
       \[ \boxed{W(\Pi)} := \Hom(\Pi, W)  \]
   at the level of objects, where morphisms are taken in the log category, and by:
       \[ \begin{array}{cccc}
                 \boxed{W(\psi)} : & W(\Pi_2)  & \to & W(\Pi_1) \\
                  &   \nu  & \mapsto &  \nu \circ \psi
            \end{array} \]
    at the level of morphisms. Here, $\psi : \Pi_1 \to \Pi_2$ is a morphism
    of log spaces.  $\boxed{W(\Pi)}$ is called the
    {\bf set of $\Pi$-points of $W$}.
\end{definition}

Recall that the triviality log subspace $W^{\mathrm{triv}}$ of $W$ was introduced in \autoref{def:trivlocus}.  It turns out that the underlying topological space $|W|$ of $W$, its triviality log subspace $W^{\mathrm{triv}}$ and its rounding $W_{\mathrm{log}}$ may all be seen as sets of $\Pi$-points
of $W$, for special types of complex log points in the sense of \autoref{def:logpoint}. Here is the
precise statement:

\begin{proposition}   \label{prop:threetypespointsets}
     Let $W$ be a  complex log space. One has the following identities:
        \[   \left\{
                  \begin{array}{ccl}
                      W(\Pi^{\mathrm{triv}}) & = &  |W^{\mathrm{triv}}|, \\
                      W(\Pi^{\mathrm{pol}}) & = &  W^{\odot},  \\
                      W(\Pi^{\mathrm{\emptyset}}) & = &  |W|.
                  \end{array}
              \right.
        \]
      Moreover, the natural morphisms
         \[ \Pi^{\mathrm{triv}}  \leftarrow  \Pi^{\mathrm{pol}} \leftarrow \Pi^{\mathrm{\emptyset}}\]
      of complex log points induce the morphisms
         \[ |W^{\mathrm{triv}}| \rightarrow   W^{\odot} \rightarrow |W| \]
         in which the left arrow represents the natural inclusion of the triviality locus
         $|W^{\mathrm{triv}}|$ in the rounding of $W$ and the right arrow
         is the rounding map $\tau_{W}$.
\end{proposition}

  We leave the proof of \autoref{prop:threetypespointsets} as an excellent exercise allowing
  the reader to understand the notion of morphism of log spaces and the definitions of the
  topological spaces $ |W^{\mathrm{triv}}|$ and $W^{\odot}$.

\begin{remark}  \label{rem:functspacevp}
    The identification $W^{\odot} = W(\Pi^{\mathrm{pol}})$ was noticed already by Kato and Nakayama
    in the paper in which they introduced the notion of rounding (see \cite[Page 165]{KN 99}).
    A proof of this identification may be found in \cite[Remark (1.1)]{A 21}.
\end{remark}

\begin{remark}   \label{rem:analogsPi}
    This point of view on the sets $W^{\mathrm{triv}}$, $W^{\odot}$ and $|W|$ allows to
    construct variants of them by working with other complex log points. For instance, in
     \cite[Definition 4.1]{C 16}, Cauwbergs introduced (with a different terminology)
    an {\em extended rounding} of $W$ as $W(\Pi)$, where $\Pi$ is defined by the morphism of
    monoids:
        \[   \begin{array}{ccc}
                                       ( \R_{\geq 0} \times \R_{\geq 0}  \times \bS^1, \cdot) & \to  &  (\CC, \cdot)    \\
                                         (q, r, u)        & \mapsto  & r \cdot u.
                 \end{array}    \]
     This allowed him to obtain models of Milnor fibrations endowed with canonical geometric
     monodromy transformations, similar to A'Campo's models from \cite{A 75}
     (see \cite[Definition 5.1]{C 16}). An analogous
     construction was performed recently by Campesato, Fichou and Parusinski \cite[Section 2.1]{CFP 21}
     under the name of {\em complete log space} associated to a complex log space.
\end{remark}

  \medskip
  \subsection{Special cases of Nakayama and Ogus' local triviality theorem}
  \label{ssec:specaseNO}  $\:$
  \medskip

  In this subsection we discuss Nakayama and Ogus' 2010 local triviality theorem for the roundings of
  special types of log morphisms
  (see \autoref{thm:logehresm}) and two of its consequences
  for global and, respectively, local complex one-parameter families of analytic varieties
  (see Corollaries \ref{cor:quasitorrelcoh} and \ref{cor:milntubelog}).
  The local corollary may be used in the study of {\em Milnor fibers} of {\em smoothings of
  singularities}.
  \medskip

       Using Siebenmann's topological local triviality theorem from~\cite[Corollary 6.14]{S 72},
       Nakayama and Ogus proved  the following log version of
       Ehresmann's theorem (see~\cite[Theorems 3.5 and 5.1]{NO 10}):

    \begin{theorem}   \label{thm:logehresm}
         Let $\phi \colon V \to W$ be a morphism of log complex analytic spaces, where $W$ is fine and
         $V$ is relatively coherent. Assume that $f$ is proper, separated, exact
         and relatively log smooth. Then,
         its rounding $\phi^{\odot} \colon V^{\odot} \to W^{\odot}$ is a locally trivial fibration whose fibers
         are oriented topological manifolds-with-boundary. The union of the boundaries of the fibers
         of $\phi^{\odot}$ consists of those points of $V^{\odot}$ sent by the rounding map
         $\tau_V\colon V^{\odot}\to |V|$ to points of $|V|$ where $\phi$ is not vertical.
    \end{theorem}

    We will not give the definitions of the terms involved in the statement which we have not discussed until now
       (\emph{relative coherence, separated, exact and relatively log smooth
       morphisms, points where a morphism is vertical}),
       since our interest in this result lies in two of its consequences,
       namely, Corollaries \ref{cor:quasitorrelcoh}  and \ref{cor:milntubelog} discussed below.

     \begin{remark}\label{rem:logehresm}
         \autoref{thm:logehresm} generalizes Kawamata's~\cite[Theorem 2.4]{K 02}, which
         concerned his notion of real oriented blowups of quasi-smooth toroidal
         varieties (see \autoref{rem:Kawrob}).
    \end{remark}

 The next corollary to~\autoref{thm:logehresm} can be proved by
       translating the notions of \emph{relative coherence},
       {\em separatedness, exactness}, \emph{log smoothness},
       \emph{relative log smoothness} and \emph{verticality}
       into the toroidal language, when the target is a {\bf divisorial log disk}
       $\boxed{\D^{\dagger}_{0_{\CC}}} := (\D, \cO_{\D}^{\star}(- 0_{\CC}))$:

       \begin{corollary}  \label{cor:quasitorrelcoh}
             Let $ \varphi \colon V \to \D$ be a proper complex analytic morphism from
             a complex manifold $V$ to an open disk $\D$ of $\CC$ centered at the origin
             $0_{\CC}$ of $\CC$.
             Assume that the reduced fiber $Z(\varphi) := \varphi^{-1}(0_{\CC})$ is a divisor with simple
             normal crossings in $V$.  Consider the  log enhancement
         $\varphi^{\dagger} \colon V^{\dagger}_{Z(\varphi)} \to \D^{\dagger}_{0_{\CC}}$
         of $\varphi$ associated to
         the divisors $Z(\varphi)$ and $0_{\CC}$,  in the sense of~\autoref{def:logenhancement}.
          Then, its rounding
              \[  (\varphi^{\dagger})^{\odot} \colon (V^{\dagger}_{Z(\varphi)})^{\odot}
                    \to (\D^{\dagger}_{0_{\CC}})^{\odot}  \]
          is a locally trivial topological fibration.
       \end{corollary}

       \begin{remark} \label{rem:recentuses}
           In \cite{AO 20}, Achinger and Ogus studied the monodromy of such families
          $ \varphi \colon V \to \D$ using \autoref{cor:quasitorrelcoh}. In \cite{U 01},
          using an analog of \autoref{cor:quasitorrelcoh}, Usui had already studied
           the monodromies of such families in the case of
          {\em semi-stable degenerations}, that is, morphisms $\varphi$ as above with reduced
          scheme-theoretic central fiber. A theory of {\em real log schemes}
          was developed by Arg\"uz \cite{A 21} in order to study real one-parameter degenerations
          of real algebraic varieties. This theory was used by Ambrosi and Manzaroli \cite{AM 22}
          to evaluate Betti numbers of real algebraic varieties and by Rau \cite{R 23} to study totally
          real semi-stable degenerations.
      \end{remark}

      A version of \autoref{cor:quasitorrelcoh} can be used as a tool to study a germ
       $f : (X, x) \to (\CC, 0)$ of holomorphic function defined on a germ
             of irreducible complex variety,  such that $f^{-1}(0_{\CC})$ contains the singular locus of $(X,x)$.
             In this case, the other fibers of $f$ are {\em smooth} in the neighborhood of $x$.
             For this reason, one
             says that $f$ is a {\bf smoothing} of the singularity  $(Z(f), x) := (f^{-1}(0_{\CC}),x)$.
             Such a smoothing $f$ admits an associated {\bf Milnor fibration} or {\bf Milnor-L\^e fibration}
             over the circle (see, in chronological order, \cite{M 68}, \cite{L 77}, \cite[Theorem 5.1]{S 19}
             and \cite[Theorem 6.7.1]{CS 21}).  Choose
             a resolution $\pi : (\tilde{X}, E) \to (X,x)$ such that the reduced fiber
             $Z(f \circ \pi) := (f \circ \pi)^{-1}(0_{\CC}) \hookrightarrow \tilde{X}$ of the composition $f \circ \pi$
            is a divisor with simple normal crossings.
             We aim to apply \autoref{cor:quasitorrelcoh} to
            $V := \tilde{X}$ and $\varphi := f \circ \pi$.

       In order to achieve the \emph{properness} of $f \circ \pi$, we work with a
       \textbf{Milnor tube representative} of $f$, which we still denote by $f$ for simplicity.
       Such a representative is obtained by first considering the part of a representative
       $X \hookrightarrow \CC^n$
       of $(X,x)$ contained in a Milnor ball, and then restricting this set further to the preimage
       by $f$ of a sufficiently small Euclidean disk $\D$ centered at the origin of $\CC$.
       There is a slight difference between this setting and that of~\autoref{cor:quasitorrelcoh},
       as  $V := \tilde{X}$ has a topological boundary. However, since $f \circ \pi$
       is locally trivial near that boundary,
       it is straightforward to show that a version of~\autoref{cor:quasitorrelcoh}
       applies to this slightly different context:

        \begin{corollary}  \label{cor:milntubelog}
             Let $f \colon X \to \D$ be a Milnor tube representative of a smoothing
             $f : (X, x) \to (\CC, 0_{\CC})$ such that $f^{-1}(0_{\CC})$ contains the singular locus of $(X,x)$.
             Let $\pi : (\tilde{X}, E) \to (X,x)$ be a resolution of $(X,x)$ such that the reduced fiber
             $Z(f \circ \pi) \hookrightarrow \tilde{X}$ of $f \circ \pi : \tilde{X} \to \D$ over $0_{\CC}$
             is a divisor with simple normal crossings. Consider the  log enhancement
             $(f \circ \pi)^{\dagger} \colon \tilde{X}^{\dagger}_{Z(f \circ \pi)} \to \D^{\dagger}_{0_{\CC}}$
             of $f \circ \pi$ associated to
             the divisors $Z(f \circ \pi)$ and $0_{\CC}$,  in the sense of~\autoref{def:logenhancement}.
          Then, its rounding
              \[ ((f \circ \pi)^{\dagger})^{\odot} \colon (\tilde{X}^{\dagger}_{Z(f \circ \pi)})^{\odot}
                    \to (\D^{\dagger}_{0_{\CC}})^{\odot}  \]
          is a locally trivial topological fibration.
             Moreover, the restriction of this fibration to the boundary
          circle \linebreak  $(0_{\CC}, \cO_{\D | 0}^{\star}(- 0_{\CC}))^{\odot} $ of the cylinder
          $(\D^{\dagger}_{0_{\CC}})^{\odot} = (\D, \cO_{\D}^{\star}(- 0_{\CC}))^{\odot}$
          is isomorphic to the Milnor fibration of $f$ over the circle.
       \end{corollary}

     \begin{remark}   \label{rem:moregentoroid}
         The analog of \autoref{cor:milntubelog} is still true if $(\tilde{X}, Z(f \circ \pi))$
         is only assumed to be a toroidal variety. In the joint work \cite{CPPS 23}
         with Cueto and Stepanov, the author used
         an even more general version of \autoref{cor:milntubelog}
         (still resulting from \autoref{thm:logehresm}), to prove Neumann and Wahl's
         {\em Milnor fiber conjecture}. Namely, in \cite[Corollary 4.56]{CPPS 23}, the morphism $\pi$ could be a
         modification such that $Z(f \circ \pi)$ is a special kind of
         subdivisor of a divisor $\partial \tilde{X}$ with the property
         that $(\tilde{X}, \partial \tilde{X})$ is a toroidal variety in the sense of
         \autoref{def:toroidal}, but without $(\tilde{X}, Z(f \circ \pi))$ being toroidal.
         The point is that if any subdivisor of a divisor with
         simple normal crossings in a complex manifold is again a divisor with simple
         normal crossings, the analogous statement for toroidal boundaries is false.
         Variants of \autoref{cor:milntubelog}  were also used in the recent papers
          \cite{CFP 21} of Campesato, Fichou and Parusi\'nski,
          \cite{PS 23} of Portilla Cuadrado and Sigurdsson and \cite{FP 24}
          of Fern\'andez de Bobadilla and Pe\l ka, in which
          various types of questions related to the monodromy of smoothings of singularities were addressed.
     \end{remark}

   \medskip
  \subsection{Applications of rounding to singularity theory}
  \label{ssec:approundsing}  $\:$
  \medskip

  In this final subsection, we explain two ways of using the operation of rounding
  of complex log structures in singularity theory. Namely, we describe the steps
  which have to be followed when one
  wants to use rounding in the study of  {\em links of isolated complex singularities} or of
  {\em Milnor fibrations of smoothings of singularities}. For simplicity, we will consider only {\em smooth}
  total spaces of resolutions. But, as mentioned in \autoref{rem:moregentoroid}, one may extend both
  approaches to toroidal contexts.
  \medskip

  {\bf Using rounding in the study of links of isolated complex singularities.}

  \begin{enumerate}
   \item Start from a representative of the singularity:
      \[ (X,x).  \]

  \item Choose a normal crossings resolution of it:
     \[  \xymatrix{(\tilde{X},  E)   \ar[r]^-{\pi}  &  (X,x).} \]

  \item Consider the divisorial log structure on $\tilde{X}$ induced by $E$ (see \autoref{def:divlogstruct}):
    \[   \tilde{X}_E^{\dagger}.\]

  \item Restrict that log structure to $E$ (see \autoref{def:pb}):
    \[ E^{\dagger}.\]

  \item Round that restriction:
    \[ \xymatrix{ (E^{\dagger})^{\odot}   \ar[r]^-{\tau_{E^{\dagger}}} & E.} \]
 \end{enumerate}

   In this way, we get a representative $(E^{\dagger})^{\odot} $
   of the link of the starting singularity (see \autoref{prop:topmantoroid}).
   It is canonically decomposed into pieces, which are
   the preimages under the map $\tau_{E^{\dagger}}$ of the irreducible components of
   the reduced divisor $E$.

    \medskip

    {\bf Using rounding in the study of Milnor fibrations of smoothings.}

 \begin{enumerate}
   \item Start from a Milnor tube representative of a smoothing (see \autoref{ssec:specaseNO}):
      \[  \xymatrix{ (X,Z(f))   \ar[r]^-f & (\D, 0_{\CC}).} \]

  \item Choose a normal crossings resolution $\pi$ of $f$:
     \[  \xymatrix{(\tilde{X},  Z(f \circ \pi))  \ar[r]^-{\pi}  &   (X, Z(f))   \ar[r]^-f & (\D, 0_{\CC}).} \]

  \item Consider the divisorial log structures on $\tilde{X}$ and $\D$ induced by $Z(f \circ \pi)$
      and $0_{\CC}$ and the corresponding log enhancement of the morphism $f \circ \pi$
      (see \autoref{def:logenhancement}):
    \[    \xymatrix{   \tilde{X}_{Z(f \circ \pi)}^{\dagger}\ar[r]^-{(f \circ \pi)^{\dagger}}  & \D^{\dagger}.} \]

  \item Restrict this log enhancement to $Z(f \circ \pi)$ and $0_{\CC}$:
     \[  \xymatrix{  Z(f \circ \pi)^{\dagger}\ar[r]^-{(f \circ \pi)^{\dagger}}   & 0_{\CC}^{\dagger}}.\]

  \item Round this log morphism (see \autoref{thm:torsorfibre}):
      \[     \xymatrixcolsep{3pc}
               \xymatrix{
                                (Z(f \circ \pi)^{\dagger})^{\odot}
                                  \ar[r]^-{((f \circ \pi)^{\dagger})^{\odot}}
                                   \ar[d]_{\tau_{Z(f \circ \pi)^{\dagger}}}  &
                                      (0_{\CC}^{\dagger})^{\odot} \ar[d]^{\tau_{0_{\CC}^{\dagger}}}    \\
                                         Z(f \circ \pi)   \ar[r]_-{f \circ \pi}         &   0_{\CC}.}
                          \]

 \end{enumerate}
   In this way, we get a representative
      $\xymatrixcolsep{3pc} \xymatrix{   (Z(f \circ \pi)^{\dagger})^{\odot}
        \ar[r]^-{((f \circ \pi)^{\dagger})^{\odot}}  & (0_{\CC}^{\dagger})^{\odot}} $ of the Milnor fibration
        of the starting smoothing $f$ (see \autoref{cor:milntubelog}).
   It is canonically decomposed into pieces, which are the preimages under the map
   $\tau_{Z(f \circ \pi)^{\dagger}}$ of the irreducible components of the reduced divisor $Z(f \circ \pi)$.

     \medskip
     In both cases, rounding may be avoided when one wants to prove that parts of the links
        of two singularities or of the Milnor fibrations of two smoothings are isomorphic.
      Indeed,  if those parts are preimages under the corresponding rounding maps $\tau$
        of unions of irreducible components of divisors of type $E$ or $Z(f \circ \pi)$,  it is enough to
        prove the isomorphism of the restrictions to those unions of the log structures of those divisors.
        This is the approach of \cite{CPPS 23} to prove Neumann and
        Wahl's {\em Milnor fiber conjecture} for splice type singularities (see also \autoref{rem:moregentoroid}
        and \cite{PP 25}).

\bigskip

\noindent
\textbf{\small{Author's address:}}
\smallskip
\

\noindent
\small{P.\ Popescu-Pampu,
  Univ.~Lille, CNRS, UMR 8524 - Laboratoire Paul Painlev{\'e}, F-59000 Lille, France.
  \\
\noindent \emph{Email address:} \url{patrick.popescu-pampu@univ-lille.fr}}
\vspace{2ex}

\end{document}